\documentclass{article}
\usepackage[a4paper, total={6in, 9in}]{geometry}
\usepackage[utf8]{inputenc}
\usepackage[english]{babel}
\usepackage{amsthm}
\usepackage{amsfonts}
\usepackage{tikz-cd}
\usepackage{tikz}
\usepackage{mathrsfs}
\usetikzlibrary{decorations.markings}
\usetikzlibrary{decorations.pathmorphing}
\usetikzlibrary{calc}
\usetikzlibrary{matrix}
\usepackage{subcaption}
\usepackage{comment}
\usepackage{algorithm}
\usepackage{algpseudocode}
\usepackage[most]{tcolorbox}
\usepackage{rotating}
\usepackage{adjustbox}
\usepackage{floatpag}
\usepackage{yhmath}
\usepackage{colortbl}

\newcommand{\lineUp}[1]{
  \begin{scope}[shift={(#1)}]
    \draw (0, 0) -- (0, 1);   
  \end{scope}
}
\newcommand{\lineUpBendRight}[1]{
  \begin{scope}[shift={(#1)}]
    \draw (0,0) to[bend right=50] (0,1);
  \end{scope}
}
\newcommand{\lineUpBendLeft}[1]{
  \begin{scope}[shift={(#1)}]
    \draw (0,0) to[bend left=50] (0,1);
  \end{scope}
}
\newcommand{\lineUpBendLeftWArrow}[1]{
  \begin{scope}[shift={(#1)}]
    \draw[->] (0,0) to[bend left=50] (0,1);
  \end{scope}
}
\newcommand{\lineRightBendUp}[1]{
  \begin{scope}[shift={(#1)}]
    \draw (0,0) to[bend left=50] (1,0);
  \end{scope}
}

\newcommand{\lineRightBendDown}[1]{
  \begin{scope}[shift={(#1)}]
    \draw (0,0) to[bend right=50] (1,0);
  \end{scope}
}
\newcommand{\negCrossing}[1]{
  \begin{scope}[shift={(#1)}]
    \draw (0,0)  -- (0.35,0.35);
    \draw (0.65,0.65)  -- (1,1);
    \draw (1,0) -- (0,1);
  \end{scope}
}
\newcommand{\posCrossing}[1]{
  \begin{scope}[shift={(#1)}]
    \draw (0,1)  -- (0.35,0.65);
    \draw (0.65,0.35)  -- (1,0);
    \draw (0,0) -- (1,1);
  \end{scope}
}
\newcommand{\negCrossingWArrows}[1]{
  \begin{scope}[shift={(#1)}]
    \draw (0,0)  -- (0.35,0.35);
    \draw[->] (0.65,0.65)  -- (1,1);
    \draw[->] (1,0) -- (0,1);
  \end{scope}
}
\newcommand{\negCrossingWArrowsRight}[1]{
  \begin{scope}[shift={(#1)}]
    \draw (0,0)  -- (0.35,0.35);
    \draw[->] (0.65,0.65)  -- (1,1);
    \draw[<-] (1,0) -- (0,1);
  \end{scope}
}
\newcommand{\posCrossingWArrows}[1]{
  \begin{scope}[shift={(#1)}]
    \draw[<-] (0,1)  -- (0.35,0.65);
    \draw (0.65,0.35)  -- (1,0);
    \draw[->] (0,0) -- (1,1);
  \end{scope}
}

\tikzset{->-/.style={decoration={
  markings,
  mark=at position #1 with {\arrow{>}}},postaction={decorate}}}

\usepackage{amsmath}
\usepackage[all]{xy}
\usepackage{enumitem}
\usepackage{hyperref}
\usepackage{amssymb}
\usepackage{graphicx}
\usepackage{multirow}
\usepackage{mathtools}
\usepackage{ stmaryrd }

\usepackage{bm}

\usepackage{comment}

\usepackage[nottoc,numbib]{tocbibind}
\usepackage{csquotes}

\usepackage[doi=false,isbn=false,url=false,eprint=false,style=alphabetic,sorting=nyt]{biblatex}
\newtheorem{theorem}{Theorem}[section]

\newtheorem{lemma}[theorem]{Lemma}
\newtheorem{proposition}[theorem]{Proposition}
\newtheorem{conjecture}[theorem]{Conjecture}

\theoremstyle{definition}

\theoremstyle{remark}
\newtheorem{example}[theorem]{Example}
\newtheorem{remark}[theorem]{Remark}

\makeatletter \let\c@table\c@figure \makeatother

\newcommand{\Z}{{\mathbb {Z}}}

\newcommand{\Q}{{\mathbb {Q}}}

\newcommand{\C}{\mathcal{C}}

\newcommand{\overone}{0^{\boldsymbol{1}} 10}
\newcommand{\overx}{0^{\boldsymbol{x}} 10}

\newcommand{\bone}{\boldsymbol{1}}
\newcommand{\bx}{\boldsymbol{x}}
\newcommand{\bminus}{\boldsymbol{-}}
\newcommand{\term}[1]{\textit{#1}}
\newcommand{\bRoman}[1]{\textbf{\uppercase\expandafter{\romannumeral#1\relax}}}

\DeclareMathOperator*{\Mat}{\mathbf{Mat}}
\DeclareMathOperator*{\Cob}{\bm{\mathcal{C}\mathit{ob}}^3_{\bullet/l}}
\DeclareMathOperator*{\Kom}{\mathbf{Kom}}
\DeclareMathOperator*{\Hom}{Hom}
\DeclareMathOperator*{\Enh}{Enh}
\DeclareMathOperator*{\hdeg}{hdeg}

\title{Discrete Morse Theory for Khovanov Homology}

\author{Tuomas Kelomäki}

\begin{document}

\maketitle
\begin{abstract}

The standard methods for calculating Khovanov homology rely either on long exact/spectral sequences or on the algorithmic ``divide and conquer" approach developed by Bar-Natan. In this paper, we employ an alternative and arguably simpler tool, discrete Morse theory, which is new in the context of knot homologies.
The method is applied for 2- and 3-torus braids in Bar-Natan's dotted cobordism category, where Khovanov complexes of tangles live. This grants a recursive description of the complexes of 2- and 3-torus braids yielding an inductive result on integral Khovanov homology of links containing those braids. The result, accompanied with some computer data, advances  the recent progress on a conjecture by Przytycki and Sazdanović which claims that closures of 3-braids only have $\Z/2\Z$ torsion in their Khovanov homology.

\end{abstract}

\tableofcontents

\section{Introduction}

The knot invariant Jones polynomial \cite{JonesPolynomialOriginal} is defined by its skein relations and thus for a link diagram $L$ one can calculate the Jones polynomial $J(L)$ by repeatedly applying the relations (albeit exponentially many times). 
In the process of categorifying the Jones polynomial into Khovanov homology \cite{khovanov1999categorification}, the main skein relation gets transformed into a long exact sequence in homology. Iteratively using this long exact sequence would compel one to not only keep track of each picture but also the differentials in between them. In certain situations an approach like this can be taken directly \cite{Stosic2007}, but more often than not the standard trick is to push this information into one or more spectral sequences \cite{LeeSpectralSequence},\cite{TurnerSpectralSequenceQcoef}.

\begin{comment}
    
In \cite{FastKhovavnovComputations}, Bar-Natan took a theoretically simpler approach  to keep track of this data on a computer.
A part of his algorithm considers a sequence of chain homotopic complexes $\mathcal{C}_1,\mathcal{C}_2,\dots, \mathcal{C}_n$ where  $\mathcal{C}_{i+1}$ is obtained from $\mathcal{C}_{i}$ by a change of basis and a removal of a contractible subcomplex, or  ``Gaussian elimination" as he puts it. The problem of using Bar-Natan's algorithm as a human is that $n$ can be large and in order to move from  $\mathcal{C}_{i}$ to  $\mathcal{C}_{i+1}$ one needs to know whether a certain matrix element is an isomorphism  in $\mathcal{C}_i$. An algebraic version of discrete Morse theory by Sköldberg \cite{SkoldbergAlgebraicMorseTheory} (and independently by Jöllenbeck and Welker \cite{jollenbeck2009minimal}) provides an effective graph-theoretic condition ensuring that all of the needed  matrix elements  are indeed  isomorphisms. and enables one to move from $\mathcal{C}_1$ \textit{directly} to $\mathcal{C}_n$. % which in Sköldberg's terminology is called the Morse complex of $\mathcal{C}_1$.

\end{comment}

In \cite{FastKhovavnovComputations}, Bar-Natan took a theoretically simpler ``divide and conquer" approach  to keep track of this data on a computer. His \term{scanning algorithm} builds the Khovanov complex of a tangle by gluing smaller Khovanov complexes and iterating Gaussian eliminations in between. The difficulty in using Bar-Natan's algorithm as a human can be that the number of iterations is large and each time when performing a Gaussian elimination one needs to know whether a certain matrix element is an isomorphism. An algebraic version of discrete Morse theory by Sköldberg \cite{SkoldbergAlgebraicMorseTheory} (and independently by Jöllenbeck and Welker \cite{jollenbeck2009minimal}) provides an effective graph-theoretic condition ensuring that all of the needed  matrix elements are indeed  isomorphisms at every step. More importantly, it enables one to move from the initial chain complex \textit{directly} to the \term{Morse complex} without taking the steps in between. % which in Sköldberg's terminology is called the Morse complex of $\mathcal{C}_1$.

The main application of discrete Morse theory in this paper proves that \textit{Khovanov complexes of 3-torus braids with $k$ twists $(\sigma_1\sigma_2)^k$ have minimal, periodic and inductively defined representatives in their homotopy equivalence classes.} After conducting this work, the author noticed that Schütz had obtained a similar result for the same braids in the chronological cobordism complexes which define the odd Khovanov homology \cite{Schtz2022}. His proof used the scanning algorithm with a slight modification; complexes are glued via mapping cone construction instead of via the tensor product which Bar-Natan used for the even Khovanov homology. Since our discrete Morse theory based method does not use gluing, it can be adjusted for the odd Khovanov homology, see Remark \ref{Odd Khovanov homology remark}. A comparison of the two approaches is shown in Table \ref{Comparison of Scanning and DMT}. 
\begin{table}
    \centering
    \begin{tabular}{p{0.45\linewidth}|p{0.45\linewidth}}
  
  \textbf{Bar-Natan/Schütz scanning algorithm} & \textbf{Discrete Morse theory} \\
  \hline
  Scan through a tangle diagram, piece by piece. Whenever a crossing is scanned:
  \begin{enumerate}[label=\Roman*.]
      \item Glue the new small complex into the current one.
      \item Deloop any circles which are formed.
      \item Iterate Gaussian eliminations.
  \end{enumerate}
  & Take in the whole tangle diagram.
  \begin{enumerate}
      \item Deloop all of the circles.
      \item Find a suitable set of matrix elements.
      \item Cancel this set, all at once.
  \end{enumerate}
  \\
  
\end{tabular}
\caption{A rough comparison of Bar-Natan/Schütz scanning algorithm and discrete Morse theory. The advantage of Discrete Morse theory is that in total only 3 steps are needed and gluing is not required. The challenges in using discrete Morse theory lie in graph theory: finding the suitable set of arrows and classifying all of the paths between critical cells. }
%Graph theory is hidden in steps ....
%Future work, suitable set. Maybe leave this to be written after I choose what I want to do with presenting the conjecture.
%Delooping is potraid in figure asd
    \label{Comparison of Scanning and DMT}
\end{table}

Compared to long exact and spectral sequences, a significant benefit of the scanning algorithm and discrete Morse theory is that we can work on smaller pieces of links, tangles and braids. While tangles and braids do not admit homology per se, examining chain homotopy type of their complexes can yield homological results about any links that contain them. Our main theorem is of this flavour. 

\begin{comment}
    
%Moreover the representatives could be made explicit. %see Propositions \ref{first commutative square of T3} and \ref{second commutative square of T3}. 
After setting up the notation, approximately 90\% of the work for proving this is pure graph-theory while the remaining 10\% is straightforward algebra. 
% Before the main application of DMT, a previously known similar result cite,  is also obtained for [1^m].
Partly as a warm-up we also do a similar, but simpler calculation for 2-torus braids with $m$ twists $\sigma_1^m$ and in Section \ref{Section: Composing tangles} these tangles are composed. By comparing the gradings of these compositions we obtain the main theorem.

\end{comment}

\begin{figure}
    \centering
\begin{tikzpicture}[scale=0.75]
    \begin{scope}[shift={(-9,0)}]
      
    %\draw[dashed] (0,0) circle (3.75);

      \begin{scope}[shift={(1,1.5)},scale=0.4]
        
  \draw[dashed] (0.5,0) circle (2.5);

  \draw (-1, -1) rectangle (2, 1);

    \node at (0.5,0) {$\sigma_1^{m}$};

  % Arrows from the origin to nodes
  \draw[->] (1,1) -- (1,2.46);
  \draw[->] (0,1) -- (0,2.46);
  \draw[<-] (1,-1) -- (1,-2.46);
  \draw[<-] (0,-1) -- (0,-2.46);
  %\draw[->] (0,0) -- (-1,3.87);
  %\draw[->] (0,0) -- (1,-3.87);
  %\draw[->] (0,0) -- (-1,-3.87);

  \end{scope}

  \begin{scope}[shift={(-2,-2.5)},scale=0.3]
    \draw[dashed] (10, 3.5) circle (4);

    \draw (12.8, 5.8) rectangle (7.2, 1.2);

    \node at (10,3.5) {$(\sigma_1\sigma_2)^{k}$};

    \draw[<-] (9,1.2) -- (9,-0.38);
    \draw[<-] (11,1.2) -- (11,-0.38);
    
    \draw[->] (9,5.8) -- (9,7.38);
    \draw[->] (11,5.8) -- (11,7.38);

    \draw[<-] (10,1.2) -- (10,-0.5);
    \draw[->] (10,5.8) -- (10,7.5);
  \end{scope}

      \begin{scope}[shift={(-2,0)}, scale=0.4]
    %\draw[dashed] (0,0) circle (1);
    
    \draw[dashed] (0,0) circle (3);

    \negCrossing{0,-2}
    \posCrossing{1,-1}
    \posCrossing{-1,-1}
    \negCrossing{0,0}

    \lineUpBendLeftWArrow{-1,-2}
    \lineRightBendDown{-1,-2}
    \lineRightBendUp{1,1}
    \lineUpBendRight{2,0}

    \lineUp{-1,0}
    \negCrossing{-1,1}

    \draw[<-] (1,-2.84) -- (1,-2);
    
    \draw[<-] (2.84,-1) -- (2,-1);

    \draw[<-] (0,2) -- (0,3);
    \draw[<-] (-1,2) -- (-1,2.84);

    \end{scope}

    \draw[<-, looseness=1, out=270, in=270] (0.7,-2.6) to (-0.5,-2);
    \draw[-, looseness=1, out=90, in=0] (-0.5,-2) to (-0.88,-0.405);
    
    \draw[->, looseness=1, out=90, in=270] (0.7,-0.3) to (1,0.52);

    \draw[->, looseness=1, out=90, in=270] (1,-0.28) to (1.4,0.52);

    \draw[-, looseness=1, out=270, in=180] (-1.6,-1.12) to (-1,-3);
    \draw[->, looseness=1, out=0, in=270] (-1,-3) to (1,-2.62);

    \draw[-, looseness=1.5, out=90, in=90] (1.3,-0.3) to (3,-1.5);
    \draw[->, looseness=1.5, out=270, in=270] (3,-1.5) to (1.3,-2.6);

    \draw[->, looseness=0.7, out=90, in=90] (1,2.48) to (-2,1.2);
    
    \draw[->, looseness=1.3, out=90, in=90] (1.4,2.48) to (-2.4,1.1);

  \end{scope}
  
\end{tikzpicture}

    \caption{A link diagram $L(k,m)$ representing an example function $L$ for Theorem \ref{simplified algorithm theorem for introduction}. The boxes $(\sigma_1\sigma_2)^k$ and $\sigma_1^m$ exhibit the twisting of 3 and 2 strings for $k$ and $m$ times respectively, see Figures \ref{T2 braid figure} and \ref{Braid diagram of T3} for more specific illustrations.} %, or alternatively $D((\sigma_1\sigma_2)^{k_1},\sigma_1^m, T)$ with  3-input planar arc diagram $D$ and tangle $T$. }
    \label{Link diagram for simplified Theorem}
\end{figure}
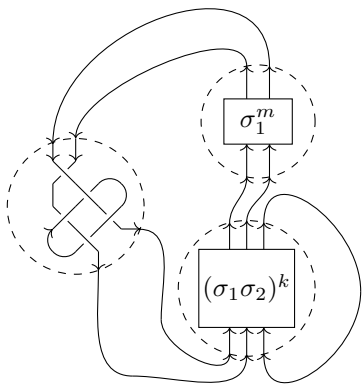

\begin{theorem} \label{simplified algorithm theorem for introduction}
    Let $L:\Z^2 \to \{\text{Links}\}$ be any function which creates a link $L(k,m)$ by uniformly gluing together a fixed tangle, a 3-torus braid with $k$ twists $(\sigma_1\sigma_2)^k$ and a 2-torus braid with $m$ twists $\sigma_1^m$; see Figure \ref{Link diagram for simplified Theorem}. Then, 
Khovanov homologies of finitely many $L(k,m)$ determine Khovanov homologies of all $L(k,m)$. Namely, there are finite index sets $K,M\subset \Z$ such that for all integers $k_1,m_1,i_1,j_1\in \Z$ there exist $k_2\in K, m_2\in M, i_2,j_2\in \Z$ for which we have an isomorphism of Khovanov homology groups
     $$
     \mathscr{H}^{i_1,j_1}L(k_1,m_1)\cong \mathscr{H}^{i_2,j_2}L(k_2,m_2).
     $$
\end{theorem}

\noindent A consequence of the theorem is that in order to find out the set of Khovanov homology groups (including all torsion groups) of the set of links $\{L(k,m) \mid k,m\in \Z\}$, one needs to only calculate a finite amount of data which can be done on a computer.

The Jones polynomial, as an Euler characteristic of Khovanov homology, is only affected by the rank of the Khovanov homology groups. Knot theorists have therefore been asking, which torsion groups can arise in Khovanov homology of various links? Using the scanning algorithm, Mukherjee and Schütz constructed links with $\Z/3^k\Z$, $\Z/5^k\Z$ and $\Z/7^k\Z$, $k\geq 1$ torsion in their Khovanov homologies \cite{Mukherjee2021}.  On the other hand, in \cite{TorsionInKhovanovHomologyOfSemi-adequateLinks} Przytycki and Sazdanović conjectured obstructions to torsion order from the number strands on a braid closure:
\begin{comment}
    
Which torsion groups can arise in Khovanov homology of various links?
Knot theorists have therefore been interested in which torsion groups can arise in  Khovanov homology, 

The Jones polynomial, as an Euler characteristic of Khovanov homology, is only affected by the rank of the Khovanov homology groups. Hence torsion in Khovanov homology has gained a particular interest among knot theorists, which is exemplified the following conjecture from \cite{TorsionInKhovanovHomologyOfSemi-adequateLinks}. %TÄHÄN MUNKERJEE-SCHUTZ 

\end{comment}
\begin{conjecture}[Przytycki, Sazdanović] \label{PS conjecture}
The Khovanov homology of closures of...
    \begin{enumerate}
        \item ... $3$-braids only contain $\Z / 2\Z$ torsion.  \label{conjecture: remaining conjecture}
\item ... $4$-braids do not contain $\Z / p^r\Z$ torsion with $p\neq 2$. \label{conjeture: case 2} %\alert{\cite{SearchforTorsioninKhovanovHomology}}
\item ... $4$-braids only contain $\Z / 2\Z$ or $\Z / 4\Z$ torsion. %\alert{\cite{SearchforTorsioninKhovanovHomology}}
\item ... $n$-braids do not contain $\Z / p^r\Z$ for $p>n$ ($p$ prime). %\alert{\cite{OnOddTorsioninEvenKhovanovHomology}}
\item ... $n$-braids do not contain $\Z / p^r\Z$ for $p^r >n$. \label{conjecture: case 5} %\alert{\cite{SearchforTorsioninKhovanovHomology}}
    \end{enumerate}
\end{conjecture}
\noindent The closure of a braid indicates the link which is obtained by connecting the ends of the braid.  A $3$-braid and its braid closure are illustrated in Figure \ref{Braid closure figure} and it is noteworthy that every link can be realized as a braid closure. Parts \ref{conjeture: case 2} to \ref{conjecture: case 5} of Conjecture \ref{PS conjecture} have been disproven with counterexamples \cite{SearchforTorsioninKhovanovHomology}, \cite{OnOddTorsioninEvenKhovanovHomology} but Part \ref{conjecture: remaining conjecture} still remains open. 

\begin{figure}[h]
  \centering

\begin{tikzpicture} [scale=0.5]
%quickdirty way to align
%\node at (0,-1) {};

    \pgfmathsetmacro{\sideToCircle}{0.358}

    \draw[dashed]  (1, 3) circle (3.5);

    \negCrossing{0,0}
    \lineUp{2,0}
    \negCrossing{1,1}
    \lineUp{0,1}
    \posCrossing{1,2}
    \lineUp{0,2}
    \negCrossing{0,3}
    \lineUp{2,3}
    \negCrossing{1,4}
    \lineUp{0,4}
    \posCrossing{0,5}
    \lineUp{2,5}

    \draw[->] (0,6) -- (0,6+\sideToCircle);
    \draw[->] (2,6) -- (2,6+\sideToCircle);
    \draw[->] (0,-\sideToCircle) -- (0,0);
    \draw[->] (2,-\sideToCircle) -- (2,0);
    \draw[->] (1,6) -- (1,6.5);
    \draw[->] (1,-0.5) -- (1,0);

    \node at (-2.2,3) {$*$};

    %\caption{First subfigure}

\begin{scope}[shift={(8,0)}]

    \draw[dashed]  (1, 3) circle (3.5);
            \negCrossing{0,0}
    \lineUp{2,0}
    \negCrossing{1,1}
    \lineUp{0,1}
    \posCrossing{1,2}
    \lineUp{0,2}
    \negCrossing{0,3}
    \lineUp{2,3}
    \negCrossing{1,4}
    \lineUp{0,4}
    \posCrossing{0,5}
    \lineUp{2,5}
    
    \draw[->] (0,6) -- (0,6+\sideToCircle);
    \draw[->] (2,6) -- (2,6+\sideToCircle);
    \draw[->] (0,-\sideToCircle) -- (0,0);
    \draw[->] (2,-\sideToCircle) -- (2,0);
    \draw[->] (1,6) -- (1,6.5);
    \draw[->] (1,-0.5) -- (1,0);

    \draw (2,6+\sideToCircle) to[bend left=50] (5,6);
    \draw (1,6+0.5) to[bend left=50] (5.5,6);
    \draw (0,6+\sideToCircle) to[bend left=50] (6,6);

    \draw (2,-\sideToCircle) to[bend right=50] (5,0);
    \draw (1,-0.5) to[bend right=50] (5.5,0);
    \draw (0,-\sideToCircle) to[bend right=50] (6,0);

    \draw[->] (5,6) -- (5,0);
    \draw[->] (5.5,6) -- (5.5,0);
    \draw[->] (6,6) -- (6,0);

    \node at (-2.2,3) {$*$};
  \end{scope}

    \end{tikzpicture}

  \caption{A braid diagram for word $\sigma_1 \sigma_2 \sigma_2^{-1} \sigma_1 \sigma_2 \sigma_1^{-1}$ (left) and its braid closure (right).}
  \label{Braid closure figure}
\end{figure}
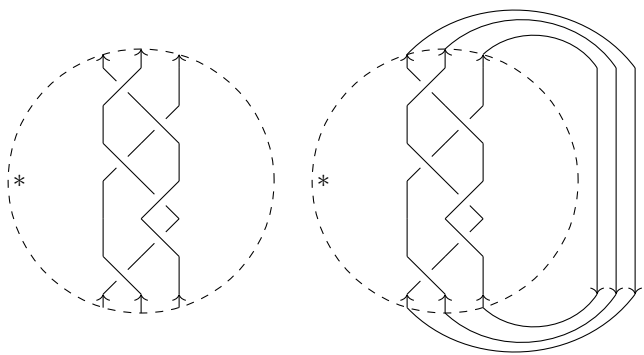

% muokkaa tätä kappaletta laita location/thinness kappale tämän eteen, jolloin CLSS tekniikat tulevat esiteltyä schumakowiczin valossa.

The location of non-trivial homology groups in the bigraded Khovanov homology table is another aspect of Khovanov homology, which is not picked up by the Jones polynomial and which has been investigated in the literature. %In addition qualitative analysis on individual homology groups, i.e.,  questions about torsion, the location of the non-trivial homology groups in the bigraded Khovanov homology table has also been investigated in the literature. 
\term{Homological thickness} is the minimal number of adjacent diagonals on which the Khovanov homology of a link is supported  and it was determined for all closures of 3-braids in \cite{Lowrance2011}. Particularly nice are links with homological thickness 2, called \term{thin links}, which contain all non-split alternating links. Building on the work of Lee  \cite{LeeSpectralSequence}, Shumakovitch proved that the Khovanov homology of thin links admit only $\Z/2\Z$ torsion and it can be read from the Jones polynomial and the signature of the link \cite{Shumakovitch2021}. 

The global spectral sequence techniques used by Shumakovitch were made local in \cite{chandler_lowrance_sazdanović_summers_2022} to cover thin regions of the homological table and allowing an application towards closures of 3-braids. By using an old classification of 3-braids  \cite{murasugi1974closed}, the task of proving Part \ref{conjecture: remaining conjecture} of Conjecture \ref{PS conjecture} was broken down into seven steps, $\Omega_0,\dots,\Omega_6$ defined in Table \ref{table for defining Omega0 to Omega6}. Then, combining prior calculations with $\Q$ and $\Z/ 2\Z$ coefficients \cite{TurnerSpectralSequenceQcoef}, \cite{benheddi2017khovanov} and the new techniques, the authors also verified that closures of braids from the sets $\Omega_0,\dots ,\Omega_3$ contain only $\Z/2\Z$ torsion in their Khovanov homology. 

\begin{comment}
    
An old classification of 3-braids  \cite{murasugi1974closed} was picked up in \cite{chandler_lowrance_sazdanović_summers_2022}, where the authors broke the task of proving Part \ref{conjecture: remaining conjecture} into showing it for sets of braids $\Omega_0,\dots,\Omega_6$ defined in Table \ref{table for defining Omega0 to Omega6}. By combining many spectral sequences and older calculations with $\Q$ and $\Z/ 2\Z$ coefficients \cite{TurnerSpectralSequenceQcoef} \cite{benheddi2017khovanov}, the authors also verified that the conjecture holds for the sets $\Omega_0,\dots ,\Omega_3$.  %, see Theorems \ref{Omega0 ... Omega3 contain only 2-torsion} and \ref{Omega4 and Omega5 contain only 2 torsion}.
\end{comment}

\begin{table}[h] 
    \centering
    \begin{tikzpicture}
    \node at (0,0) {$\begin{aligned}
    \textcolor{blue}{\checkmark} \qquad\qquad &\textcolor{blue}{\checkmark} \qquad &&\Omega_0 = \{(\sigma_1\sigma_2)^{3k}\mid k\in \Z\} \\
    \textcolor{blue}{\checkmark}\qquad\qquad&\textcolor{blue}{\checkmark} &&\Omega_1 = \{(\sigma_1\sigma_2)^{3k+1}\mid k\in \Z\} \\
    \textcolor{blue}{\checkmark} \qquad\qquad&\textcolor{blue}{\checkmark} &&\Omega_2 = \{(\sigma_1\sigma_2)^{3k+2}\mid k\in \Z\} \\
    \textcolor{blue}{\checkmark}\qquad\qquad&\textcolor{blue}{\checkmark} &&\Omega_3 = \{(\sigma_1\sigma_2)^{3k}\sigma_1\mid k\in \Z\} \\
    &\textcolor{blue}{\checkmark} &&\Omega_4 = \{(\sigma_1\sigma_2)^{3k}\sigma_1^{-m}\mid m\in \Z_{>0}, k\in \Z\} \\
    &\textcolor{blue}{\checkmark} &&\Omega_5 = \{(\sigma_1\sigma_2)^{3k}\sigma_2^{m}\mid m\in \Z_{>0}, k\in \Z\} \\
    & &&\Omega_6 = \{(\sigma_1\sigma_2)^{3k} \sigma_1^{-p_1} \sigma_2^{q_1}\dots \sigma_1^{-p_r} \sigma_2^{q_r}\mid r,p_i,q_i\in \Z_{>0}, k\in \Z\}
    \end{aligned}$   
    };
    %\node at (-7.5,2.2) {PS Conjecture \\
    %Part 1 \\ Proven in: };
    \node[align=center, font=\footnotesize] at (-5.9,2.35) {Proven in \\ \cite{chandler_lowrance_sazdanović_summers_2022}:};
    \node[align=center, font=\footnotesize] at (-4.2,2.37) {Proven in  \\ this article: };
\end{tikzpicture}
    
    \caption{Proving Part \ref{conjecture: remaining conjecture} of Conjecture \ref{PS conjecture} amounts to showing that closures of 3-braids from the sets  $\Omega_0,\dots,\Omega_6$ only contain  $\Z/2\Z$ torsion in their Khovanov homology. }
    \label{table for defining Omega0 to Omega6}
\end{table}

In addition to reproving the conjecture for the cases $\Omega_0,\dots ,\Omega _3$ with discrete Morse theory, we take a step further and prove it for $\Omega_4$ and $\Omega_5$. With Theorem \ref{simplified algorithm theorem for introduction} the closures of infinite families of braids $\Omega_0,\dots, \Omega_5$ are reduced to a finite set of 1077 links (13 for $\Omega_0,\dots ,\Omega_3$ and 532 for $\Omega_4$ and $\Omega_5$ each).  The Khovanov homologies of these 1077 links are calculated with software KHOCA \cite{Khocasoftwarearticle}, contain only $\Z/2Z$ torsion and can be accessed in \cite{computerDataFor3BraidPaper}. In contrast to the links of type $\Omega_0,\dots ,\Omega_3$, which are locally thin, the links of type $\Omega_4$ and $\Omega_5$ can be arbitrarily locally thick. More specifically, for all $t\geq 4 $ there exists a link $L$ of type $\Omega_4$ with nontrivial Khovanov homologies $\mathscr H^{0,1}(L)\cong\Z \cong \mathscr H^{0,-2t+1}(L)$.

%the closure $L$ of $(\sigma_1\sigma_2)^{3t}\sigma_1^{-2t}$ has   

%Furthermore, we construct 3-braids which are arbitrarily locally thick, that is, for all $l\in \Z$ asdasdasd

%The hardest class of braids $\Omega_6$ still remains open, and proving the conjecture for it would likely require a more in depth analysis on the specific compositions of 2- and 3-torus braids. 

\begin{comment}
    
After conducting this work, the author found out that similar small and recursive representatives of homotopy classes for odd Khovanov complexes of 3-strand torus braids were obtained in cite. Although it was not written there explicitly, simple modifications could have also lead to nice representatives for complexes of in the even case, a result which would match our main application of Discrete Morse theory. The main difference with our independent work is that the single simplification by Discrete Morse theory does not involve glueing of chain complexes with tensor products or mapping cones as in cite.  
While discrete Morse theory has not been explicitly used for knot homologies previously the B
\end{comment}

%munkerjee Schutz tähän "machine assisted gaussian elimination"

While discrete Morse theory has not been explicitly used for knot homologies previously, somewhat similar ``large scale Gaussian elimination" was used for Rouquier complexes \cite{maltoni2022reducing} which define the triply graded HOMFLY-PT homology \cite{KhovanovTriplyGrHomology}. On the other hand, discrete Morse theory was explicitly used for chromatic homologies of graphs \cite{CHANDLER_SAZDANOVIC_BrokenCircuitsChormaticHomology} which on some gradings agree with Khovanov homology of an associated link \cite{Khovanov_Chromatic_Homology_2018}.

%ennen kuin tämä kappale valmiiksi, niin liitä munkerjee joteinkin

%If Khovanov homology of a link is supported on 2 adjacent diagonals, we say that it is \term{homologically thin}. Non

\begin{comment}
    
Khovanov homology is bigraded and an interesting 
Thickness
Shumakovich thin are only 2 torsion
summers 3 braids are thick
CLSS techniques work on locally thin regions of Khovanov homology table
knight move

\end{comment}

\textbf{Acknowledgement:} The author was supported by The Emil Aaltonen Foundation Grant 220085.

\section{Preliminaries} \label{section: preliminaries}

    A category $\mathbf{C}$ is called \term{additive}, if  it satisfies:
    \begin{enumerate}
        \item  For any objects $A,B\in \mathbf{C}$ the set of morphisms $\Hom(A,B)$ is a $\Z$-module  and we have bilinearity of composition of morphisms.%: $(g+h)\circ f=g\circ f+h\circ f$ and $f\circ(g+h) =f\circ g+f\circ h$ whenever either compositions makes sense.
        \item The category $\mathbf{C}$ contains a zero object and for any objects $A,B$ there exist a coproduct $A\oplus B\in \mathbf{C}$. %with canonical projection and inclusion morphisms $\pi_A, \pi_B, \iota_A, \iota_B$.
    \end{enumerate}
For any two finite coproducts $\bigoplus_j A_j$ and $\bigoplus_k B_k$  in an additive category there is a $\Z$-module isomorphism
$$
\Hom\Big(\bigoplus_j A_j,\bigoplus_k B_k\Big)\cong  \bigoplus_{j,k}\Hom(A_j,B_k).
$$
The morphisms of $\Hom(A_j,B_k)$ are referred to as  \term{matrix elements} and we often define a map between coproducts by describing them.

 A \term{cochain complex} $(\C, d)$ over an additive category $\mathbf{C}$ is a pair, where $\mathcal{C}=(\C^i)_{i\in \Z}$ is a collection of objects $\mathcal{C}^i$ called \term{chain spaces}. We limit to cochain complexes, henceforth simply called \term{complexes}, for which  $\C^i\not \cong 0$ with only finitely many $i$. For morphisms $d=(d^i\colon \C^i \to \C^{i+1} )_{i\in \Z}$ we assume $d^{i+1}d^i=0$ so in other words, we take the convention that differentials go up in \term{homological grading}. All categories will be bolded: $\mathbf{C}, \mathbf{D}$ and all complexes will be written with calligraphic font: $\mathcal{C}, \mathcal{D}$. We denote $\Kom(\mathbf{C})$ for the category of complexes over $\mathbf{C}$ and notice that there is enough structure $\Kom(\mathbf{C})$ to talk about chain homotopy. However the concept of homology might not make sense in $\Kom(\mathbf{C})$, if we do not have kernels and cokernels in $\mathbf{C}$ as will be the case for Khovanov complexes of tangles. 

An additive category $\mathbf{C}$ is called \term{graded} if the following conditions hold:
    \begin{enumerate}
        \item  For any objects $A,B\in \mathbf{C}$ the set of morphisms $\Hom(A,B)$ is a graded $\Z$-module with $\deg(g\circ f)=\deg (g)+\deg( f)$, whenever the composition makes sense.
        \item There is a grading shift $\Z$-action on the objects $(m,A)\mapsto A\{m\}$. If one forgets the grading, this action leaves the morphisms unchanged: $f\in \Hom(A,B)$ can be also considered as an element in $\Hom(A\{m\},B\{n\})$. With gradings however, if $f\in \Hom(A,B)$ has degree $d$, in $\Hom(A\{m\},B\{n\})$ it will have degree $d+n-m$.

    \end{enumerate}
 Morphisms of a complex $\mathcal{C}$ over a graded category $\mathbf{C}$ have two gradings, homological and the \term{internal grading} from $\mathbf{C}$. The \term{bigraded complexes} that we will study are Khovanov complexes and all morphisms in them have internal degree zero.  Bigraded complexes admit homological and internal grading shift operations, denoted by $[\cdot]$ and $\{\cdot\}$ and defined by $(\C[t])^i=\C^{i-t}$ and $(\C\{t\})^i=(\C^i)\{t\}$.

\subsection{Algebraic Discrete Morse theory}

We say that a complex $(\mathcal{C},d)$ over an additive category $\mathbf{C}$ is \term{based} if there exists index sets $(J_i)_{i\in \Z}$ and objects $\mathcal{C}^i_j$, such that $\mathcal{C}^i=\bigoplus_{j\in J_i} \mathcal{C}^i_j$.
The summands of chain spaces are frequently referred to as \term{cells} which they were if one would apply this theory to cellular homology. For a based complex we define arrows $d_{j\to k}^i\colon \mathcal{C}^i_j \to \mathcal{C}^{i+1}_k$, where $j\in J_i$ and $k\in J_{i+1}$, as the matrix elements of the differential $d^i$.

\begin{lemma}\label{reversing one arrow}
    Assume that $d^i_{C\to E}$ is an isomorphism. The based complex 
%https://tikzcd.yichuanshen.de/#N4Igdg9gJgpgziAXAbVABwnAlgFyxMJZARgBpiBdUkANwEMAbAVxiRAEEQBfU9TXfIRQAmUgAYqtRizYAhbrxAZseAkVHDJ9Zq0QgAwgr4rBRAMzit03SAAiRpf1VDkFzdW0y9AUQfKBaigALORWOmwAYn5OpihioR7WbNEmgcgArAlS4XrckjBQAObwRKAAZgBOEAC2SPEgOBBIotleIABGDpU1ddSNSGQgDHTtMAwACjGBQzBlOCCJOSB0XVW1iIP9iCFDI2OTqUIgFViFABbzi21Qqz2ILVsWrTas1MOjE1NHDLPzPOVrJBPLbpf4gbrrHYgsEQgZ9JqIJ6eGwAY1u6weCJ2yLYZXRSAAbPDes82AA6Cn4xCZBoIgDsXAoXCAA
\[
\begin{tikzcd}
                    &                                   & \mathcal{C}^{i}_B \arrow[rdd, "d^{i}_{B\to E}" {near start, description}] \arrow[r, "d^{i}_{B\to D}"] &[4em] \mathcal{C}^{i+1}_D \arrow[rd,"d^{i+2}_{D\to E}"] &             &    \\
\cdots \arrow[r] & \mathcal{C}^{i-1}_A \arrow[rd, ""] \arrow[ru, "d^{i-1}_{A\to B}"] &           \oplus                         &     \oplus         & \mathcal{C}^{i+2}_F \arrow[r] & \cdots \\
                    &                                   & \mathcal{C}^{i}_C \arrow[ruu, "d^{i}_{C\to D}" {near start, description}] \arrow[r, "d^{i}_{C\to E}"]  & \mathcal{C}^{i+1}_E \arrow[ru, ""]  &             &   
\end{tikzcd}
\]    
is isomorphic to the based complex 
\[
\begin{tikzcd}[row sep=small]
                    &                                   & \mathcal{C}^{i}_B\arrow[r, "f"] &[4em] \mathcal{C}^{i+1}_D \arrow[rd,"d^{i+1}_{D\to F}"] &             &    \\
\cdots \arrow[r] & \mathcal{C}^{i-1}_A  \arrow[ru, "d^{i-1}_{A\to B}"] &           \oplus                         &     \oplus         & \mathcal{C}^{i+2}_F \arrow[r] & \cdots \\
                    &                                   & \mathcal{C}^{i}_C  \arrow[r, "d^{i}_{C\to E}"]  & \mathcal{C}^{i}_E   &             &   
\end{tikzcd}
\]
which is chain homotopy equivalent to the based complex
\[
\begin{tikzcd}
\cdots \arrow[r] & \mathcal{C}^{i-1}_A \arrow[r,"d^{i-1}_{A\to B}"] & \mathcal{C}^{i}_B\arrow[r, "f"] &[4em] \mathcal{C}^{i+1}_D \arrow[r,"d^{i+1}_{D\to F}"] & \mathcal{C}^{i+2}_F \arrow[r] & \cdots
\end{tikzcd}
  \]  
where $f=d^{i}_{B\to D}-d^{i}_{C\to D}(d^{i}_{C\to E})^{-1}d^{i}_{B\to E}$.

\end{lemma}
\begin{proof}
    The proof is simple linear algebra and can be found in \cite{FastKhovavnovComputations} where it is appropriately named ``Gaussian elimination, made abstract". 
\end{proof}

Upon finding an isomorphism between two summands, Lemma \ref{reversing one arrow} enables moving to a smaller complex with less summands. This immediately yields a naive algorithm where one iteratively uses the previous lemma until no isomorphisms are found.  Such a procedure was part of the scanning algorithm for Khovanov homology in \cite{FastKhovavnovComputations} which lead huge leap in computation times and in turn enabled discoveries of theoretically interesting counter-examples including some of the ones disproving parts of Conjecture \ref{PS conjecture}. 
On the other hand this naive algorithm does not really seem like a nice way for a human to do calculations as keeping track of the complexes for each reversal of an arrow could be tedious. %Nonetheless, this approach was taken in \cite{KhovanovHomologyOfPrezelsRevisited} where an induction argument simplified the bookkeeping. 
Luckily there is a systematic way of streamlining and avoiding most of this work: algebraic discrete Morse theory.% formulated by Sköldberg. %\cite{SkoldbergAlgebraicMorseTheory}.

 Let $M$ be a subset of the matrix element arrows $d^i_{j\to k}$ for a based complex $\mathcal{C}$. We define a directed graph $G(\mathcal{C},M)$, written $G$ when $\mathcal{C}$ and $M$ are clear from the context, whose vertices are the cells $\mathcal{C}^i_j$. The directed edges of $G$  are non-zero arrows $d_{j\to k}^i$ with the arrows in $M$ having their direction reversed.  We call $M$ a \term{Morse matching} if four conditions are met:
 \begin{itemize}
     \item  $M$ is finite.
     \item  $M$ induces a partial matching for the graph $G(\mathcal{C},M)$, i.e. no vertex of $G$ is touched by two edges of $M$.
     \item  Every arrow of $M$ is an isomorphism in $\mathbf{C}$.
     \item  $G(\mathcal{C},M)$ has no directed cycles.
 \end{itemize}
%, it is a partial matching, i.e. no vertex is touched by two arrows in $M$, every arrow of $M$ is an isomorphism and $G(\mathcal{C},M)$ has no directed cycles.
We are particularly interested in paths which alternate between two neighbouring homological levels and call them \term{zig-zag paths}. To indicate that an edge $z\to x$ in $G$ goes up (or down) by one in homological degree, we write $z\nearrow x$ (or $z\searrow x$  respectively, in which case $(x\nearrow z)\in M$). If $M$ is a partial matching and we want to prove acyclicity, it suffices to check there does not exist a zig-zag cycle $z_1\nearrow x_1 \searrow z_2 \nearrow ... \nearrow x_n\searrow z_1$. We call a cell $\C^i_j$, which is not touched by an arrow of $M$, \term{unmatched} or \term{critical} and denote the set of critical cells with homological degree $i$ by $U^i$. Given that $M$ is a Morse matching, in order to find out all paths between critical cells $\mathcal{C}^i_j\to \mathcal{C}^{i+1}_k$ it is enough to consider only zig-zag paths.

%A directed graph is not a category, but we can make $G(\mathcal{C},M)$ into a category $\mathbf{P}(\mathcal{C},M)$, shorted $\mathbf{P}$, whose objects are vertices of $G$ and whose morphisms are paths. The paths of length zero, which we allow, play the role of identities. Given that $M$ is a Morse matching we can define a covariant \term{remembering functor} $R(\mathcal{C},M)\colon \mathbf{P}(\mathcal{C},M) \to \mathbf{C}$, shorted $R$, which from inputs $\mathcal{C}$ and $M$ recalls the data we forgot when passing from $\mathbf{C}$ to the graph $G$. Formally, we put $R(x)=x$ on the objects and for $f\in M$ we assign $Rf=f^{-1}$ and $Rg=g$ whenever $g\notin M$.

A directed graph is not a category, but we will need to make $G(\mathcal{C},M)$ into a category in order use the concept of a natural transformation. The category of paths $\mathbf{P}(\mathcal{C},M)$, shorted $\mathbf{P}$, has vertices of $G$ as objects and whose paths as morphisms. The paths of length zero, which we allow, play the role of identities. Given that $M$ is a Morse matching on $\mathcal{C}$ we can define a covariant \term{remembering functor} $R(\mathcal{C},M)\colon \mathbf{P}(\mathcal{C},M) \to \mathbf{C}$, shorted $R$, which from inputs $\mathcal{C}$ and $M$ recalls the data we forgot when passing from $\mathbf{C}$ to the graph $G$. Formally, we put $R(x)=x$ on the objects and for $f\in M$ we assign $Rf=f^{-1}$ and $Rg=g$ whenever $g\notin M$.

%We call a complex $(\mathcal{C},d)$ over $\mathbf{C}$ based if there exists an index set $J=\bigsqcup_{j\in \Z} J_i$ and objects $\mathcal{C}^i_j$, which we call cells, such that $\mathcal{C}^i=\bigoplus_{j\in J_i} \mathcal{C}^i_j$.  %In the original geometric theory of Forman the object..... We call the objects $\mathcal{C}^i_j$ cells as that is the role which they would take if one applies this theory to a cellular complex. 

\begin{theorem}[Discrete Morse theory] \label{Discrete Morse theory}
    Let $(\mathcal{C},d)$ be a based complex over $\mathbf{C}$ with a Morse matching $M$. Then $\mathcal{C}$ is chain homotopy equivalent to the ``Morse complex" $(M\C, \partial)$ with chain spaces
    $$
    %D^i= \bigoplus_{\substack{j\in J_i\\ j \text{ not matched by arrow in } M}} A^i_j
    (M\C)^i= \bigoplus_{\mathcal{C}^i_j\in U^i} \mathcal{C}^i_j.
    $$
    The differential $\partial$ is defined by matrix elements with
    $$
    \partial^i_{j\to k} =\sum_{p\in N} (-1)^{r(p)} Rp
    $$
    where $N=\operatorname{Mor}_{\mathbf{P}}(\mathcal{C}^i_j, \mathcal{C}^{i+1}_k)$  and $r(p)$ is the number of arrows of $M$ in $p$.

%(paths going up and down between the homological layers) 
    
\end{theorem}

\begin{proof}
    We order the set $M$ into a sequence
    $$
    d^{i_1}_{j_1\to k_1} , \dots ,d^{i_n}_{j_n\to k_n}
    $$
    in a way so that the matrix element $d^{i_m}_{j_m\to k_t} $ is a zero-map whenever $t<m$ and $i_m=i_t$. Such order can be found as we assumed that $G$ does not have directed cycles. Then we can use Lemma \ref{reversing one arrow} inductively at each arrow of the sequence which simultaneously proves the claim as well as verifies that $M\mathcal{C}$ in fact is a complex, that is, $\partial^{i+1}\partial^i=0$.  
\end{proof}

%As a reassurance that we have understood the notation correctly, we can recover Lemma \ref{reversing one arrow} from Theorem \ref{Discrete Morse theory} by reversing a single arrow. 
In \cite{SkoldbergAlgebraicMorseTheory} the previous theorem is written with slightly different notation and formulated in $\mathbf{Mod}_A$, that is the category of (left) modules over a fixed ring $A$. There is a benefit in working with modules: infinite direct sums make sense in $\mathbf{Mod}_A$ as do complexes with them and with infinitely many arrows. In that paper infinite matchings are allowed which forces the acyclicity requirement of the graph to get replaced by a well-foundedness requirement.

\begin{example}
    Let $\mathcal{C}$ be a complex over $\mathbf{Mod}_{\Z}$ with $\mathcal{C}^i\cong 0$ for every $i$, except for $i=0$ and $i=1$. The chain spaces $\mathcal C^0$ and $\mathcal C^1$ and the differential $d^0$ are given by 
    $$
    % https://tikzcd.yichuanshen.de/#N4Igdg9gJgpgziAXAbVABwnAlgFyxMJZAFgBoAGAXVJADcBDAGwFcYkQAdDgW3pwAsARoOAAtAL4hxpdJlz5CKMgEZqdJq3ZdeA4WPEB9YlJkgM2PASIA2CmoYs2iTjz5CREoydkWFN0qo0DprO2m56nsbSPvJWKABMdkEaTi467vreZnKWisiJgeqOWq66HpLR2b5xyACsSUUhLhBoLAiV5rF59YXBqVwtbVmduUQAzAH2KSWDzO2mI34oE1TJxaEcs-MxoyjkDX0l6QDGTAAEAMIAeuQAvMM5S8j7vdMbJ+fXyvfiajBQAHN4ERQAAzABOEG4SDIIBwECQtUqEKhMJo8KQY2RkOhiH2cIRiCRphRuMSBKQymxqLx6MJVJJOKQ5IxiCxlHEQA
\begin{tikzcd}[row sep=1.2cm, column sep=0.3cm]
\mathcal C^0= &  & \mathbb{Z} \arrow[d] \arrow[rrrrd] & \oplus & \mathbb{Z} \arrow[lld] \arrow[d] & \oplus & \mathbb{Z}/4\mathbb Z \arrow[lld] \arrow[d] \\
\mathcal C^1= &  & \mathbb{Z}                         & \oplus & \mathbb{Z}/4\mathbb Z                     & \oplus & \mathbb{Z}/4\mathbb Z                      
\end{tikzcd} 
    \qquad \quad 
    d^0= \begin{bmatrix}
\operatorname{id}_{\mathbb Z} & \operatorname{id}_{\mathbb Z} & 0 \\
0 & \pi & \operatorname{id}_{\mathbb Z / 4 \Z} \\
\pi & 0 & \operatorname{id}_{\mathbb Z/4 \Z} 
\end{bmatrix}  
    $$
    where $\pi$ is the canonical projection. Choosing $M$ to be the set of the two south-west pointing arrows (entries (1,2) and (2,3) of the matrix $d^0$) leads to the directed graph
    \begin{equation}
    \begin{tikzcd}[row sep=0.25cm, column sep=0.5cm]
 &  & \bullet \arrow[dd] \arrow[rrrrdd] &  & \bullet \arrow[lldd,<-] \arrow[dd] &  & \bullet \arrow[lldd,<-] \arrow[dd] \\
G(\mathcal{C},M)= & & & & & & \\
 &  & \bullet                         &  & \bullet                     &  & \bullet                      
\end{tikzcd}    \label{Equation: example graph G(C,M)}
    \end{equation}
from which we can see that $M$ is a Morse matching. Thus $\mathcal{C}$ is homotopy equivalent to the Morse complex $M\mathcal{C}$ whose non-zero chain spaces are given by the two unmatched cells
$$
% https://tikzcd.yichuanshen.de/#N4Igdg9gJgpgziAXAbVABwnAlgFyxMJZABgBpiBdUkANwEMAbAVxiRAAoBZAAgB1eAtnRwALAMaNuAYQCUAPWIBeEAF9S6TLnyEUZAIxVajFmy59Bw8ZNly9ytRux4CRAEzlD9Zq0Qh+Q0QAjQO4ALVV1EAwnbTdSA2ovE19-S2Cw7gB6bgAWcwCRdPCVQxgoAHN4IlAAMwAnCAEkdxAcCCQAZkTjHxAa1QoVIA
\begin{tikzcd}
(M \mathcal C)^0=   & \mathbb Z \arrow[d, "\partial^0"] \\
(M \mathcal C)^1=   & \mathbb Z / 4 \mathbb Z. 
\end{tikzcd}
$$
The differential map $\partial^0$ is the sum of the two paths from top-left to bottom-right in Graph \ref{Equation: example graph G(C,M)} $$\partial^0= (-1)^0\pi +(-1)^2 \operatorname{id}_{\mathbb Z / 4 \Z}\  \operatorname{id}_{\mathbb Z / 4 \Z}^{-1}  \  \pi \   \operatorname{id}_{\mathbb Z}^{-1} \  \operatorname{id}_{\mathbb Z} =2\pi$$
or equivalently $\partial^0(x)=[2x]$.
\end{example}

The next proposition allows us to compare Morse complexes with each other by mostly using graph theory. We define the category of critical paths $\mathbf{CP}(\mathcal{C}, M, i)$, shorted $\mathbf{CP}$, as a subcategory of $\mathbf{P}(\mathcal{C}, M)$ and denote $I_{\mathcal{C}}= I(\mathcal{C}, M, i)$ as the corresponding inclusion functor. The morphisms of $\mathbf{CP}(\mathcal{C}, M, i)$ are limited to sub-paths of paths between critical cells of homological degree $i$ and $i+1$ and the objects are vertices along those paths. 

%The point of discrete Morse theory for us is that it enables us to work in a simpler setting, a graph, and only when forced  we will go from $G$ to work in the category $\mathbf{C}$. We hope to illustrate this benefit by applying the following proposition in our proceedings. 

\begin{proposition}\label{Morse comparison with natural transformation}

    Let $\mathcal{C}$ and $\mathcal{D}$ be complexes over $\mathbf{C}$ and let $M$ and $N$ be Morse matchings of them respectively. Suppose there is a functor $F\colon \mathbf{CP}(\mathcal{C}, M, i)\to \mathbf{CP}(\mathcal{D}, N, i)$ that satisfies the following:
    \begin{itemize}
        \item The functor $F$ sends unmatched cells of $\mathcal{C}$ bijectively to unmatched cells of $\mathcal{D}$ and preserves homological degree.
        
        \item Suppose $c_1$ and $c_2$ are unmatched cells. The functor $F$ induces a bijection $\operatorname{Mor}(c_1,c_2)\to\operatorname{Mor}(F(c_1), F(c_2))$ and if $p\in \operatorname{Mor}(c_1,c_2)$, then the number of reversed edges in $p$ and $Fp$ agree modulo 2.
        
        %\item For any edge $e$, the number of reversed arrows in $e$ and $Fe$ agrees modulo 2.

        \item There exists a natural transformation $\gamma\colon R(\mathcal{C}, M)\circ I_{\mathcal{C}} \Rightarrow R(\mathcal{D}, N) \circ I_{\mathcal{D}} \circ F$. 
    \end{itemize}
    Denote by $\pi_a$ and $\iota_a$ the canonical projections and inclusions and define a diagonal matrix morphism $f^{i}=\sum_{c\in U^i} \iota_{F(c)}\gamma_c \pi_c$ which is an isomorphism if all $\gamma_c$ are. Define $f^{i+1}$ similarly and we have a commutative diagram of chain spaces
    $$
    % https://tikzcd.yichuanshen.de/#N4Igdg9gJgpgziAXAbVABwnAlgFyxMJZABgBpiBdUkANwEMAbAVxiRAFkAdTgWzpwAWAY0bAAwgF8AelhATS6TLnyEUZAIxVajFmy69+w0ZKnAsAanUS5CkBmx4CRdaU3V6zVog7c+gkQzAACLSZpbW8ooOKs7kWh663vp+RoEhplgRWjBQAObwRKAAZgBOEDxIZCA4EEhWtqXlSADM1DVIAEyRII0ViC7VtYgd7jpeINwAKlgMsMDcuXQ8fKEWVjbFZX1V7Yit2p5sUzNzC0srMnIUEkA
\begin{tikzcd}
(M\mathcal{C})^i \arrow[d] \arrow[r, "f^i"] 
& (N\mathcal{D})^{i} \arrow[d] \\
(M\mathcal{C})^{i+1} \arrow[r, "f^{i+1}"]
& (N\mathcal{D})^{i+1}.        
\end{tikzcd}
    $$

\end{proposition}
\begin{proof}
    %Denote with $d^i$ and $\partial^i$ as the differential of $M\mathcal{C}$ and $N\mathcal{D}$ respectively.
    Denote $d^i$ as the differential and $U^{i}$ as the set of critical cells of degree $i$ for complex $M\mathcal{C}$. Similarly denote $\partial^i$ and $V^{i}$ for $N\mathcal{D}$ and let $a\in U^i$. A direct calculation shows
    \begin{align*}
        f^{i+1}d^i \iota_a %&=f^{i+1} \sum_{c\in U^{i+1}} \iota_{F(c)} \sum_{p\in \operatorname{Mor}(a,c) }(-1)^{r(p)} R(\mathcal{C},M)p \\
        &= \sum_{c\in U^{i+1}} \iota_{F(c)} \sum_{p\in \operatorname{Mor}(a,c) }(-1)^{r(p)} \gamma_c (R_{\mathcal{C}}p) \\
        &= \sum_{c\in U^{i+1}} \iota_{F(c)} \sum_{p\in \operatorname{Mor}(a,c) }(-1)^{r(Fp)} (R_{\mathcal{D}}Fp)\gamma_a \\
        &= \sum_{c'\in V^{i+1}} \iota_{c'} \sum_{p'\in \operatorname{Mor}(F(a),c') }(-1)^{r(p')} (R_{\mathcal{D}}p')\gamma_a \\
        &=\partial^{i+1} f^i \iota_a
    \end{align*}
    which sufficies for $f^{i+1}d^i=\partial^{i+1} f^i$.    
\end{proof}

\subsection{Khovanov complexes of tangles}

We will be using Bar-Natan's dotted cobordism formulation of Khovanov homology; see Section 11.2 of  \cite{BarNatanKhovanovTangles}. 
Bar-Natan's version is chosen, so that we can do manipulations for complexes of smaller pieces of links, tangles. From Bar-Natan's theory, we choose the dotted cobordism one, since it brings about a simple \term{delooping isomorphism} $\Psi$ which enables us the concrete use of discrete Morse theory. This theoretical framework, compared to the original theory of Khovanov \cite{khovanov1999categorification}, allows for more flexible composition, making it possible to obtain results about larger classes of links.

%%%%
%\textbf{Uusi alku}
Let us write down our conventions for the Khovanov complexes of tangles following \cite{BarNatanKhovanovTangles} and starting with the definition of categories $\Cob(2b)$ where $b\in \mathbb Z_{\geq 0}$. An object $o$ of $\Cob(2b)$ is a compact 1-manifold with $2b$ boundary points inside a closed disk $\mathbb D^2$. The boundary points of $o$ are contained in the boundary of the disk $\mathbb S^1 \subset \mathbb D^2$ and there is an additional marked point $(*)$ on $\mathbb S^1$. Moreover, $o$ is considered up to an orietatation preserving diffeomorphism on $\mathbb D^2$.

The set of morphisms $\operatorname{Hom}_{\Cob(2b)}(o_1,o_2)$ is a $\mathbb Z$-module spanned by dotted, up-to-isotopy cobordisms from $o_1$ to $o_2$ inside the cylinder $\mathbb D^2 \times [0,1]$. We require that in the boundary $\mathbb S^1 \times [0,1]$ these spanning cobordisms are constant, so that $n$:th boundary point of $o_1$, clockwise from $(*)$, will be connected by a path in $\mathbb S^1 \times [0,1]$ (or an equivalence class of paths to be precise) to the $n$:th boundary point of $o_2$, clockwise from $(*)$. 
The local relations of $\operatorname{Hom}_{\Cob(2b)}(o_1,o_2)$ are portrayed in Figure \ref{Relations of dotted cobordisms}. 
We can upgrade $\Cob(2b)$ into an additive category $\Mat (\Cob(2b))$, whose objects are formal direct sums of objects of $\Cob (2b)$. The morphisms of $\Mat (\Cob(2b))$ are matrices, the matrix elements of which are morphisms of $\Cob(2b)$. The category $\Mat(\Cob(2b))$ is also graded; a cobordism $f$ admits 
$\deg(f)=\chi(f)-b -2\cdot\#\{\text{dots in $f$}\} $
which we luckily do not need to worry about since all of the morphisms in question will be of degree 0.

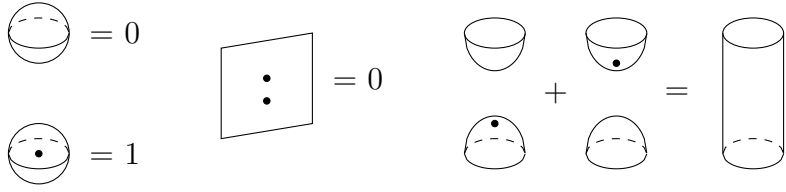
\begin{figure}[ht]
    \centering    
\begin{tikzpicture}[scale=0.4]
  
  \def\xradius{1};
  \def\yradius{0.5};
  \def\ballsxcoor{-15}
  \def\rectanglexcoor{-9}

%\draw (-2,0) circle(\xradius)
      \draw (\ballsxcoor,4) ellipse (\xradius cm and \xradius cm);
    \draw[dashed] plot[domain=-\xradius:\xradius] ({\x +\ballsxcoor},{4+1*\yradius * sqrt(1 - (\x / \xradius)^2)});
  \draw plot[domain=-\xradius:\xradius] ({\x +\ballsxcoor},{4-1*\yradius * sqrt(1 - (\x / \xradius)^2)});
\draw (\ballsxcoor+2.5,4) node {\large = 0};

      \draw (\ballsxcoor,0) ellipse (\xradius cm and \xradius cm);
    \draw[dashed] plot[domain=-\xradius:\xradius] ({\x +\ballsxcoor},{1*\yradius * sqrt(1 - (\x / \xradius)^2)});
  \draw plot[domain=-\xradius:\xradius] ({\x +\ballsxcoor},{-1*\yradius * sqrt(1 - (\x / \xradius)^2)});
\draw (\ballsxcoor+2.5,0) node {\large = 1};

  % Define the four points
  \coordinate (A) at (\rectanglexcoor,0.5);
  \coordinate (B) at (\rectanglexcoor+3,1);
  \coordinate (C) at (\rectanglexcoor+3,4);
  \coordinate (D) at (\rectanglexcoor,3.5);
  
  % Draw the closed loop
  \draw (A) -- (B) -- (C) -- (D) -- cycle;

    \filldraw (\rectanglexcoor+1.5,1.75) circle (3pt);
    \filldraw (\rectanglexcoor+1.5,2.5) circle (3pt);
    \draw (\rectanglexcoor+4.5,2.5) node {\large = 0};

    \filldraw (\ballsxcoor,0) circle (3pt);

  \draw (0,4) ellipse (\xradius cm and \yradius cm);
  \draw (4,4) ellipse (\xradius cm and \yradius cm);

\draw[dashed] plot[domain=-\xradius:\xradius] ({\x},{1*\yradius * sqrt(1 - (\x / \xradius)^2)});
\draw plot[domain=-\xradius:\xradius] ({\x},{-1*\yradius * sqrt(1 - (\x / \xradius)^2)});

  \draw plot[domain=-\xradius:\xradius] ({\x},{4 - 2.5*\yradius * sqrt(1 - (\x / \xradius)^2)});
  \draw plot[domain=-\xradius:\xradius] ({\x + 4},{4 - 2.5*\yradius * sqrt(1 - (\x / \xradius)^2)});

  \draw plot[domain=-\xradius:\xradius] ({\x},{2.5*\yradius * sqrt(1 - (\x / \xradius)^2)});
  \draw plot[domain=-\xradius:\xradius] ({\x + 4},{2.5*\yradius * sqrt(1 - (\x / \xradius)^2)});
  \draw[dashed] plot[domain=-\xradius:\xradius] ({\x + 4},{1*\yradius * sqrt(1 - (\x / \xradius)^2)});
  \draw plot[domain=-\xradius:\xradius] ({\x + 4},{-1*\yradius * sqrt(1 - (\x / \xradius)^2)});
  
  \draw (2,2) node {\large +};

    \draw (6,2) node {\large =};

    \draw (8.5,4) ellipse (\xradius cm and \yradius cm);
  \draw[dashed] plot[domain=-\xradius:\xradius] ({\x + 8.5},{1*\yradius * sqrt(1 - (\x / \xradius)^2)});
  \draw plot[domain=-\xradius:\xradius] ({\x + 8.5},{-1*\yradius * sqrt(1 - (\x / \xradius)^2)});

    \draw (7.5,0) -- (7.5,4);
    \draw (9.5,0) -- (9.5,4);
    
    \filldraw (4,3) circle (3pt);
    \filldraw (0,1) circle (3pt);
  
\end{tikzpicture}
    
    \caption{Local relations of dotted cobordisms.}
    \label{Relations of dotted cobordisms}
\end{figure}

The \term{Khovanov complex} $\llbracket T \rrbracket$ of a tangle diagram $T$ with $2b$ boundary points is a complex over the category  $\Mat (\Cob(2b))$, see Figure \ref{Figure: example of a Khovanov complex} for an example. A \term{smoothing} of $T$ is an object of $\Cob(2b)$ obtained by replacing every \term{crossing} \begin{tikzpicture}[baseline={(0,0.05)},scale=0.3]
    \negCrossing{0,0} 
\end{tikzpicture} with either \begin{tikzpicture}[baseline={(0,0.05)},scale=0.3]
    \lineRightBendUp{0,0}
    \lineRightBendDown{0,1}
\end{tikzpicture} or \begin{tikzpicture}[baseline={(0,0.05)},scale=0.3]
    \lineUpBendRight{0,0}
    \lineUpBendLeft{1,0}
\end{tikzpicture}. If $T$ has $n$ crossings, $n_+$ of which are \term{positive} and $n_-$ \term{negative}, there are $2^n$ smoothings of $T$; see Figure \ref{Positive negative crossings zero one smoothings} for conventions and Figure \ref{Enhanced words diagram} for an example. There are also $2^n$ elements in the set of \term{$01$-words} of length $n$ and given an order on the crossings, we obtain a correspondence 
$$
\{\text{smoothings of }T\} \longleftrightarrow \{0,1\}^n.
$$
The $(k-n_-)$:th chain space of the complex $\llbracket T\rrbracket$ is the formal direct sum of all of the smoothings corresponding to those $01$-words with $k$ ones in them.  Additionally, the chain space at level $k-n_-$ is also shifted with $\{ k+n_+-2n_-\}$.

\begin{figure}[ht]
  \centering
  \hspace{8 mm}
  \begin{subfigure}[b]{0.25\textwidth}
        \centering
        \begin{tikzpicture}%[scale=0.3]
    \begin{scope}[shift={(-1,0)}, scale=0.5]
        \negCrossing{0,0}
        \negCrossingWArrowsRight{1,0}
        \draw[dashed] (1,0.5) circle (1.118);
        \node at (1,0.5) { $\Bigg\llbracket$ \hspace{1.2cm} $\Bigg\rrbracket$};
        
    \node at (1,1.42) {$*$};
    \end{scope}

    \node at (0.75,0.25) { $=$};

    \begin{scope}[shift={(2,0)}, scale=0.5]
        %\posCrossing{0,0}
        %\posCrossingWArrowsRight{1,0}
        \lineRightBendUp{0,0}
        \lineRightBendDown{0,1}
        \lineRightBendUp{1,0}
        \lineRightBendDown{1,1}
        
        \draw[dashed] (1,0.5) circle (1.118);
        \node at (1,2) {$00$};
        
    \node at (1,1.42) {$*$};
    \end{scope}

    \begin{scope}[shift={(4,0.75)}, scale=0.5]
        %\posCrossing{0,0}
        %\posCrossingWArrowsRight{1,0}
        \lineUpBendRight{1,0}
        \lineUpBendLeft{2,0}
        \lineRightBendUp{0,0}
        \lineRightBendDown{0,1}
        
        \draw[dashed] (1,0.5) circle (1.118);
        \node at (1,2) {$01$};
        
    \node at (1,1.42) {$*$};
        
        %\node at (1,0.5) { $\Bigg\llbracket$ \hspace{1.2cm} $\Bigg\rrbracket$};
    \end{scope}

    \begin{scope}[shift={(4,-0.75)}, scale=0.5]
        %\posCrossing{0,0}
        %\posCrossingWArrowsRight{1,0}
        
        \lineUpBendRight{0,0}
        \lineUpBendLeft{1,0}
        \lineRightBendUp{1,0}
        \lineRightBendDown{1,1}
        
        \draw[dashed] (1,0.5) circle (1.118);
        \node at (1,-1) {$10$};
        
    \node at (1,1.42) {$*$};
        %\node at (1,0.5) { $\Bigg\llbracket$ \hspace{1.2cm} $\Bigg\rrbracket$};
    \end{scope}

    \begin{scope}[shift={(6,0)}, scale=0.5]
        %\posCrossing{0,0}
        %\posCrossingWArrowsRight{1,0}
        \lineUpBendRight{0,0}
        \lineUpBendLeft{1,0}       
        \lineUpBendRight{1,0}
        \lineUpBendLeft{2,0}
        \node at (1,2) {$11$};
        \draw[dashed] (1,0.5) circle (1.118);
        
    \node at (1,1.42) {$*$};
        %\node at (1,0.5) { $\Bigg\llbracket$ \hspace{1.2cm} $\Bigg\rrbracket$};
    \end{scope}

    \node at (4.5,0.25) {$\oplus$};

    \draw[->] (3.2,0.5) -- (3.8, 0.8)  node[midway, above] {$f_1$};
    \draw[->] (3.2,0) -- (3.8, -0.3)  node[midway, below] {$f_2$};
    \draw[->] (5.2,0.8) -- (5.8, 0.5)  node[midway, above] {$f_3$};
    \draw[->] (5.2,-0.3) -- (5.8, 0)  node[midway, below] {$-f_4$};

    \node at (9,0.25) {$f_4=$};

\end{tikzpicture}
    %\caption{Figure 1}
    %\label{fig:figure1}
  \end{subfigure}
  \hfill
  \begin{subfigure}[b]{0.25\textwidth}
  \centering
    \includegraphics[scale=0.6]{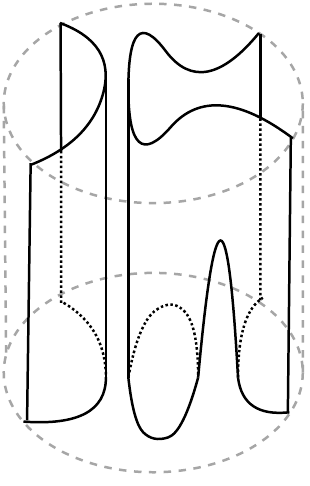}
  \end{subfigure}
  \caption{Khovanov complex of $\sigma_1^{-2}$ and one of its morphisms (gradings omitted). Contrary to our usual conventions, the braid $\sigma_1^{-2}$ is rotated 90 degrees clockwise so that the cobordism $f_4$ is easier to perceive.}
  \label{Figure: example of a Khovanov complex}
\end{figure}

\begin{figure}[t]
    \centering
    \begin{tikzpicture}
        \posCrossingWArrows{0,0}
        
        \negCrossingWArrows{2,0}

        \negCrossing{7,0}
        \lineRightBendUp{5,0}
        \lineRightBendDown{5,1}

        \lineUpBendRight{9,0}
        \lineUpBendLeft{10,0}

    \node at (0.5,1.25) {$+$};
    
    \node at (2.5,1.25) {$-$};
    
    \node at (5.5,1.25) {$0$};

    \node at (9.5,1.25) {$1$};
\draw[->,dashed] (6.8, 0.5) to (6.2, 0.5);

\draw[->,dashed] (8.2, 0.5) to (8.8, 0.5);
    \end{tikzpicture}    
    \caption{Sign and smoothing conventions for crossings.}
    \label{Positive negative crossings zero one smoothings}
\end{figure}
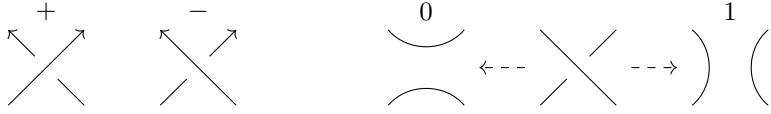

The smoothings and non-zero matrix elements form the 1-skeleton of an $n$-dimensional hypercube.
%Morphisms of the Khovanov complex are matrices whose elements are zero morphisms and signed cobordisms modulo relations presented in Figure \ref{Relations of dotted cobordisms}. 
The non-zero matrix elements are cobordisms, which occur between $01$-words when the words differ at exactly one index. Such a cobordism forms a saddle at the crossing of the given index and are identity elsewhere, again see Figure \ref{Figure: example of a Khovanov complex}. %Moreover the cobordisms are considered modulo relations in Figure \ref{Relations of dotted cobordisms} and they are given signs.  
The sign of this cobordism is $(-1)^k$ where $k$ is the number of ones before $r$ in the domain word (or codomain word) and $r$ is the index in which the words differ. The situation resembles  simplicial complexes -- different orderings lead to different (but isomorphic) cochain complexes. % to define a chain complex from a simplicial complex, one is formally required to order the vertices even if the complexes are isomorphic   

\subsection{Enhanced words of a braid diagram}\label{Section: enhanced words of a braid diagram}

By an \term{$n$-braid word} we mean a finite sequence of characters $\sigma_1,\dots, \sigma_{n-1}, \allowbreak \sigma_1^{-1},\dots, \sigma_{n-1}^{-1}$. A braid word defines a braid diagram for example in the case of 3-braids by assigning 
$$
\sigma_1\mapsto \tikz[baseline={(0,0.05)},scale=0.3]{
    \negCrossing{0,0} \lineUp{2,0}
}
\qquad
\sigma_2\mapsto\begin{tikzpicture}[baseline={(0,0.05)},scale=0.3]
    \negCrossing{1,0} \lineUp{0,0}
\end{tikzpicture}
\qquad
\sigma_1^{-1}\mapsto\begin{tikzpicture}[baseline={(0,0.05)},scale=0.3]
    \posCrossing{0,0} \lineUp{2,0}
\end{tikzpicture}
\qquad
\sigma_2^{-1} \mapsto \begin{tikzpicture}[baseline={(0,0.05)},scale=0.3]
    \posCrossing{1,0} \lineUp{0,0}
\end{tikzpicture},
$$
stacking these pictures on top of each other, orienting from bottom to top and enclosing in a circle, all of which is illustrated in Figure \ref{Braid closure figure}.  When creating the Khovanov complex of a braid word, the crossings are ordered from bottom to top and one should note that if the braid words $T$ and $T'$ are the same as elements of the braid group, the Reidemeister invariance of Khovanov complexes guarantees $\llbracket T \rrbracket \simeq \llbracket T' \rrbracket$.

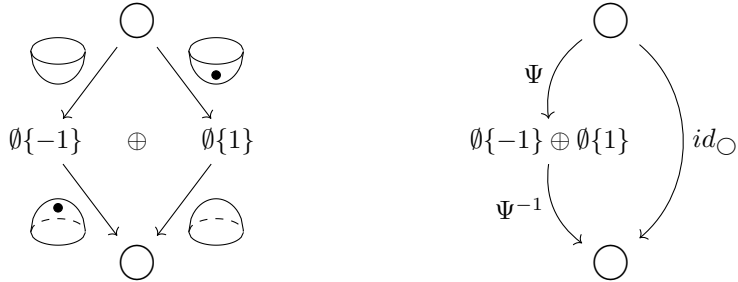
\begin{figure}[ht]
  \centering
  \begin{subfigure}[b]{0.4\textwidth}
  \centering
    \begin{tikzpicture}[scale=0.8]

  % Define the vertices
  \node (A) at (0, 4) {\Large$\bigcirc$};
\node (B) at (1.5, 2) { $\emptyset \{1\}$};
  \node (C) at (-1.5, 2) { $\emptyset \{-1\}$};
  \node (D) at (0, 0) {\Large$\bigcirc$};

\node (E) at (0, 2) {$\oplus$};

  % Draw the arrows
  \draw[->] (A) -- (B);
  \draw[->] (A) -- (C);
  \draw[->] (B) -- (D);
  \draw[->] (C) -- (D);

%\draw (0,4.7) ellipse (0.6,0.3);

 %\draw (0,4.4) ellipse (0.7cm and 0.35cm);

 %\draw (0,-0.5) ellipse (0.7cm and 0.35cm);

%\draw (0,4.7) circle (0.6);

%\draw (0,-0.7) circle (0.6);

  % Define the x and y radii of the ellipses
  \def\xradius{0.45};
  \def\yradius{0.22};

\def\a{1.3}
\def\b{3.5}

\def\c{-1.3}
\def\d{3.5}

\def\e{1.3}
\def\f{0.5}

\def\g{-1.3}
\def\h{0.5}

\draw plot[domain=-\xradius:\xradius] ({\a+\x},{\b+1*\yradius * sqrt(1 - (\x / \xradius)^2)});
\draw plot[domain=-\xradius:\xradius] ({\a+\x},{\b-1*\yradius * sqrt(1 - (\x / \xradius)^2)});
\draw plot[domain=-\xradius:\xradius] ({\a+\x},{\b-2.5*\yradius * sqrt(1 - (\x / \xradius)^2)});
\filldraw (\a,\b-0.4) circle (2pt);

\draw plot[domain=-\xradius:\xradius] ({\c+\x},{\d+1*\yradius * sqrt(1 - (\x / \xradius)^2)});
\draw plot[domain=-\xradius:\xradius] ({\c+\x},{\d-1*\yradius * sqrt(1 - (\x / \xradius)^2)});
\draw plot[domain=-\xradius:\xradius] ({\c+\x},{\d-2.5*\yradius * sqrt(1 - (\x / \xradius)^2)});
%\filldraw (\c,\d) circle (3pt);

\draw[dashed] plot[domain=-\xradius:\xradius] ({\e+\x},{\f+1*\yradius * sqrt(1 - (\x / \xradius)^2)});
\draw plot[domain=-\xradius:\xradius] ({\e+\x},{\f-1*\yradius * sqrt(1 - (\x / \xradius)^2)});
\draw plot[domain=-\xradius:\xradius] ({\e+\x},{\f+2.5*\yradius * sqrt(1 - (\x / \xradius)^2)});
%\filldraw (\e,\f) circle (3pt);

\draw[dashed] plot[domain=-\xradius:\xradius] ({\g+\x},{\h+1*\yradius * sqrt(1 - (\x / \xradius)^2)});
\draw plot[domain=-\xradius:\xradius] ({\g+\x},{\h-1*\yradius * sqrt(1 - (\x / \xradius)^2)});
\draw plot[domain=-\xradius:\xradius] ({\g+\x},{\h+2.5*\yradius * sqrt(1 - (\x / \xradius)^2)});
\filldraw (\g,\h+0.4) circle (2pt);

\end{tikzpicture}

  \end{subfigure}
  %\hfill
  \begin{subfigure}[b]{0.4\textwidth}
  \centering
    \begin{tikzpicture}[scale=0.8][node distance=2cm]
  % Nodes
  \node (F) at (1,2) {\Large$\bigcirc$};
  \node (G) at (0,0) {$\emptyset\{-1\}\oplus \emptyset\{1\}$};
  \node (H) at (1,-2) {\Large$\bigcirc$};
  
  % Arrows
  \draw[->] (F) to[bend right=30] node[left] {$\Psi$} (G);
  \draw[->] (G) to[bend right=30] node[left] {$\Psi^{-1}$} (H);
  \draw[->] (F) to[bend left=50] node[right] {$\mathop{{id}}_{\bigcirc}$} (H);

  %quick and dirty way to create space between figures
  %\node(asd) at (-2.5,0) {};
\end{tikzpicture}

  \end{subfigure}
    \caption{Delooping isomorphism $\Psi$ for a single circle and its inverse $\Psi^{-1}$ as a matrices of cobordism (left) and diagrammatically (right). }
    \label{Delooping isomorphism}
\end{figure}

An object $o\in \Cob(2b)$ can be written as a disjoint sum 
$$
o=\bigsqcup_{p\in P} p \sqcup \bigsqcup_{c\in C} c
$$
where $P$ is a $b$-sized set of paths and $C$ is a finite set of closed loops. By considering $o$ as an object in $\Mat(\Cob(2b))$ we can define the delooping isomorphism
$$
\Psi\colon o\to \bigoplus_{l\in L}\bigg(\bigsqcup_{p\in P} p \bigg)\left\{\int l \, d \mu  \right\}
$$
where $L$ is the set of functions $C\to \{-1,1\}$, $\mu$ is the counting measure and $\Psi$ is depicted locally in Figure \ref{Delooping isomorphism}.  By some careful notation, we can bake the information of $l$ into the $01$-words. If $l(c)=-1$ for a circle $c$, we put a superscripts $\bx$ on the character of the 01-word corresponding to the unique crossing that splits $c$ from the furthest below. Correspondingly, we put superscript $\bone$ if $l(c)=1$. For characters in the 01-word, which do not correspond to crossings splitting circles from the furthest below, we put superscript $\bminus$. All of this can be better absorbed by examining Figure \ref{Enhanced words diagram}, where we have  correspondences 
$$
\Big(l(  \bigcirc
)=1  \Big) \longleftrightarrow  1^{\bminus}0^{\boldsymbol 1}1^{\bminus}0^{\bminus}0^{\bminus}1^{\bminus}  \quad \text{ and } \quad \Big( l( \bigcirc) =-1 \Big) \longleftrightarrow  1^{\bminus}0^{\boldsymbol x}1^{\bminus}0^{\bminus}0^{\bminus}1^{\bminus}.
$$

\begin{figure}
    \centering\begin{tikzpicture}[scale=0.5]

    \pgfmathsetmacro{\sideToCircle}{0.5}%{0.358}

    \begin{scope} [shift={(0,0)}]

    \node at (1,-1.5) {$\sigma_1\sigma_2\sigma_1\sigma_2\sigma_1\sigma_2$};    
    %\draw[dashed]  (1, 3) circle (3.5);
    \negCrossing{0,0}
    \lineUp{2,0}
    \negCrossing{1,1}
    \lineUp{0,1}
    \negCrossing{0,2}
    \lineUp{2,2}
    \negCrossing{1,3}
    \lineUp{0,3}
    \negCrossing{0,4}
    \lineUp{2,4}
    \negCrossing{1,5}
    \lineUp{0,5}

    \draw[->] (0,6) -- (0,6+\sideToCircle);
    \draw[->] (2,6) -- (2,6+\sideToCircle);
    \draw[->] (0,-\sideToCircle) -- (0,0);
    \draw[->] (2,-\sideToCircle) -- (2,0);
    \draw[->] (1,6) -- (1,6.5);
    \draw[->] (1,-0.5) -- (1,0);
    
    \end{scope}

    \begin{scope} [shift={(6,0)}]
        \node at (1,-1.5) {$101001$};    
    %\draw[dashed]  (1, 3) circle (3.5);
    
    \lineUpBendRight{0,0}
    \lineUpBendLeft{1,0}    
    \lineUp{2,0}
    \lineRightBendUp{1,1}
    \lineRightBendDown{1,2}
    \lineUp{0,1}
    \lineUpBendRight{0,2}
    \lineUpBendLeft{1,2}
    \lineUp{2,2}
    \lineRightBendUp{1,3}
    \lineRightBendDown{1,4}
    \lineUp{0,3}
    \lineRightBendUp{0,4}
    \lineRightBendDown{0,5}
    \lineUp{2,4}
    \lineUpBendRight{1,5}
    \lineUpBendLeft{2,5}
    \lineUp{0,5}

    \draw (0,6) -- (0,6+\sideToCircle);
    \draw (2,6) -- (2,6+\sideToCircle);
    \draw (0,-\sideToCircle) -- (0,0);
    \draw (2,-\sideToCircle) -- (2,0);
    \draw (1,6) -- (1,6.5);
    \draw (1,-0.5) -- (1,0);
    
    \end{scope}

    \begin{scope} [shift={(12,0)}]
        \node at (1,-1) {\small$1^{\bminus}0^{\boldsymbol 1}1^{\bminus}0^{\bminus}0^{\bminus}1^{\bminus}$};   
    %\draw[dashed]  (1, 3) circle (3.5);
    
    \lineUpBendRight{0,0}
    \lineUpBendLeft{1,0}    
    \lineUp{2,0}
    \lineRightBendUp{1,1}

    \lineUp{0,1}
    \lineUpBendRight{0,2}

    \lineRightBendDown{1,4}
    \lineUp{0,3}
    \lineRightBendUp{0,4}
    \lineRightBendDown{0,5}
    \lineUp{2,4}
    \lineUpBendRight{1,5}
    \lineUpBendLeft{2,5}
    \lineUp{0,5}

    \draw (0,6) -- (0,6+\sideToCircle);
    \draw (2,6) -- (2,6+\sideToCircle);
    \draw (0,-\sideToCircle) -- (0,0);
    \draw (2,-\sideToCircle) -- (2,0);
    \draw (1,6) -- (1,6.5);
    \draw (1,-0.5) -- (1,0);
    
    \end{scope}
            
    \begin{scope} [shift={(16,0)}]
        \node at (1.5,-1.6) {\small$ 1^{\bminus}0^{\boldsymbol x}1^{\bminus}0^{\bminus}0^{\bminus}1^{\bminus}$};
        
    %\draw[dashed]  (1, 3) circle (3.5);
    
    \lineUpBendRight{0,0}
    \lineUpBendLeft{1,0}    
    \lineUp{2,0}
    \lineRightBendUp{1,1}

    \lineUp{0,1}
    \lineUpBendRight{0,2}

    \lineRightBendDown{1,4}
    \lineUp{0,3}
    \lineRightBendUp{0,4}
    \lineRightBendDown{0,5}
    \lineUp{2,4}
    \lineUpBendRight{1,5}
    \lineUpBendLeft{2,5}
    \lineUp{0,5}

    \draw (0,6) -- (0,6+\sideToCircle);
    \draw (2,6) -- (2,6+\sideToCircle);
    \draw (0,-\sideToCircle) -- (0,0);
    \draw (2,-\sideToCircle) -- (2,0);
    \draw (1,6) -- (1,6.5);
    \draw (1,-0.5) -- (1,0);
    
    \end{scope}

    \node at (15,3) {$\oplus$};

    \draw[->, decorate, decoration={snake}] (3,3) -- (5.5,3); 
    
    \draw[->] (9,3) -- (11.5,3) node[midway, above] {$\Psi$};

    %\draw[decorate, decoration={snake, segment length=5mm, amplitude=15mm}] (0,0) -- (10,10);
    
\end{tikzpicture}
    \caption{A braid diagram of 6 crossings (6 negative, 0 positive), one of the smoothings from $\{0,1\}^6$ and the two enhanced words corresponding to that smoothing. The dashed circles bounding the braids are omitted for clarity.}
    \label{Enhanced words diagram}
\end{figure}
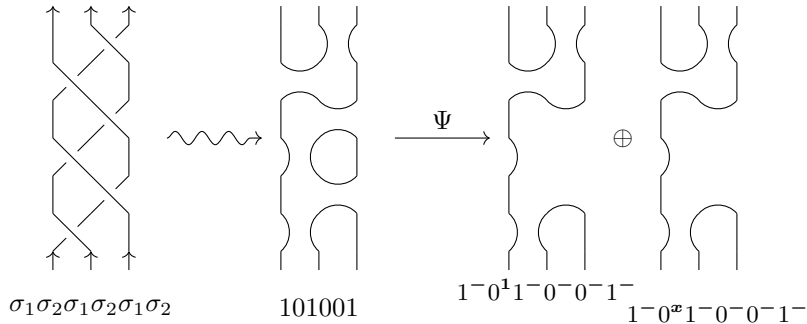

For a braid diagram $T$ with $k$ crossings  we call these $01$-tuples of length $k$ with superscripts \term{enhanced words} and mark the set of them with $\Enh(T)$. As sets, we can see the inclusion $\Enh(T)\subset Y^{k}$ where $Y=\{0^{\bx}, 0^{\bone},\allowbreak 0^{\bminus}, \allowbreak 1^{\bx}, 1^{\bone}, 1^{\bminus}\}$ is the set of symbols.
%We denote the set of symbols $Y=\{0^{\bx}, 0^{\bone},\allowbreak 0^{\bminus}, \allowbreak 1^{\bx}, 1^{\bone}, 1^{\bminus}\}$ resulting in $\Enh(T)\subset Y^{k}$. 
We can further make this set $\Enh(T)$ into a based bigraded complex $(\llbracket T \rrbracket_{\Enh}, \Tilde{d})$ by endowing the words with appropriate gradings and  by pulling  back the differentials $d$ of $\llbracket T \rrbracket$ via delooping isomorphism. In other words, we define $(\llbracket T \rrbracket_{\Enh})^i=\Psi (\llbracket T \rrbracket^i)$ and $\Tilde{d}=\Psi d \Psi^{-1}$ immediately yielding $\llbracket T \rrbracket \cong\llbracket T \rrbracket_{\Enh}$.

When employing discrete Morse theory, we need to investigate graphs  $G(\llbracket T \rrbracket_{\Enh}, M)$ and know which matrix elements of $\llbracket T \rrbracket_{\Enh}$ are zero morphisms and which are isomorphisms. Let $a$ and $b$ be enhanced words of $\llbracket T \rrbracket_{\Enh}$ and let $\Tilde{f}\colon a \to b$ be a matrix element pulled back from cobordism $f$ of $\llbracket T \rrbracket$. If $f$ merges a circle marked with superscript $\bone$ in $\llbracket T \rrbracket_{\Enh}$ or splits into a circle which is marked with $\bx$ in the codomain and if  $\Tilde{f}$  keeps otherwise  the words and superscripts intact, then $\Tilde{f}=\pm \operatorname{id}$. If $f$  merges two circles marked with $\bx$ or changes a marking on a circle not affected by the cobordism, then $\Tilde{f}$ is the zero morphism.    

To further analyze the structure of graphs $G(\llbracket T \rrbracket_{\Enh}, M)$ we define two combinatorial integer valued functions, $O$ and $L$. The function $O$ takes in an enhanced word $a$ and spits out the number of non-superscript 1:s before the first 0 in the word. The function $L$ takes in an edge $a\nearrow b$ (or $c\searrow d)$ and gives out the unique index in the enhanced words, for which $0$ changes to $1$ (or $1$ changes to $0$).

\section{Discrete Morse theory for 2-torus braids} \label{Section: DMT for T2}

In this section we investigate braid diagrams $\sigma_1^m$, $m\geq 0$, see Figure \ref{T2 braid figure} and their  Khovanov complexes $\llbracket\sigma_1^m\rrbracket\in \Kom(\Mat(\Cob(4)))$. The concluding result of this section is Proposition \ref{commutative squares of T2} which relates the Morse complexes of $\llbracket \sigma_1^m\rrbracket_{\Enh}$ and $\llbracket \sigma_1^{m+2}\rrbracket_{\Enh}$. A more explicit description of these complexes has been obtained earlier in \cite{KhovanovHomologyOfPrezelsRevisited}. With that in mind, the point of including this section is that we get to lay down our notation, but more importantly that we get to practice the method of Discrete Morse theory in a less combinatorially tedious setting compared to complexes $\llbracket (\sigma_1 \sigma_2)^k\rrbracket_{\Enh}$. Apart from explicit classification of paths, the complexes $\llbracket (\sigma_1 \sigma_2)^{k}\rrbracket_{\Enh}$ will be studied identically to the way $\llbracket \sigma_1^m\rrbracket_{\Enh}$ are examined.
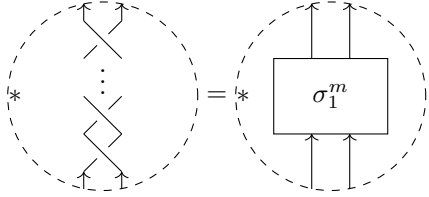
\begin{figure}[ht]
    \centering

\begin{tikzpicture}[scale=0.5]
  \begin{scope}[shift={(0,0)}]
        
  \draw[dashed] (0.5,0) circle (2.5);

\negCrossing{0,-2};

\negCrossing{0,-1};

\negCrossing{0,1};
   \node at (0.5,0.6) {$\vdots$};

  % Arrows from the origin to nodes
  \draw[->] (1,2) -- (1,2.46);
  \draw[->] (0,2) -- (0,2.46);
  \draw[<-] (1,-2) -- (1,-2.46);
  \draw[<-] (0,-2) -- (0,-2.46);
  %\draw[->] (0,0) -- (-1,3.87);
  %\draw[->] (0,0) -- (1,-3.87);
  %\draw[->] (0,0) -- (-1,-3.87);
  
  \end{scope}

  \begin{scope}[shift={(6,0)}]
        
  \draw[dashed] (0.5,0) circle (2.5);

  \draw (-1, -1) rectangle (2, 1);

    \node at (0.5,0) {$\sigma_1^m$};

  % Arrows from the origin to nodes
  \draw[->] (1,1) -- (1,2.46);
  \draw[->] (0,1) -- (0,2.46);
  \draw[<-] (1,-1) -- (1,-2.46);
  \draw[<-] (0,-1) -- (0,-2.46);
  %\draw[->] (0,0) -- (-1,3.87);
  %\draw[->] (0,0) -- (1,-3.87);
  %\draw[->] (0,0) -- (-1,-3.87);
  
  \end{scope}

    \node at (3.5,0) {$=$};
    
    \node at (-1.8,0) {$*$};
    \node at (4.2,0) {$*$};
  
\end{tikzpicture}
    \caption{Braid diagram of $\sigma_1^m$ with $m\geq 0$ drawn in two ways.  }
    \label{T2 braid figure}
\end{figure}

A smoothing of braid $\sigma_1^m$ contains closed loops between subsequent 0-smoothed crossings (which may have any amount of 1-smoothed crossings between them). Thus, the base of complex $\llbracket \sigma_1^m\rrbracket_{\Enh}$ consists of enhanced words of length $m$ with $\bone$ or $\bx$ superscripts over every 0 except for the last 0 and $\bminus$ superscripts elsewhere. 
In order to ease our notation, we omit $\bminus$ superscripts and only mark $\bone$ and $\bx$ superscripts. We define $M$, a set of matrix elements of $\llbracket \sigma_1^m \rrbracket_ {\Enh}$, by\footnote{Although we work with the explicit definition, this set $M$ could also be obtained by scanning through the braid and collecting isomorphisms between vertices, which have not been matched yet.} 
$$
M=\left\{z\nearrow x  \ \middle\vert  \begin{array}{c}
       z=1\dots1 0^{\bx}0^{\bx} \dots 0^{\bx} 0^{\bone} 0^s y_j \dots y_m \\ 
       x=1\dots1 0^{\bx}0^{\bx} \dots 0^{\bx} 0^{s} 1 \  y_j \dots y_m \\
       3\leq j \leq m+1,\ y_j,\dots, y_m \in Y,\ s\in \{\bone, \bx, \bminus \} 
\end{array}\right\}
$$
and observe that $M$ is finite, it contains only isomorphims and  it induces a partial matching  on $G=G(\llbracket\sigma_1^m \rrbracket_{\Enh}, M)$. We still need to show that $G$ has no directed cycles and to that end we need to prove a lemma and recall some notation. We write $a\nearrow b$ and $c\searrow d$ for edges in $G$ to indicate that $\hdeg(a)+1=\hdeg(b)$ and $\hdeg(c)-1=\hdeg(d)$ where $\hdeg$ denotes the homological degree. The function $L$ takes in an edge of $G$ and outputs the changing index of the enhanced words and $O$ takes in an enhanced word and gives the number of 1:s at the start of that word.

%The unmatched cells are of the form $1\dots 1 0^{\bx} \dots ^{\bx} 0$, there is at most one at each homological degree.

\begin{lemma}\label{T2 acyclicity/paths helper}
    Let $x\searrow z\nearrow x'\searrow z'$ be a zig-zag path in $G(\llbracket\sigma_1^m \rrbracket_{\Enh}, M)$ and assume $L(x\searrow z) > O(z)+2$. Then for some superscript $s\in \{\bx,\bone, \bminus\}$ and symbols $y_j,\dots y_m \in Y$ we have
    \begin{align*}
        z&=1\dots 1 0^{\bx}0^{\bx}\dots 0^{\bx} 0^{\bone} 0^s y_j\dots y_m \\
        x'&=1\dots 1 0^{\bx}0^{\bx}\dots 0^{\bx} 1 \ \, 0^s y_j\dots y_m \\
        z'&=1\dots 1 0^{\bx}0^{\bx}\dots 0^{\bone} 0^{\bx} 0^s y_j\dots y_m 
    \end{align*}
    and $L(x\searrow z)-1=L(x'\searrow z')$. Consequently, let $v\searrow w \nearrow v'$ be a zig-zag path in $G(\llbracket\sigma_1^m \rrbracket_{\Enh}, M)$ and assume $L(v\searrow w) = O(v)+2$ and $v'$ is not matched upwards. Then $O(v)+1=O(v')$.
\end{lemma}
\begin{proof}
The fact that $z$ is of this form, follows directly from the definition of $M$. The claim follows from showing every case except $L(z\nearrow x')=L(x \searrow z)-1$ to be impossible.

    Case $O(z)+1=L(z\nearrow x')<L(x \searrow z)-1$. From the definition of $M$  one can see that it is impossible for $x'$ to be matched downwards. 

    Case $O(z)+1<L(z\nearrow x')<L(x \searrow z)-1$. These differentials merge two circles with $\bx$ superscripts which means that the cobordisms are zero morphisms. Hence such edges $z\nearrow x'$ do not exist in $G$.
    %These differentials merge two circles with $\bx$ superscripts. These cobordisms correspond to zero maps, which means that such arrows $z\to x'$ do not exist in $G$.

    Case $L(z\nearrow x')=L(x \searrow z)$. The corresponding arrow would be contained in $M$, so in $G$ there is no edge in this direction.    
    %This edge comes from a reversed isomorphism so in $G$ it has the wrong direction.

    Case $L(z\nearrow x')>L(x \searrow z)$. This would lead $x'$ to be matched upwards violating $x'\searrow z'$.

    The second claim is proven similarly.
\end{proof}

\begin{lemma}
    The graph $G(\llbracket\sigma_1^m \rrbracket_{\Enh}, M)$ has no directed cycles and thus $M$ is a Morse matching.% and therefore $M$ is a Morse matching.
\end{lemma}

\begin{proof}
    Recall that a binary relation is a \term{strict preorder} if it is irreflexive and transitive. We impose a strict preorder $\prec$ on $M$ by assigning  to arrows $(a\nearrow b),(a'\nearrow b')\in M$ the following:
$$
\begin{array}{cl}
   O(b)<O(b')  & \implies (a\nearrow b)\prec(a'\nearrow b') \\[2pt]
     \left(\begin{array}{c}
O(b)=O(b')        \\
 L(a\nearrow b)>L(a'\nearrow b')      
\end{array}  \right)& \implies (a\nearrow b)\prec(a'\nearrow b').
\end{array}
$$
Suppose that $x\searrow z \nearrow x' \searrow z'$ is a zig-zag path in $G(\llbracket\sigma_1^m \rrbracket_{\Enh}, M)$, so that $(z\nearrow x),(z' \nearrow x')\in M$. It follows from the definition of $M$, that $L(x\searrow z)\geq O(z)+2$. Thus, from  Lemma \ref{T2 acyclicity/paths helper} we obtain $(z\nearrow x) \prec(z' \nearrow x')$. Since $\prec$ is a strict preorder, we can deduce that $G$ has no directed cycles.
\end{proof}
 We have now shown that $M$ is a Morse matching and hence by Theorem \ref{Discrete Morse theory} the complex $\llbracket \sigma_1^m \rrbracket_{\Enh}$ is chain homotopic to the Morse complex $M\llbracket \sigma_1^m \rrbracket_{\Enh}$ shown in Figure \ref{Morse complex of T2}. %In particular one should note that in every homological degree of $G$ there is at most 1 unmatched vertex corresponding a sole summand in $M\llbracket \sigma_1^m \rrbracket_{\Enh}$. On the other hand the number of summands in homological levels of $\llbracket \sigma_1^m \rrbracket_{\Enh}$ grow exponentially as $m$ grows. 

\begin{figure}[ht]
    \centering
    \begin{tikzpicture}[scale=0.9]

    \begin{scope}[shift={(12.5,0)},rotate=90]
    \lineRightBendUp{0,0}
    \lineRightBendDown{0,1}
    \draw[dashed] (0.5, 0.5) circle (0.707);

    \node at (0.5,1.1) {$*$};
    \end{scope}

    \begin{scope}[shift={(2.5,0)}]
    \lineRightBendUp{0,0}
    \lineRightBendDown{0,1}
    \draw[dashed] (0.5, 0.5) circle (0.707);
    %\node at (0.5,1.8)[text width=3cm] {hdeg: $-m$ \hspace{1cm} shift: $\{-3m+1\}$};
    \node at (0.35,1.5) {\small$-m;-3m+1$};
    %\node        (B) [right of=A,fill=blue!25,text width=3cm]{This is a demonstration text for showing how line breaking works.};
    
    \node at (-0.1,0.5) {$*$};
    \end{scope}

    \begin{scope}[shift={(5,0)}]
    \lineRightBendUp{0,0}
    \lineRightBendDown{0,1}
    \draw[dashed] (0.5, 0.5) circle (0.707);
    \node at (0.65,1.5) {\small$-m+1;-3m+3$};
    
    \node at (-0.1,0.5) {$*$};
    \end{scope}

    \begin{scope}[shift={(9,0)}]
    \lineRightBendUp{0,0}
    \lineRightBendDown{0,1}
    \draw[dashed] (0.5, 0.5) circle (0.707);
    \node at (0.5,1.5) {\small$-1;-m-1$};
    
    \node at (-0.1,0.5) {$*$};
    \end{scope}

    \node(first0) at (1.25,0.5) {$0$};
    \node(middle dots) at (7.5,0.5) {$\cdots$};
    \node(last0) at (13.75,0.5) {$0$};
    \node at (12,1.5) {\small$0;-m$};
    
    %\node at (10,3.5) {$(\sigma_1\sigma_2)^k$};
    \draw[->] (first0) -- (2,0.5);
    \draw[->] (4,0.5) -- (4.5,0.5);
    \draw[->] (6.5,0.5) -- (middle dots);
    \draw[->] (middle dots) -- (8.5,0.5);
    \draw[->] (10.5,0.5) -- (11,0.5);
    \draw[->] (13,0.5) -- (last0);
    %\draw[->] (7,0.5) -- (7.5,0.5);
    %\draw[->] (8.5,0.5) -- (9,0.5);
    %\draw[->] (11.5,0.5) -- (12,0.5);

\end{tikzpicture}

    \caption{Morse complex $M\llbracket \sigma_1^m\rrbracket_{\Enh}$. The homological gradings and the internal grading shifts are written above the chain spaces and separated by semicolons.  }
    \label{Morse complex of T2}
\end{figure}
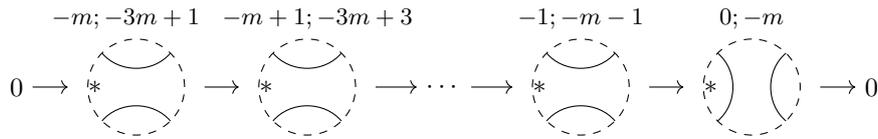

 \subsection{Adding a full twist to 2-torus braids}
 
 Next up, we shall investigate what happens to Morse complexes $M\llbracket \sigma_1^m \rrbracket_{\Enh}$ when we go from $\sigma_1^m$ to $\sigma_1^{m+2}$, that is, when adding a full twist to the braid. The following graph-theoretic lemma gives us sufficient information about the differentials of the complexes.

\begin{lemma}\label{classifying the paths for T2}
    Let $z$ and $x$ be the unique unmatched cells of homological degrees $t$ and $t+1$, i.e. $z=1\dots 10^{\bx} 0^{\bx} \dots 0^{\bx} 0$ and $x=1\dots 11 0^{\bx} \dots 0^{\bx} 0$. There are exactly two paths $\alpha_1,\alpha_2\colon z\to x$. The path $\alpha_1=(z\nearrow x)$ consists of one edge whereas
    \begin{align*}
\alpha_2=&(1...10^{\bx}0^{\bx}...^{\bx}0\nearrow 1...10^{\bx}...^{\bx}01 \searrow 1...10^{\bx}0^{\bx}...^{\bx}0^{\bone}0\nearrow 1...10^{\bx}...^{\bx}0^{\bx} 10\searrow \dots \\
&\ldots \searrow1...10^{\bx}0^{\bone}0^{\bx}...^{\bx}0 \nearrow 1...10^{\bx}10^{\bx}...^{\bx}00 \searrow 1...10^{\bone}0^{\bx}...^{\bx}00\nearrow 1...110^{\bx}...^{\bx}0). 
\end{align*}
In case $t=-1$ implying $x=1\dots 1$, the paths $\alpha_1$ and $\alpha_2$ coincide so there is only one path $\alpha\colon z\to y$ consisting of a single edge.

\end{lemma}
\begin{proof}
    %The paths are determined by using tree-search, which takes into account Lemma 
    There are exactly two edges (non-zero morphisms) that start from $z$ and these will turn out to correspond to paths $\alpha_1$ and $\alpha_2$. After this initial step, every subsequent edge will be determined by Lemma \ref{T2 acyclicity/paths helper}
    %Repeated use of Lemma  \ref{T2 acyclicity/paths helper} reduces the possible paths down to exactly these two.
\end{proof}

\begin{proposition}\label{commutative squares of T2}
    For $t\geq -m$ there are isomorphisms $\psi_t$ between the chain spaces making the diagrams 
    $$
\begin{tikzcd}
(M\llbracket \sigma_1^m \rrbracket_{\Enh})^t \arrow[r,"\psi_t"] \arrow[d] & (M\llbracket \sigma_1^{m+2} \rrbracket_{\Enh}\{2\})^{t} \arrow[d] \\
(M\llbracket \sigma_1^m \rrbracket_{\Enh})^{t+1} \arrow[r, "\psi_{t+1}"]       & (M\llbracket \sigma_1^{m+2} \rrbracket_{\Enh}\{2\})^{t+1}          
\end{tikzcd}
    $$
commute. Consequently, for $t\leq -1$ there are isomorphisms $\varphi_t$ making the diagrams
$$
\begin{tikzcd}
(M\llbracket \sigma_1^m \rrbracket_{\Enh})^{t-1} \arrow[r, "\varphi_{t-1}"] \arrow[d] & (M\llbracket \sigma_1^{m+2} \rrbracket_{\Enh}[2]\{6\})^{t-1} \arrow[d] \\
(M\llbracket \sigma_1^m \rrbracket_{\Enh})^{t} \arrow[r, "\varphi_{t}"]       & (M\llbracket \sigma_1^{m+2} \rrbracket_{\Enh}[2]\{6\})^{t}      
\end{tikzcd}
    $$
     commute. 
\end{proposition}

\begin{proof}
    Both of the claims are proven with Proposition \ref{Morse comparison with natural transformation}. For the first diagram we assume $t\geq -m$ and define functor $F$ with 
    $$
    F\colon \mathbf{CP}(\llbracket \sigma_1^m \rrbracket_{\Enh}, M, t)\to \mathbf{CP}(\llbracket \sigma_1^{m+2} \rrbracket_{\Enh}\{2\}, M, t), \  
    F(y_1\dots y_m)=11y_1\dots y_m
    $$
    and on edges with $F(a\to b) =(F(a)\to F(b))$. The planar diagrams of $a$ and $F(a)$ are the same up to which means that we can define a natural transformation 
    $$
    \gamma\colon R_m\circ I_{m} \Rightarrow R_{m+2} \circ I_{m+2} \circ F,  \quad \gamma_c=\operatorname{id}_{R_m(c)}
    $$ 
    where $I_m$ and $R_m$ are the inclusion and remembering functors assosiated to $M\llbracket \sigma_1^m \rrbracket_{\Enh}$ and $I_{m+2}$, $R_{m+2}$ are assosiated to $M\llbracket \sigma_1^{m+2} \rrbracket_{\Enh} \{2\}$. Lemma \ref{classifying the paths for T2} help us to verify that all conditions are met for Proposition \ref{Morse comparison with natural transformation}, yielding the first diagram. 

    For the second diagram, we assume $t\leq -2$ and define functor
    $$
    A\colon \mathbf{CP}(\llbracket \sigma_1^m \rrbracket_{\Enh}, M, t)\to \mathbf{CP}(\llbracket \sigma_1^{m+2} \rrbracket_{\Enh}[2]\{6\}, M, t)
    $$
    which on enhanced words injects a subword $0^{\bx}0^{\bx}$ before the first zero:
    $$
    A(1\dots 1 0^s y_j\dots y_m)=1\dots 1 0^{\bx}0^{\bx}0^s y_j\dots y_m.
    $$
    For an edge $a\to b$ with $L(a\to b)> O(a)+1$ we assign $A(a\to b)=(A(a)\to A(b))$. For the edge $c\nearrow d$ with 
\begin{align}\label{nontrivial path for functor A in T2}
    c=1\dots10^{\bone}0^{\bx}\dots 0^{\bx}0\quad \text{and} \quad d= 1\dots110^{\bx}\dots 0^{\bx}0      
\end{align}
we put $A(c\nearrow d)=(A(c)\nearrow q_1\searrow q_2\nearrow q_3 \searrow q_4\nearrow A(d) )$ where
\begin{alignat*}{5}
    q_1&=1\dots 1&&0^{\bx}&&0^{\bx}&&1&&0^{\bx}\dots 0^{\bx}0 \\
    q_2&=1\dots 1&&0^{\bx}&&0^{\bone}&&0^{\bx}&&0^{\bx}\dots 0^{\bx} 0\\
    q_3&=1\dots 1&&0^{\bx}&&1&&0^{\bx}&&0^{\bx}\dots 0^{\bx} 0 \\
    q_4&=1\dots 1&&0^{\bone}&&0^{\bx}&&0^{\bone}&&0^{\bx}\dots 0^{\bx}0. 
\end{alignat*}
From Lemma \ref{classifying the paths for T2} one can see that this is enough to define $A$ and that $A$ meets requirements of 
Proposition \ref{Morse comparison with natural transformation} apart from naturality.
   
To meet the naturality condition, we again define $\eta_c =\operatorname{id}_{R_m(c)}$ to get
$$
    \eta\colon R_m\circ I_{m} \Rightarrow R_{m+2} \circ I_{m+2} \circ F.
$$ Verifying that this is a natural transformation amounts to checking that for all edges  $e\colon v\to w$, we have $R_m(e)=R_{m+2}(Ae)$. This is straightforward for edges $a\to b$ with $L(a\to b)> O(a)+1$ and for edges $c\nearrow d$ of Form \ref{nontrivial path for functor A in T2} it follows from the fact that all relevant edges correspond to $\pm\operatorname{id}$. Additionally one needs to check that the signs of all of these cobordisms match, which they do.    \end{proof}

\begin{comment}
    
To meet the naturality condition, we again define $\eta_c =\operatorname{id}_{R_m(c)}$. These $\eta_c$ combine into a natural 
transformation 
$$
    \eta\colon R_m\circ I_{m} \Rightarrow R_{m+2} \circ I_{m+2} \circ F
$$ given that for all edges  $e\colon v\to w$, we have $R_m(e)=R_{m+2}(Ae)$. For an edge $a\to b$ with $L(a\to b)> O(a)+1$ the commutation is straightforward and for edges $c\nearrow d$ of Form \ref{nontrivial path for functor A in T2} it follows from the fact that all relevant edges correspond to $\pm\operatorname{id}$. Additionally one needs to check that the signs of all of these cobordisms match, which they do. 
\end{comment}

\section{Discrete Morse theory for 3-torus braids}\label{Section: DMT for T3}

In this section we investigate braid diagrams $(\sigma_1\sigma_2)^k$ with $k\geq 0$, see Figure \ref{Braid diagram of T3} and their Khovanov complexes $\llbracket(\sigma_1\sigma_2)^k\rrbracket\in \Kom( \Mat( \Cob(6)))$. The main objectives of the section are Propositions \ref{first commutative square of T3} and \ref{second commutative square of T3} which are analogous to Proposition \ref{commutative squares of T2}.

A circle in a smoothing of $(\sigma_1\sigma_2)^k$ occurs between two 0-smoothed crossings which have an odd number of 1-smoothed crossings in between them. Thus the cells of complex $\llbracket (\sigma_1 \sigma_2)^k\rrbracket_{\Enh}$ are enhanced words of length $2k$ with $\bx$ and $\bone$ superscripts on 0:s which are followed up by an odd number of 1:s and a 0.  The rest of the 0:s and 1:s have $\bminus$ superscripts, which we again omit.

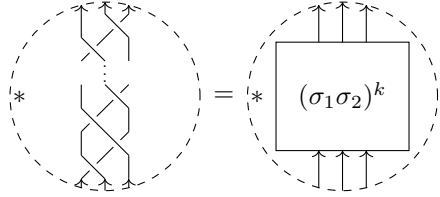
\begin{figure}[ht]
    \centering
    
\begin{tikzpicture}[scale=0.3125]
    %first circle

    \negCrossing{0,0}
    \lineUp{2,0}
    \negCrossing{1,1}
    \lineUp{0,1}
    \negCrossing{0,2}
    \lineUp{2,2}
    \negCrossing{1,3}
    \lineUp{0,3}

    \node[font=\small] at (1,4.7) {$\scalebox{0.7}{\vdots}$};

    \negCrossing{0,5}
    \lineUp{2,5}
    \negCrossing{1,6}
    \lineUp{0,6}

    \draw[<-] (1,0) -- (1,-0.5);
    \draw[->] (1,7) -- (1,7.5);
    
    \draw[->] (0,7) -- (0,7.38);
    
    \draw[->] (2,7) -- (2,7.38);
    \draw[<-] (0,0) -- (0,-0.38);
    \draw[<-] (2,0) -- (2,-0.38);

    \draw[dashed] (1, 3.5) circle (4);

    \node at (-2.6,3.5) {$*$};
    \node at (7.4,3.5) {$*$};
    %eq

    \node at (6,3.5) {$=$};

    %second circle

    \begin{scope}[shift={(1,0)}]
        
    \draw[dashed] (10, 3.5) circle (4);

    \draw (12.8, 5.8) rectangle (7.2, 1.2);

    \node at (10,3.5) {$(\sigma_1\sigma_2)^k$};

    \draw[<-] (9,1.2) -- (9,-0.38);
    \draw[<-] (11,1.2) -- (11,-0.38);
    
    \draw[->] (9,5.8) -- (9,7.38);
    \draw[->] (11,5.8) -- (11,7.38);

    \draw[<-] (10,1.2) -- (10,-0.5);
    \draw[->] (10,5.8) -- (10,7.5);
    
    \end{scope}

\end{tikzpicture}

    \caption{Braid diagram of $(\sigma_1\sigma_2)^k$ with $k\geq 0$ drawn in two ways.  }
    \label{Braid diagram of T3}
\end{figure}

This time, the matching $M$ on $\llbracket(\sigma_1\sigma_2)^k\rrbracket_{\Enh}$ is defined\footnote{Once more, the set $M$ could also be obtained by scanning through the braid and collecting isomorphisms between vertices, which have not been matched yet.}  by a mouthful
$$
M=\left\{
\begin{array}{l}
     z_1\nearrow x_1,  \\[2pt]
     z_2\nearrow x_2, \\[2pt]
     z_3\nearrow x_3
\end{array}
\ \middle\vert  
\begin{array}{l}
       z_1=1\dots1 00^{\bx}100^{\bx}1\dots 00^{\bx}1 0^{\bone} 10^s y_j \dots y_{2k} \\
       x_1=1\dots1 00^{\bx}100^{\bx}1\dots 00^{\bx}1 0^{s} 11 y_j \dots y_{2k} \\[2pt]
       z_2=1\dots1 00^{\bx}100^{\bx}1\dots 00^{\bx}1 00^{\bone} 10^s y_j \dots y_{2k} \\
       x_2=1\dots1 00^{\bx}100^{\bx}1\dots 00^{\bx}1 00^{s} 11 y_j \dots y_{2k} \\[2pt]
       z_3=1\dots1 00^{\bx}100^{\bx}1\dots 00^{\bx}1 000^s y_j \dots y_{2k} \\
       x_3=1\dots1 00^{\bx}100^{\bx}1\dots 00^{\bx}1 0^{\bx} 10 y_j \dots y_{2k} \\[2pt]
       4\leq j \leq 2k+1,\ y_j,\dots,y_{2k} \in Y,\ s\in \{\bone, \bx, \bminus \} 
\end{array}\right\}
$$
%$1\dots1$, $00^{\bx}100^{\bx}1\dots 00^{\bx}1$ and $y_j\dots y_{2k}$
where one should notice that the length of each of the three parts $$1\dots1, \qquad 00^{\bx}100^{\bx}1\dots 00^{\bx}1, \qquad y_j\dots y_{2k}$$ can be zero (as long as the total length stays $2k$). Again, it is relatively straightforward to see that $M$ is finite,  it consists of isomorphims cobordisms and it induces a partial matching  on $G=G(\llbracket (\sigma_1 \sigma_2)^k\rrbracket_{\Enh}, M)$. To conclude that $M$ is a Morse matching, we still need acyclicity which is again proven by constructing a strict preorder.

%observe that $M$ is finite, consists of isomorphims cobordisms and induces a partial matching  on $G=G(\llbracket\sigma_1^m \rrbracket_{\Enh}, M)$
\begin{lemma}
    The graph $G(\llbracket (\sigma_1 \sigma_2)^k\rrbracket_{\Enh}, M)$ has no directed cycles and thus $M$ is a Morse matching.
\end{lemma}
\begin{proof}
    Denote $M_{12}$ as the subset of $M$ corresponding to arrows $0^{\bone}10^s\nearrow 0^{s}11$ (arrows $z_1\nearrow x_1$ and $z_2\nearrow x_2$ in the definition of $M$) and denote $M_3=M\setminus M_{12}$. We define strict preorder $\prec$ on $M$ by assigning $(a\nearrow b)\prec (a'\nearrow b')$ if one of the following mutually exclusive conditions hold:
    \begin{enumerate}
        \item $O(a)<O(a')$\label{T3 preorder condition 1}
        \item $O(a)=O(a')$, $(a\nearrow b)\in M_3$ and $(a'\nearrow b')\in M_{12}$ \label{T3 preorder condition 2}
        \item $O(a)=O(a')$, $(a\nearrow b),(a'\nearrow b')\in M_{12}$ and $L(a\nearrow b)> L(a'\nearrow b')$ \label{T3 preorder condition 3}
        \item $O(a)=O(a')$, $(a\nearrow b),(a'\nearrow b')\in M_{3}$ and $L(a\nearrow b)< L(a'\nearrow b')$. \label{T3 preorder condition 4}
    \end{enumerate}

    Let $x\searrow z \nearrow x' \searrow z'$ be a zig-zag path, meaning $(z\nearrow x),(z' \nearrow x')\in M$. The goal, $(z\nearrow x)\prec(z' \nearrow x')$, is obtained by splitting to several cases. Assume first, that $(z\nearrow x)\in M_{3}$ which implies 
    $$
    z=1\dots1 00^{\bx}100^{\bx}1\dots 00^{\bx}1 000^s y_{j} \dots y_{2k}
    $$
    where $j=L(x\searrow z)+2$.  

    Case: $L(z\nearrow x')=O(z)+1$. By Condition \ref{T3 preorder condition 1} we have $(z\nearrow x)\prec(z' \nearrow x')$.

    Case: $O(z)+1 < L(z\nearrow x')< L(x\searrow z)$. Then $(z'\nearrow x')\in M_{12}$ implying $(z\nearrow x)\prec(z' \nearrow x')$ by Condition \ref{T3 preorder condition 2}.

    Case: $L(z\nearrow x')= L(x\searrow z)$. There are two choices for the superscript of the subword of $x'$  corresponding to indices $L(z\nearrow x')-1,L(z\nearrow x'), L(z\nearrow x')+1$: either $0^{\bx} 10^s$ or $0^{\bone} 10^s$. The subword $0^{\bx} 10^s$ is impossible since it implies $x=x'$ and in  $G$ arrow is $z\nearrow x$ goes in the wrong direction. The subword $0^{\bone} 10^s$ is also impossible as then $x'$ would be matched up in homological degree, whereas $x' \searrow  z'$ compels $x'$ to be matched down.

    Case: $L(z\nearrow x')> L(x\searrow z)$. We get $(z\nearrow x)\prec(z' \nearrow x')$ either from Condition \ref{T3 preorder condition 2} or \ref{T3 preorder condition 4}.
    
    The proof for case $(z\nearrow x)\in M_{12}$ is similar.
\end{proof}

    Now that we have proven $M$ to be a Morse matching, we obtain a chain homotopic Morse complex $M\llbracket (\sigma_1 \sigma_2)^k\rrbracket_{\Enh}$ for every $k$. A combinatorial investigation shows that the critical cells are enhanced words of the following forms
\begin{align*}
        &1\dots 100^{\bx}100^{\bx}1\dots 00^{\bx}1 0\\ %\label{critical cell of form a}\\
    &1\dots 100^{\bx}100^{\bx}1\dots 00^{\bx}101\\ % \label{critical cell of form b} \\
    &1\dots 100^{\bx}100^{\bx}1\dots 00^{\bx}1001 \\ % \label{critical cell of form c} \\
    &1\dots 100^{\bx}100^{\bx}1\dots 00^{\bx}100 \\ %\label{critical cell of form d} \\
    &1\dots 1 % \label{critical cell of form e}
\end{align*}
    and that there are at most 2 cells in each homological level of $M\llbracket (\sigma_1\sigma_2)^k \rrbracket_{\Enh}$. 

    \subsection{Adding a full twist to 3-torus braids}
    
    From the Propositions \ref{first commutative square of T3} and \ref{second commutative square of T3} one will be able to see that an analogous phenomenon is happening in  Morse complexes of 3-torus braids as happened in the 2-torus braids. The Morse complex $M\llbracket (\sigma_1 \sigma_2)^{3k}\rrbracket_{\Enh}$ and its similarity to $M\llbracket (\sigma_1 \sigma_2)^{3k+3}\rrbracket_{\Enh}$ are illustrated in Figure \ref{Figure for T3 complex and similarity}. The appearance and behaviour of complexes $M\llbracket (\sigma_1 \sigma_2)^{3k+1}\rrbracket_{\Enh}$ and $M\llbracket (\sigma_1 \sigma_2)^{3k+2}\rrbracket_{\Enh}$ is analogous.  
    
    In the proofs that follow, we will use the shorthand $G_{k,t}$ for the  subgraph of  $G(\llbracket(\sigma_1\sigma_2)^k\rrbracket_{\Enh},M)$ induced by restricting to vertices of homological degree $t$ and $t+1$.
    
    %The behaviour of complexes $M\llbracket (\sigma_1 \sigma_2)^{3k+1}\rrbracket_{\Enh}$ and $M\llbracket (\sigma_1 \sigma_2)^{3k+2}\rrbracket_{\Enh}$ when adding a full twist closely resemble the behaviour of  $M\llbracket (\sigma_1 \sigma_2)^{3k}\rrbracket_{\Enh}$  and the congruence of $k$ modulo 3  affects only the lowest homological degrees as well as the grading shift. 
    
%adding the full twist moneen paikkaan

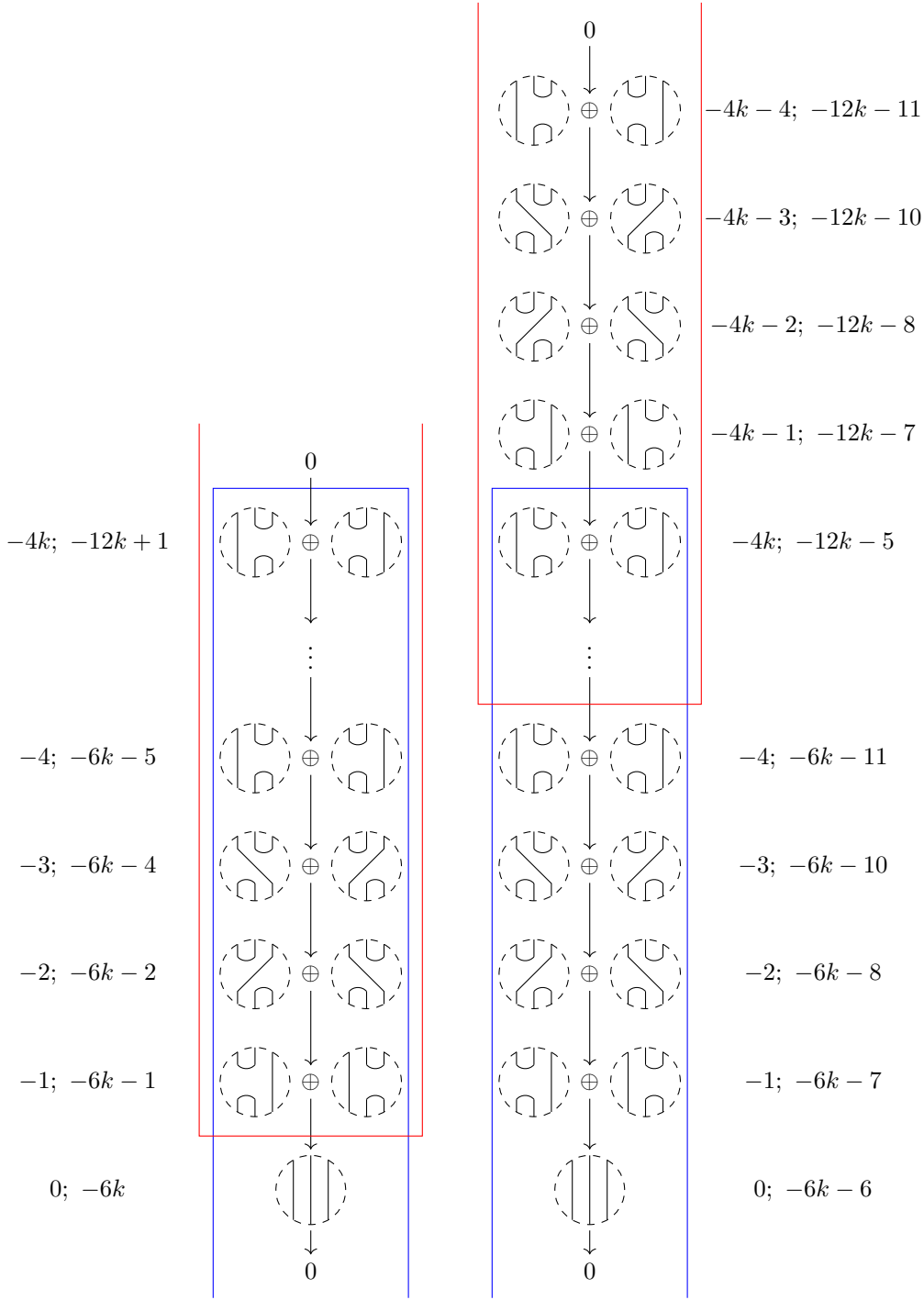
\begin{figure}
    \centering

    \begin{tikzpicture}%[yscale=1]
\pgfmathsetmacro{\cbound}{1.732}
\pgfmathsetmacro{\scalef}{0.25}
\pgfmathsetmacro{\yscalef}{0.78}
\pgfmathsetmacro{\zf}{3.2}

    \begin{scope}[shift={(0,0*\yscalef)}, scale=\scalef, xscale=1]
        \draw[dashed] (0,0) circle (2);
        \draw(-1,-\cbound) -- (-1,\cbound);
        \draw(1,-\cbound) -- (1,\cbound);
        \draw(0,-2) -- (0,2);
    \end{scope}  

    \begin{scope}[shift={(4,0*\yscalef)}, scale=\scalef, xscale=1]
        \draw[dashed] (0,0) circle (2);
        \draw(-1,-\cbound) -- (-1,\cbound);
        \draw(1,-\cbound) -- (1,\cbound);
        \draw(0,-2) -- (0,2);
    \end{scope}

    \begin{scope}[shift={(-0.8,2*\yscalef)}, scale=\scalef, xscale=-1]
        \draw[dashed] (0,0) circle (2);
        
        \draw (0,-2) -- (0,-1);
        \lineRightBendUp{0,-1}
        \draw (1,-\cbound) -- (1,-1);
        
        \draw (0,2) -- (0,1);
        \lineRightBendDown{0,1}
        \draw (1,1) -- (1,\cbound);
        
        \draw(-1,-\cbound) -- (-1,\cbound);
    \end{scope} 

    \begin{scope}[shift={(0.8,2*\yscalef)}, scale=\scalef, xscale=1]
        \draw[dashed] (0,0) circle (2);
        
        \draw (0,-2) -- (0,-1);
        \lineRightBendUp{0,-1}
        \draw (1,-\cbound) -- (1,-1);
        
        \draw (0,2) -- (0,1);
        \lineRightBendDown{0,1}
        \draw (1,1) -- (1,\cbound);
        
        \draw(-1,-\cbound) -- (-1,\cbound);
    \end{scope}

    \begin{scope}[shift={(4-0.8,2*\yscalef)}, scale=\scalef, xscale=-1]
        \draw[dashed] (0,0) circle (2);
        
        \draw (0,-2) -- (0,-1);
        \lineRightBendUp{0,-1}
        \draw (1,-\cbound) -- (1,-1);
        
        \draw (0,2) -- (0,1);
        \lineRightBendDown{0,1}
        \draw (1,1) -- (1,\cbound);
        %\node at (-1,0) {$*$};
        \draw(-1,-\cbound) -- (-1,\cbound);
    \end{scope}    

    \begin{scope}[shift={(4+0.8,2*\yscalef)}, scale=\scalef, xscale=1]
        \draw[dashed] (0,0) circle (2);
        
        \draw (0,-2) -- (0,-1);
        \lineRightBendUp{0,-1}
        \draw (1,-\cbound) -- (1,-1);
        
        \draw (0,2) -- (0,1);
        \lineRightBendDown{0,1}
        \draw (1,1) -- (1,\cbound);
        
        \draw(-1,-\cbound) -- (-1,\cbound);
    \end{scope}

    \begin{scope}[shift={(4-0.8,4*\yscalef)}, scale=\scalef, xscale=1]
        \draw[dashed] (0,0) circle (2);
        
        \draw (0,-2) -- (0,-1);
        \lineRightBendUp{0,-1}
        \draw (1,-\cbound) -- (1,-1);
        
        \draw (0,2) -- (0,1);
        \lineRightBendDown{-1,1}
        \draw (-1,1) -- (-1,\cbound);

        \draw (-1,-\cbound) -- (-1,-1) -- (1,1) -- (1,\cbound);
    \end{scope}

    \begin{scope}[shift={(4+0.8,4*\yscalef)}, scale=\scalef, xscale=-1]
        \draw[dashed] (0,0) circle (2);
        
        \draw (0,-2) -- (0,-1);
        \lineRightBendUp{0,-1}
        \draw (1,-\cbound) -- (1,-1);
        
        \draw (0,2) -- (0,1);
        \lineRightBendDown{-1,1}
        \draw (-1,1) -- (-1,\cbound);

        \draw (-1,-\cbound) -- (-1,-1) -- (1,1) -- (1,\cbound);
    \end{scope}

    \begin{scope}[shift={(-0.8,4*\yscalef)}, scale=\scalef, xscale=1]
        \draw[dashed] (0,0) circle (2);
        
        \draw (0,-2) -- (0,-1);
        \lineRightBendUp{0,-1}
        \draw (1,-\cbound) -- (1,-1);
        
        \draw (0,2) -- (0,1);
        \lineRightBendDown{-1,1}
        \draw (-1,1) -- (-1,\cbound);

        \draw (-1,-\cbound) -- (-1,-1) -- (1,1) -- (1,\cbound);
    \end{scope}

    \begin{scope}[shift={(0.8,4*\yscalef)}, scale=\scalef, xscale=-1]
        \draw[dashed] (0,0) circle (2);
        
        \draw (0,-2) -- (0,-1);
        \lineRightBendUp{0,-1}
        \draw (1,-\cbound) -- (1,-1);
        
        \draw (0,2) -- (0,1);
        \lineRightBendDown{-1,1}
        \draw (-1,1) -- (-1,\cbound);

        \draw (-1,-\cbound) -- (-1,-1) -- (1,1) -- (1,\cbound);
    \end{scope}

    \begin{scope}[shift={(4-0.8,6*\yscalef)}, scale=\scalef, xscale=-1]
        \draw[dashed] (0,0) circle (2);
        
        \draw (0,-2) -- (0,-1);
        \lineRightBendUp{0,-1}
        \draw (1,-\cbound) -- (1,-1);
        
        \draw (0,2) -- (0,1);
        \lineRightBendDown{-1,1}
        \draw (-1,1) -- (-1,\cbound);

        \draw (-1,-\cbound) -- (-1,-1) -- (1,1) -- (1,\cbound);
    \end{scope}

    \begin{scope}[shift={(4+0.8,6*\yscalef)}, scale=\scalef, xscale=1]
        \draw[dashed] (0,0) circle (2);
        
        \draw (0,-2) -- (0,-1);
        \lineRightBendUp{0,-1}
        \draw (1,-\cbound) -- (1,-1);
        
        \draw (0,2) -- (0,1);
        \lineRightBendDown{-1,1}
        \draw (-1,1) -- (-1,\cbound);

        \draw (-1,-\cbound) -- (-1,-1) -- (1,1) -- (1,\cbound);
    \end{scope}

    \begin{scope}[shift={(-0.8,6*\yscalef)}, scale=\scalef, xscale=-1]
        \draw[dashed] (0,0) circle (2);
        
        \draw (0,-2) -- (0,-1);
        \lineRightBendUp{0,-1}
        \draw (1,-\cbound) -- (1,-1);
        
        \draw (0,2) -- (0,1);
        \lineRightBendDown{-1,1}
        \draw (-1,1) -- (-1,\cbound);

        \draw (-1,-\cbound) -- (-1,-1) -- (1,1) -- (1,\cbound);
    \end{scope}

    \begin{scope}[shift={(0.8,6*\yscalef)}, scale=\scalef, xscale=1]
        \draw[dashed] (0,0) circle (2);
        
        \draw (0,-2) -- (0,-1);
        \lineRightBendUp{0,-1}
        \draw (1,-\cbound) -- (1,-1);
        
        \draw (0,2) -- (0,1);
        \lineRightBendDown{-1,1}
        \draw (-1,1) -- (-1,\cbound);

        \draw (-1,-\cbound) -- (-1,-1) -- (1,1) -- (1,\cbound);
    \end{scope}

    \begin{scope}[shift={(-0.8,8*\yscalef)}, scale=\scalef, xscale=1]
        \draw[dashed] (0,0) circle (2);
        
        \draw (0,-2) -- (0,-1);
        \lineRightBendUp{0,-1}
        \draw (1,-\cbound) -- (1,-1);
        
        \draw (0,2) -- (0,1);
        \lineRightBendDown{0,1}
        \draw (1,1) -- (1,\cbound);
        
        \draw(-1,-\cbound) -- (-1,\cbound);
    \end{scope} 

    \begin{scope}[shift={(0.8,8*\yscalef)}, scale=\scalef, xscale=-1]
        \draw[dashed] (0,0) circle (2);
        
        \draw (0,-2) -- (0,-1);
        \lineRightBendUp{0,-1}
        \draw (1,-\cbound) -- (1,-1);
        
        \draw (0,2) -- (0,1);
        \lineRightBendDown{0,1}
        \draw (1,1) -- (1,\cbound);
        
        \draw(-1,-\cbound) -- (-1,\cbound);
    \end{scope}

    \begin{scope}[shift={(4-0.8,8*\yscalef)}, scale=\scalef, xscale=1]
        \draw[dashed] (0,0) circle (2);
        
        \draw (0,-2) -- (0,-1);
        \lineRightBendUp{0,-1}
        \draw (1,-\cbound) -- (1,-1);
        
        \draw (0,2) -- (0,1);
        \lineRightBendDown{0,1}
        \draw (1,1) -- (1,\cbound);
        
        \draw(-1,-\cbound) -- (-1,\cbound);
    \end{scope}    

    \begin{scope}[shift={(4+0.8,8*\yscalef)}, scale=\scalef, xscale=-1]
        \draw[dashed] (0,0) circle (2);
        
        \draw (0,-2) -- (0,-1);
        \lineRightBendUp{0,-1}
        \draw (1,-\cbound) -- (1,-1);
        
        \draw (0,2) -- (0,1);
        \lineRightBendDown{0,1}
        \draw (1,1) -- (1,\cbound);
        
        \draw(-1,-\cbound) -- (-1,\cbound);
    \end{scope}

    \begin{scope}[shift={(-0.8,12*\yscalef)}, scale=\scalef, xscale=1]
        \draw[dashed] (0,0) circle (2);
        
        \draw (0,-2) -- (0,-1);
        \lineRightBendUp{0,-1}
        \draw (1,-\cbound) -- (1,-1);
        
        \draw (0,2) -- (0,1);
        \lineRightBendDown{0,1}
        \draw (1,1) -- (1,\cbound);
        
        \draw(-1,-\cbound) -- (-1,\cbound);
    \end{scope} 

    \begin{scope}[shift={(0.8,12*\yscalef)}, scale=\scalef, xscale=-1]
        \draw[dashed] (0,0) circle (2);
        
        \draw (0,-2) -- (0,-1);
        \lineRightBendUp{0,-1}
        \draw (1,-\cbound) -- (1,-1);
        
        \draw (0,2) -- (0,1);
        \lineRightBendDown{0,1}
        \draw (1,1) -- (1,\cbound);
        
        \draw(-1,-\cbound) -- (-1,\cbound);
    \end{scope}

    \begin{scope}[shift={(4-0.8,12*\yscalef)}, scale=\scalef, xscale=1]
        \draw[dashed] (0,0) circle (2);
        
        \draw (0,-2) -- (0,-1);
        \lineRightBendUp{0,-1}
        \draw (1,-\cbound) -- (1,-1);
        
        \draw (0,2) -- (0,1);
        \lineRightBendDown{0,1}
        \draw (1,1) -- (1,\cbound);
        
        \draw(-1,-\cbound) -- (-1,\cbound);
    \end{scope}    

    \begin{scope}[shift={(4+0.8,12*\yscalef)}, scale=\scalef, xscale=-1]
        \draw[dashed] (0,0) circle (2);
        
        \draw (0,-2) -- (0,-1);
        \lineRightBendUp{0,-1}
        \draw (1,-\cbound) -- (1,-1);
        
        \draw (0,2) -- (0,1);
        \lineRightBendDown{0,1}
        \draw (1,1) -- (1,\cbound);
        
        \draw(-1,-\cbound) -- (-1,\cbound);
    \end{scope}

    \begin{scope}[shift={(4-0.8,16*\yscalef)}, scale=\scalef, xscale=1]
        \draw[dashed] (0,0) circle (2);
        
        \draw (0,-2) -- (0,-1);
        \lineRightBendUp{0,-1}
        \draw (1,-\cbound) -- (1,-1);
        
        \draw (0,2) -- (0,1);
        \lineRightBendDown{-1,1}
        \draw (-1,1) -- (-1,\cbound);

        \draw (-1,-\cbound) -- (-1,-1) -- (1,1) -- (1,\cbound);
    \end{scope}

    \begin{scope}[shift={(4+0.8,16*\yscalef)}, scale=\scalef, xscale=-1]
        \draw[dashed] (0,0) circle (2);
        
        \draw (0,-2) -- (0,-1);
        \lineRightBendUp{0,-1}
        \draw (1,-\cbound) -- (1,-1);
        
        \draw (0,2) -- (0,1);
        \lineRightBendDown{-1,1}
        \draw (-1,1) -- (-1,\cbound);

        \draw (-1,-\cbound) -- (-1,-1) -- (1,1) -- (1,\cbound);
    \end{scope}

    \begin{scope}[shift={(4-0.8,18*\yscalef)}, scale=\scalef, xscale=-1]
        \draw[dashed] (0,0) circle (2);
        
        \draw (0,-2) -- (0,-1);
        \lineRightBendUp{0,-1}
        \draw (1,-\cbound) -- (1,-1);
        
        \draw (0,2) -- (0,1);
        \lineRightBendDown{-1,1}
        \draw (-1,1) -- (-1,\cbound);

        \draw (-1,-\cbound) -- (-1,-1) -- (1,1) -- (1,\cbound);
    \end{scope}

    \begin{scope}[shift={(4+0.8,18*\yscalef)}, scale=\scalef, xscale=1]
        \draw[dashed] (0,0) circle (2);
        
        \draw (0,-2) -- (0,-1);
        \lineRightBendUp{0,-1}
        \draw (1,-\cbound) -- (1,-1);
        
        \draw (0,2) -- (0,1);
        \lineRightBendDown{-1,1}
        \draw (-1,1) -- (-1,\cbound);

        \draw (-1,-\cbound) -- (-1,-1) -- (1,1) -- (1,\cbound);
    \end{scope}

    \begin{scope}[shift={(4-0.8,14*\yscalef)}, scale=\scalef, xscale=-1]
        \draw[dashed] (0,0) circle (2);
        
        \draw (0,-2) -- (0,-1);
        \lineRightBendUp{0,-1}
        \draw (1,-\cbound) -- (1,-1);
        
        \draw (0,2) -- (0,1);
        \lineRightBendDown{0,1}
        \draw (1,1) -- (1,\cbound);
        
        \draw(-1,-\cbound) -- (-1,\cbound);
    \end{scope}    

    \begin{scope}[shift={(4+0.8,14*\yscalef)}, scale=\scalef, xscale=1]
        \draw[dashed] (0,0) circle (2);
        
        \draw (0,-2) -- (0,-1);
        \lineRightBendUp{0,-1}
        \draw (1,-\cbound) -- (1,-1);
        
        \draw (0,2) -- (0,1);
        \lineRightBendDown{0,1}
        \draw (1,1) -- (1,\cbound);
        
        \draw(-1,-\cbound) -- (-1,\cbound);
    \end{scope}

    \begin{scope}[shift={(4-0.8,20*\yscalef)}, scale=\scalef, xscale=1]
        \draw[dashed] (0,0) circle (2);
        
        \draw (0,-2) -- (0,-1);
        \lineRightBendUp{0,-1}
        \draw (1,-\cbound) -- (1,-1);
        
        \draw (0,2) -- (0,1);
        \lineRightBendDown{0,1}
        \draw (1,1) -- (1,\cbound);
        
        \draw(-1,-\cbound) -- (-1,\cbound);
    \end{scope}    

    \begin{scope}[shift={(4+0.8,20*\yscalef)}, scale=\scalef, xscale=-1]
        \draw[dashed] (0,0) circle (2);
        
        \draw (0,-2) -- (0,-1);
        \lineRightBendUp{0,-1}
        \draw (1,-\cbound) -- (1,-1);
        
        \draw (0,2) -- (0,1);
        \lineRightBendDown{0,1}
        \draw (1,1) -- (1,\cbound);
        
        \draw(-1,-\cbound) -- (-1,\cbound);
    \end{scope}

    \node(A) at (0,2*\yscalef) {$\oplus$};
    \node(B) at (0,4*\yscalef) {$\oplus$};
    \node(C) at (0,6*\yscalef) {$\oplus$};
    \node(D) at (0,8*\yscalef) {$\oplus$};
    \node(E) at (0,10*\yscalef) {$\vdots$};
    \node(F) at (0,12*\yscalef) {$\oplus$};
    \node(G) at (4,2*\yscalef) {$\oplus$};
    \node(H) at (4,4*\yscalef) {$\oplus$};
    \node(I) at (4,6*\yscalef) {$\oplus$};
    \node(J) at (4,8*\yscalef) {$\oplus$};
    \node(K) at (4,10*\yscalef) {$\vdots$};
    \node(L) at (4,12*\yscalef) {$\oplus$};
    \node(M) at (4,14*\yscalef) {$\oplus$};
    \node(N) at (4,16*\yscalef) {$\oplus$};
    \node(O) at (4,18*\yscalef) {$\oplus$};
    \node(P) at (4,20*\yscalef) {$\oplus$};

    \foreach \x [remember=\x as \prevx (initially A)] in {B, C, D, E, F}
        \draw[<-] (\prevx) -- (\x);
    
    \foreach \x [remember=\x as \prevx (initially G)] in {H, I, J, K, L, M, N, O, P}
        \draw[<-] (\prevx) -- (\x);

    \draw[->] (A) -- (0,0.75*\yscalef);    
    \draw[->] (G) -- (4,0.75*\yscalef);

    %\node at (-4,0*\yscalef) {asdasdaaaaaaaaaaaa};

    \node [align=center] at (-\zf, 0*\yscalef) {$ 0 ; \ -6k$};    
    \node [align=center] at (-\zf, 2*\yscalef) {$ -1 ; \ -6k-1$};
    \node [align=center] at (-\zf, 4*\yscalef) {$ -2 ; \ -6k-2$};
    \node [align=center] at (-\zf, 6*\yscalef) {$ -3 ; \ -6k-4$};
    \node [align=center] at (-\zf, 8*\yscalef) {$ -4 ; \ -6k-5$};
    \node [align=center] at (-\zf, 12*\yscalef) {$ -4k ; \ -12k+1$};

    \node [align=center] at (4+\zf, 0*\yscalef) {$ 0 ; \ -6k-6$};
    \node [align=center] at (4+\zf, 2*\yscalef) {$ -1 ; \ -6k-7$};
    \node [align=center] at (4+\zf, 4*\yscalef) {$ -2 ; \ -6k-8$};
    \node [align=center] at (4+\zf, 6*\yscalef) {$ -3 ; \ -6k-10$};
    \node [align=center] at (4+\zf, 8*\yscalef) {$ -4 ; \ -6k-11$};
    \node [align=center] at (4+\zf, 12*\yscalef) {$ -4k ; \ -12k-5$};
    \node [align=center] at (4+\zf, 14*\yscalef) {$ -4k-1 ; \ -12k-7$};
    \node [align=center] at (4+\zf, 16*\yscalef) {$ -4k-2 ; \ -12k-8$};
    \node [align=center] at (4+\zf, 18*\yscalef) {$ -4k-3 ; \ -12k-10$};
    \node [align=center] at (4+\zf, 20*\yscalef) {$ -4k-4 ; \ -12k-11$};
    
    \draw[color=blue] (-1.4,-2*\yscalef) -- (-1.4, \yscalef*13) -- (1.4, \yscalef*13) -- (1.4,-2*\yscalef);
    
    \draw[color=blue] (4-1.4,-2*\yscalef) -- (4-1.4, \yscalef*13) -- (4+1.4, \yscalef*13) -- (4+1.4,-2*\yscalef);

    \draw[color=red] (-1.6,\yscalef*14.2) -- (-1.6, \yscalef*1) -- (1.6, \yscalef*1) -- (1.6,\yscalef*14.2);

    \draw[color=red] (4-1.6,\yscalef*22) -- (4-1.6, \yscalef*9) -- (4+1.6, \yscalef*9) -- (4+1.6,\yscalef*22);

    \node(swzero) at (0,-\yscalef*1.5) {$0$};
    \node(sezero) at (4,-\yscalef*1.5) {$0$};
    \node(nwzero) at (0,\yscalef*13.5) {$0$};
    \node(nezero) at (4,\yscalef*21.5) {$0$};

    \draw[->] (nwzero) -- (F); 
    \draw[->] (nezero) -- (P); 
    \draw[<-] (swzero) -- (0,-0.75*\yscalef);    
    \draw[<-] (sezero) -- (4,-0.75*\yscalef);
        
    %\node at (8,0*\yscalef) {asdasdaaaaaaaaaaaa};
    
\end{tikzpicture}
\thisfloatpagestyle{plain}

    \caption{Morse complexes $M\llbracket(\sigma_1 \sigma_2)^{3k}\rrbracket_{\Enh}$ (left) and $M\llbracket(\sigma_1 \sigma_2)^{3k}\rrbracket_{\Enh}$ (right). The homological gradings and internal grading shifts are displayed on the sides and separated with semicolons. The marked points $(*)$ at 9 o'clock on each circle are omitted for clarity.  Up to a grading shift, the blue and red parts are the same by Propositions \ref{first commutative square of T3} and \ref{second commutative square of T3}.
    On a sidenote, one can see that the complexes $M\llbracket(\sigma_1 \sigma_2)^{k}\rrbracket_{\Enh}$ are minimal by decategorifying them. The Euler characteristics of these complexes (which take values in the Temperley-Lieb planar algebra of $\Z[q^{\pm}]$-modules see e.g. Theorem 5.2. from \cite{MORRISON2008}) have the same number of monomials, as the complexes themselves have summands.}
    \label{Figure for T3 complex and similarity}
\end{figure}

%to the Groethendieck group, see e.g. Theorem 5.2. from asd. It is the result of the fact that t

\begin{proposition}\label{first commutative square of T3}
    
    For $t\geq -\lfloor 4k/3\rfloor$ there are isomorphisms $\psi_t$ making the following diagrams commute
   \[
\begin{tikzcd}
{(M\llbracket(\sigma_1\sigma_2)^k\rrbracket_{\Enh})^{t}} \arrow[r,"\psi_t"] \arrow[d] & {(M\llbracket(\sigma_1\sigma_2)^{k+3}\rrbracket_{\Enh}\{6\})^{t}}  \arrow[d] \\
 {(M\llbracket(\sigma_1\sigma_2)^k\rrbracket_{\Enh})^{t+1}} \arrow[r,"\psi_{t+1}"]            & {(M\llbracket(\sigma_1\sigma_2)^{k+3}\rrbracket_{\Enh}\{6\})^{t+1}}.               
\end{tikzcd}
\]
\end{proposition}
\begin{proof}
    Let $t\geq -\lfloor 4k/3\rfloor$ and define an injective graph homomorphism $f\colon G_{k,t}\to G_{k+3,t}$ by
    $$
    f(y_1,\dots, y_{2k})=111111y_1,\dots, y_{2k}. 
    $$
    where $y_j\in Y$ for all $j$. The map $f$ sends critical cells to critical cells bijectively and for any $a\in G_{k,t}$ it defines a bijection 
    $$
    \{b\in G_{k,t} \mid \exists r\colon a\to b \} \to \{b\in G_{k+3,t} \mid \exists r\colon f(a)\to b \}
    $$
    by $b\mapsto f(b)$. Hence $f$ descends to a functor 
    $$
    F\colon \mathbf{CP}(\llbracket(\sigma_1\sigma_2)^{k}\rrbracket_{\Enh},M,t)\to \mathbf{CP}(\llbracket(\sigma_1\sigma_2)^{k+3}\rrbracket_{\Enh}\{6\},M,t)
    $$
    for which we have bijections $\operatorname{Mor}(x,y)\cong \operatorname{Mor}(F(x),F(y))$.
    The result can now be obtained by defining a natural transformation
    $$
    \gamma\colon R_k\circ I_{k} \Rightarrow R_{k+3} \circ I_{k+3} \circ F, \quad \gamma_c=\operatorname{id}_{R_k(c)}
    $$ 
    and  employing Proposition \ref{Morse comparison with natural transformation}.
\end{proof}

\begin{proposition}\label{second commutative square of T3}
       For $t\leq -1$ there are isomorphisms $\varphi_t$ making the following diagrams commute
   \[
\begin{tikzcd}
{(M\llbracket(\sigma_1\sigma_2)^{k}\rrbracket_{\Enh})^{t-1}} \arrow[r,"\varphi_{t-1}"] \arrow[d] & {(M\llbracket(\sigma_1\sigma_2)^{k+3}\rrbracket_{\Enh}[4]\{12\})^{t-1}}  \arrow[d] \\
 {(M\llbracket(\sigma_1\sigma_2)^{k}\rrbracket_{\Enh})^{t}} \arrow[r,"\varphi_{t}"]            & {(M\llbracket(\sigma_1\sigma_2)^{k+3}\rrbracket_{\Enh}[4]\{12\})^{t}}.               
\end{tikzcd}
\]

\end{proposition}
\noindent The proof, which again uses Proposition \ref{Morse comparison with natural transformation}, is broken down to the following five steps.
\begin{enumerate}[label=\protect\bRoman{\arabic*}]
    \item Give a possibly insufficient description of functor $A$. \label{proof steps: insufficient description}
    \item State and prove the long and tedious Lemma \ref{Stack of statements about critical paths}. \label{proof steps: Lemma}
    \item With Lemma \ref{Stack of statements about critical paths} confirm that $A$ in fact is well-defined. \label{proof steps: description sufficies}
    \item With Lemma \ref{Stack of statements about critical paths} show that $A$ induces the necessary bijections for critical cells and for paths. \label{proof steps: bijections}
    \item Verify that setting $\eta_c=\operatorname{id}_{R_k(c)}$ defines a natural transformation. \label{proof steps: natural transformation}
\end{enumerate}

\textbf{Step  \ref{proof steps: insufficient description}}: We start by defining $A$ on vertices and on some edges. Let $t\leq -2$. For a vertex $v\in G_{k,t}$ we designate a vertex $A(v)\in  G_{k+3,t-4}$ by
$$
 %A(1\dots 10^{s_n}a_{n+1}^{s_{n+1}}\dots a_k^{s_k})=1\dots 10\overx \overx^{s_n}a_{n+1}^{s_{n+1}}\dots a_k^{s_k}.
 A(1\dots 10^{s}y_{j}\dots y_{2k})=1\dots 10\overx \overx^{s}y_{j}\dots y_{2k}.
$$
This sends critical cells to critical cells bijectively. 
For an edge $(u_1\to u_2)\in G_{k,t}$ with $L(u_1\to u_2)> O(u_1)+1$ we assign $A(u_1\to u_2)=(A(u_1)\to A(u_2))$. For edges $(v_1 \nearrow v_2)\in G_{k,t}$ with 
\begin{align}\label{edge v_1 to v_2}
    v_1=1\dots100^sy_j \dots y_{2k} \quad \text{and} \quad v_2= 1\dots110^sy_j \dots y_{2k}      
\end{align}
we put $A(v_1\nearrow v_2)=(A(v_1)\nearrow q_1\searrow q_2 \nearrow q_3 \searrow q_4 \nearrow A(v_2) )$ where
\begin{alignat}{10}
    q_1&=1\dots11&&0^{\bx}&& 1&& 0&&0^{\bx} &&1 &&0&&0^s&&y_j \dots y_{2k} \label{Equation q1} \\
    q_2&=1\dots11&&0&&0&&0&&0^{\bx} &&1 &&0&&0^s&&y_j \dots y_{2k} \label{Equation q2} \\
    q_3&=1\dots11&&0&&0^{\bx}&& 1&& 0^{\bx}&&1&&0&&0^s&&y_j \dots y_{2k} \label{Equation q3} \\
    q_4&=1\dots11&&0&&0^{\bx}&& 1 &&0&&0&&0&&0^s&&y_j \dots y_{2k} \label{Equation q4}
\end{alignat}
and for edges $(w_1 \nearrow w_2)\in G_{k,t}$ with 
\begin{align}\label{edge w_1 to w_2}
w_1=1\dots101^{\bone}0^sy_j \dots y_{2k} \quad \text{and} \quad w_2= 1\dots1110^sy_j \dots y_{2k}     
\end{align}
we set $A(w_1 \nearrow w_2)=(A(w_1) \nearrow r_1 \searrow\dots \searrow r_8 \nearrow A(w_2) )$ where
\begin{alignat*}{10}
    r_1&=1\dots 1&&0        &&0^{\bx}   &&1     &&0         &&0^{\bx}       &&1     &&1      &&1      &&0^sy_j \dots y_{2k} \\
    r_2&=1\dots 1&&0        &&0^{\bx}   &&1     &&0         &&0^{\bone}     &&1     &&0^{\bx}&&1      &&0^sy_j \dots y_{2k} \\
    r_3&=1\dots 1&&0        &&0^{\bx}   &&1     &&0         &&1             &&1     &&0^{\bx}&&1      &&0^sy_j \dots y_{2k} \\
    r_4&=1\dots 1&&0        &&0^{\bx}   &&1     &&0^{\bone} &&1             &&0     &&0^{\bx}&&1      &&0^sy_j \dots y_{2k} \\
    r_5&=1\dots 1&&0        &&0^{\bx}   &&1     &&1         &&1             &&0     &&0^{\bx}&&1      &&0^sy_j \dots y_{2k} \\
    r_6&=1\dots 1&&0        &&0^{\bone} &&1     &&0^{\bx}   &&1             &&0     &&0^{\bx}&&1      &&0^sy_j \dots y_{2k} \\
    r_7&=1\dots 1&&0        &&1         &&1     &&0^{\bx}   &&1             &&0     &&0^{\bx}&&1      &&0^sy_j \dots y_{2k} \\
    r_8&=1\dots 1&&0^{\bone}&&1         &&0     &&0^{\bx}   &&1             &&0     &&0^{\bx}&&1      &&0^sy_j \dots y_{2k}.
\end{alignat*}
From this construction we can see that the amount of reversed edges in $e$ and $Ae$ agrees modulo 2 for every edge $e$ in $G_{k,t}$. This concludes \textbf{Step \ref{proof steps: insufficient description}} and we can move to \textbf{Step \ref{proof steps: Lemma}}.

\begin{lemma}\label{Stack of statements about critical paths}
Let $z_1 \nearrow x_1 \searrow z_2 \nearrow \dots \searrow z_n \nearrow x_n$ be a zig-zag path in $G_{k,t}$ assume that $z_1$ and $x_n$  are critical. The following hold:
\begin{enumerate}[label=\roman*)]

    %\item We have $\hdeg(z_t)=\hdeg(y_u)-1$ for all indices $t,u$.
    %\item If $z_t=a_1\dots a_{i+2}0a_{i+4}\dots a_k$ and $y_t=a_1\dots a_{i+2}1 \allowbreak a_{i+4}\dots a_k$, then $a_1\dots a_{i+2}$ are as in Lemma \ref{critical until}.
    \item \label{010 increase odeg by 2} If $z_i=y_1\dots y_{j-1}\overone^s y_{j+3}\dots y_{2k}$,
then for some $l\geq i$ and some symbols $b_1,\dots ,b_{j+1} \in Y$  we have $x_l=b_1\dots b_{j+1}0^sy_{j+3}\dots y_{2k} $ and $O(x_l)\geq O(z_i)+2$. Since $O$ is an increasing function along paths, we in particular obtain $O(x_n)\geq O(z_i)+2$. 

\item \label{odeg z_t +2>= odeg x_n} We have $O(w)+2\geq O(x_n)$ for all $w\in \{z_i,x_i\mid 1\leq i\leq n\}$ except when 
\begin{alignat*}{2}
w=z_1&=1\dots 100^{\bx}&&100^{\bx}1\dots 00^{\bx}101 \\
x_n&= 1\dots 111&&100^{\bx}1\dots 00^{\bx}100
\end{alignat*}
in which case we can uniquely determine
$$
x_1=1\dots 110^{\bx}100^{\bx}1\dots 00^{\bx}101.
$$

\item \label{overone moves} The following implications hold for any $j, y_j,\dots,y_{2k}, s,i$:
%\begin{align*}
%    z_l&=1\dots 10\overx\dots \overx\overone^s a_{i}\dots a_k\\
%    \implies z_{l+1}&=1\dots 10\overx\dots \overx^{\bone}100^sa_{i}\dots a_k \\
%     z_u&=1\dots 10\overx\dots \overx^{\bone} 1 b_{j}\dots b_k\\
%    \implies z_{u+1}&=1\dots 10\overx\dots \overone^{\bx}1b_{j}\dots b_k.
%\end{align*}%haarukka
\begin{alignat*}{5}
    &z_i&&=1\dots 100^{\bx}100^{\bx}1\dots 00^{\bx}10&&0^{\bone}&&1&&0^s y_{j}\dots y_{2k}\\
    \implies& x_i&&=1\dots 100^{\bx}100^{\bx}1\dots00^{\bx}10 &&1&&1&&0^s y_{j}\dots y_{2k}\\   
    \implies& z_{i+1}&&=1\dots 100^{\bx}100^{\bx}1\dots 00^{\bx}10^{\bone}&&1&&0&&0^sy_{j}\dots y_{2k}
\end{alignat*}
\begin{alignat*}{5}
    &z_i&&=1\dots 100^{\bx}100^{\bx}1\dots 00^{\bx}1&&0^{\bone}&&1&&0^s y_{j}\dots y_{2k}\\
    \implies& x_i&&=1\dots 100^{\bx}100^{\bx}1\dots00^{\bx}1 &&1&&1&&0^s y_{j}\dots y_{2k}\\   
    \implies& z_{i+1}&&=1\dots 100^{\bx}100^{\bx}1\dots 00^{\bone}1&&0^{\bx}&&1&&0^s y_{j}\dots y_{2k}
\end{alignat*}
\item \label{requirements L=O+1 or L large} We have $L(z_i\nearrow x_i)=O(z_i)+1$ or $L(z_i\nearrow x_i)\geq j-1$ given that 
$$z_i=1\dots 100^{\bx}100^{\bx}1\dots00^{\bx}100^{s}y_{j}\dots y_{2k}.$$  

\item \label{no 01111 subword} The subword $0^s1111$ cannot be contained in $z_i$ or $x_i$ for any $s$ or $i$.

\item \label{surjectivity helper} Denote the vertex set $K=\{A(v) \mid v \in G_{k-3,t+4}  \}$. 
When $x_i\in K$ and $i<n$, we have $(x_i \searrow z_{i+1} )=A(u_1\searrow u_2)$ for some $u_1,u_2\in G_{k-3,t+4}$. Assuming $z_i\in K$, one of the following holds:
\begin{itemize}
    \item $(z_i \nearrow x_{i})=A(u_1 \nearrow u_2)$ for some $u_1, u_2 \in G_{k-3,t+4}$.
    
    \item $(z_i \nearrow x_{i}\searrow \dots \nearrow x_{i+2})=A(v_1 \nearrow v_2)$ for some $v_1, v_2 \in G_{k-3,t+4}$.
    
    \item $(z_i\nearrow x_{i}\searrow \dots \nearrow x_{i+4})=A(w_1\nearrow w_2)$ for some $w_1, w_2 \in G_{k-3,t+4}$.
\end{itemize}
%If $z_l=1\dots 10\overx\dots \overx\overone^s a_{i+3}\dots a_k$ then $$. 

\end{enumerate}

\end{lemma}
\begin{proof}
    \begin{enumerate}[label=\roman*)]
         
        \item Define $l$ as the largest index so that there exists $p\leq j$ and $c_1,\dots, c_{p-1},c_{p+3},\dots ,c_{j+1}\in Y$ with 
        $$
        z_l=c_1\dots c_{p-1}0^{\bone} 10^s c_{p+3}\dots c_{j+1}0y_{j+3}\dots y_{2k}.
        $$
        (In case of $p=j$ we have overlap so the latter 0 of $0^{\bone} 10^s$ is the 0 at index $j+2$.)
        
        Case $L(z_l\nearrow x_l)\leq p-1$. Clearly $x_l$ is not critical so $l<n$. Since arrow $x_l\searrow z_{l+1}$ goes downwards in homological degree $(z_{l+1} \nearrow x_l)\in M$ and we can see from the definition of $M$ that $L(x_l\searrow z_{l+1})\geq p$ is not possible. The case $L(x_l\searrow z_{l+1})< p$ is also impossible as we assumed that $l$ was the largest index. 

        Case $L(z_l\nearrow x_l)= p$. If there would be a pattern in $c_1\dots c_{p-1}$ which would force a matching of $M$ to take place, then both $z_l$ and $x_l$  would be matched downwards or both upwards in homological degree. Since this is impossible, one of the following holds
        \begin{align}
        c_1\dots c_{p-1}=&1\dots 100^{\bx}100^{\bx}1\dots 00^{\bx}1 0\label{critical cell of form a local}\\
    c_1\dots c_{p-1}=&1\dots 100^{\bx}100^{\bx}1\dots 00^{\bx}10^s1 \label{critical cell of form b local} \\
    c_1\dots c_{p-1}=&1\dots 100^{\bx}100^{\bx}1\dots 00^{\bx}100^s1 \label{critical cell of form c local} \\
    c_1\dots c_{p-1}=&1\dots 100^{\bx}100^{\bx}1\dots 00^{\bx}100 \label{critical cell of form d local} \\
    c_1\dots c_{p-1}=&1\dots 1. \label{critical cell of form e local}
\end{align}
         The Equations \ref{critical cell of form a local} to \ref{critical cell of form d local} would force $z_{l+1}$ to violate the maximality of $l$, whereas Equation \ref{critical cell of form e local} proves the claim.

        Case $L(z_l\nearrow x_l)= p+2$. Either this arrow is reversed or there is a pattern in $c_1\dots c_{p-1}0^s1$ resulting in both $z_l$ and $x_l$ being matched downwards or both upwards in homological degree. All three of these are absurd.
        
        Case $L(z_l\nearrow x_l)\geq p+3$. Both $z_l$ and $x_l$ are matched downwards or both upwards by a pattern in $c_1\dots c_{p-1}0^{\bone} 10^s$ which creates a contradiction.

        \item Every pair $z_1,x_n$ of critical cells in neighbouring  homological degrees satisfies that $O(z_1)+2\geq O(x_n)$ except for the pair
        \begin{alignat*}{2}
z_1&=1\dots 100^{\bx}&&100^{\bx}1\dots 00^{\bx}101 \\
x_n&= 1\dots 111&&100^{\bx}1\dots 00^{\bx}100
\end{alignat*}
        where $O(z_1)+3=O(x_n)$. In this case some downwards arrow $x_i\searrow z_{i+1}$ needs to get rid of 1 at the last index $2k$ in the zig-zag path $z_1 \nearrow x_1 \searrow \dots \searrow z_n \nearrow x_n$. This means $z_{i+1}=y_1\dots y_{2k-3}\overone$ and by \ref{010 increase odeg by 2} we must get $O(z_{i+1})+2\leq O (x_n)$. If $L(z_1 \nearrow x_1)>O(z_1)+1$ then again by \ref{010 increase odeg by 2} we get some $x_l$ with $l< i+1$ and 
        $$
        O(z_1)+4\leq O(x_l)+2\leq O(z_{i+1})+2 \leq O(x_n)=O(z_1)+3 
        $$
        which is a contradiction. Hence $L(z_1\nearrow y_1)=O(z_1)+1$ proving the claim.

        \item By dividing to cases based on $L(z_i  \nearrow x_i)$ and using \ref{010 increase odeg by 2} and \ref{odeg z_t +2>= odeg x_n} one can see that these are the only possible options.

        \item Any other choice for $L(z_i \nearrow x_i)$ would lead to some $l,p\leq j, s$ and $y_p,\dots,y_{2k}$ with 
        $$
        z_l=11\dots 10^{\bone}1 0 0^{\bx} 1 \dots 0 0^{\bx} 1000^s y_p \dots y_{2k}
        $$        
        by following the implications of \ref{overone moves}. After that no matter how we will try to construct an arrow $z_l \nearrow x_l$,  the vertex $x_l$ will be matched upwards resulting in a contradiction. 
        
        \item In the scenario where $z_i$ is the last vertex in the path $z_1\nearrow x_1\searrow \dots \nearrow x_n$ which contains the subword $0^s1111$, we can deduce that $O(z_i)+5=O(x_i)$ which combined with \ref{odeg z_t +2>= odeg x_n} creates a contradiction:
        $$
        O(z_i)+5= O(x_i)\leq  O(x_n) \leq O (z_i)+2.
        $$
        Assuming that $x_l$ is the last vertex with $0^s1111$ subword, we get a similar contradiction by using \ref{010 increase odeg by 2} twice and then \ref{odeg z_t +2>= odeg x_n}.

        \item 
        When $x_i\in K$ and $i<n$ it is easy to see that $L(x_i \searrow z_{i+1})>O(x_i)+7$ which further implies that $(x_i \searrow z_{i+1} )=A(u_1 \searrow u_2)$ for some $u_1,u_2\in G_{k-3,t+4}$. 
            
        Next we assume $z_i=A(a)$ for some $a\in G_{k-3,t+4}$ and divide into cases based on $L(z_i \nearrow x_i)$.

        Case $L(z_i \nearrow x_i)> O(z_i)+7$. It follows that $x_i=A(b)$ for some $b$ and we can put $(z_i \nearrow x_i)=A(a \nearrow b)$. 
        
        Case $L(z_i \nearrow x_i)= O(z_i)+7$. We know that 
        $$
        z_i=1\dots 10\overx\overx^{u}y_jy_{j+1}y_{j+2} \dots y_{2k+6}
        $$ 
        and we claim that $j<2k+6$, $u=\bone $ and $y_jy_{j+1}=10^s$ for some $s$ by ruling out all other possibilities. From \ref{requirements L=O+1 or L large} we see that $y_j=0$ is not possible, so $y_j=1$. The homological degree tells us that there are at least 6 zeros in $z_i$ which implies $j<2k+6$. Applying \ref{no 01111 subword} to $x_i$ rules out $y_jy_{j+1}= 11$ yielding $y_jy_{j+1}=10^s$. The subscript $u=\bx$ would force the vertex $z_i$ to get matched downwards in homological degree making it impossible for the edge $z_i \nearrow x_i$ to be part of a zig-zag path between critical cells. Hence
        $$
        z_i=1\dots 10\overx\overx^{\bone}10^s y_{j+2} \dots y_{2k+6}=A(w_1)
        $$ 
        where $w_1$ is some enhanced word satisfying Equation \ref{edge w_1 to w_2}. The fact that the next 8 edges also agree with the image of an edge $A(w_1 \nearrow w_2)$ of Equation \ref{edge w_1 to w_2} follows from \ref{overone moves}.
        %Merging circles marked with $x$ yields a zero-morphism and thus no edge implying $c\neq \bx$.
        
        Case $O(z_i)+1<L(z_i\nearrow x_i)< O(z_i)+7$. These are impossible by \ref{requirements L=O+1 or L large}.

        Case $L(z_i \nearrow x_i)= O(z_i)+1$. This and the next four edges are uniquely determined to complete a path coinciding with $A(v_1 \nearrow v_2)$ where $v_1$ and $v_2$ satisfy Equation \ref{edge v_1 to v_2}. This can be determined by exhausting all other options with \ref{overone moves} and \ref{requirements L=O+1 or L large}.  
        %The instance $y_l\in K$ is similar to the case $L(z_i\to x_i)> O(z_i)+7$ from above. 
    \end{enumerate}
\end{proof}
\textbf{Step \ref{proof steps: description sufficies}:}    Notice that we defined $A$ for all edges of $G_{k,t}$ except those $a \nearrow b$ and $a'\nearrow b'$ with
\begin{alignat*}{2}
L(a\nearrow b)&=O(a)+1, \quad &&a=1\dots 10^{\bx} 1 0^s y_j \dots y_{2k}, \\ L(a' \nearrow b')&=O(a')+1, \quad
  &&a'=1\dots 10^s11 y_j \dots y_{2k}.    
\end{alignat*}
The edge $a\nearrow b$ cannot be contained in any path between critical cells, since $a$ is matched downwards. The edge $a' \nearrow b'$ cannot be contained in a critical path either due to part  \ref{odeg z_t +2>= odeg x_n} of Lemma \ref{Stack of statements about critical paths}. It follows that $A$ determines a functor 
$$
    A\colon \mathbf{CP}(\llbracket(\sigma_1\sigma_2)^{k}\rrbracket_{\Enh},M,t)\to \mathbf{CP}(\llbracket(\sigma_1\sigma_2)^{k+3}\rrbracket_{\Enh}[4]\{12\},M,t).
$$

\textbf{Step \ref{proof steps: bijections}:} The fact that critical paths are sent injectively to critical paths is given by the fact that $A$ is injective on vertices and surjectivity is proven with part \ref{surjectivity helper} of Lemma \ref{Stack of statements about critical paths}.

\textbf{Step \ref{proof steps: natural transformation}:} What remains to be done is to verify that 
    $$
    \eta\colon R_k\circ I_{k} \Rightarrow R_{k+3} \circ I_{k+3} \circ A,\quad \eta_c=\operatorname{id}_{R_k(c)}
    $$ 
is a natural transformation. In other words we need to verify that for all edges  $e\colon v\to w$, the diagram
\begin{equation}
\begin{tikzcd} \label{natural transformation in second T3}
v \arrow[d, "e"] \arrow[r,"\eta_{v}"] & A(v) \arrow[d, "Ae"] \\
w \arrow[r,"\eta_{w}"]                 & A(w)                  
\end{tikzcd}    
\end{equation}
commutes in $\Mat(\Cob(6))$. 

Let $v_1\nearrow v_2$ be an edge and $q_1,\dots q_4$ be vertices as in Equations \ref{edge v_1 to v_2}-\ref{Equation q4}. Up to a sign, the arrows 
    $$
    v_1\to A(v_1),\quad v_2\to A(v_2) \quad \text{ and } \quad q_1\searrow q_2\nearrow q_3\searrow q_4\nearrow A(v_2).
    $$
 in the category $\Mat(\Cob(6))$ are identity cobordisms whereas the arrows $v_1\nearrow v_2$ and $A(v_1)\nearrow q_1$ coincide. By carefully calculating the signs, one can see that the Diagram \ref{natural transformation in second T3}
%    $$
    % https://tikzcd.yichuanshen.de/#N4Igdg9gJgpgziAXAbVABwnAlgFyxMJZABgBpiBdUkANwEMAbAVxiVoH0BGEAX1PUy58hFGU5VajFmxrsATL34gM2PASKdS46vWatEIAI5dFA1cI3kJu6QYCCAClmcAlLwkwoAc3hFQAMwAnCABbJDIQHAgkAGY+AOCwxAiopE54kCDQtOpUxDkMrKSY3Oj8ngoeIA
%\begin{tikzcd}
%v_1 \arrow[r] \arrow[d] & A(v_1) \arrow[d] \\
%v_2 \arrow[r]           & q_1             
%\end{tikzcd}
%    $$
    commutes for $v_1\nearrow v_2$. Checking the commutation %for edges of the form  \ref{edge w_1 to w_2} or $A(u_1)\to A(u_2)$ with $O(u_1)+1<L(u_1\to u_2)$ 
    for the other possible edges is only a minor variation of this. This concludes \textbf{Step} \ref{proof steps: natural transformation}. We have now met the requirements of Proposition \ref{Morse comparison with natural transformation} and thus proven Proposition \ref{second commutative square of T3}.    \qed

    \begin{remark}\label{Odd Khovanov homology remark}
        Odd Khovanov homology is a homology theory of oriented link diagrams with oriented crossings which agrees with the regular (even) Khovanov homology when taken with $\Z/2\Z$ coefficients. The sign assignments of odd Khovanov homology are more complicated and the theories are different over integers. The odd Khovanov homology also has a formulation based on cobordisms and in Theorem 8.2 of \cite{Schtz2022} the odd Khovanov homology of 3-strand torus links was calculated using the scanning algorithm. 

        Propositions \ref{first commutative square of T3} and \ref{second commutative square of T3} in this paper were written for the purpose of even Khovanov homology, at a time when the author was unaware of Schütz's work. Nevertheless,  we believe that our discrete Morse theory based approach also works for the odd Khovanov homology with the following adjustments, which we will state in the language of \cite{Schtz2022}.  Firstly, the category $\Mat(\Cob(6))$ is replaced with $\Z \operatorname{Chr}_{\bullet/X}(B,\dot{B})$. Order of delooping matters, so we reloop ($\Psi^{-1}$) from bottom to top and deloop ($\Psi$) from top to bottom. Matchings and paths will work out the same, but the signs will be different. This also leads to a slightly different result, since  \textbf{Step} \ref{proof steps: natural transformation} of \ref{second commutative square of T3} does not hold as is. However, applying the functor $A$ in  \textbf{Step} \ref{proof steps: natural transformation} twice leads to an agreement of the signs meaning that odd Khovanov homology of 3-stranded torus links has a period of length 8 whereas the even Khovanov homology has a period of length 4.        
        %Even though this paper was written to investigate even Khovanov homology and while being unaware of Schütz work, Propositions X and Y of this paper can be adapted to reproduce the result of Schütz. To do so
        %Apart from \textbf{Step} \ref{proof steps: natural transformation} of      
    \end{remark}

%Propositions X and Y were written to investigate even Khovanov homology and while being unaware of Schütz's work on odd Khovanov homology of torus links with 3 strands. Nevertheless, our discrete Morse theory based approach also work for the odd Khovanov homology, with some adjustments. The alterations are stated in the language from \cite{Schtz2022} without presenting the definitions.

\section{Composing tangles and Khovanov homology}\label{Section: Composing tangles}

\begin{comment}
%khovanov homology 
%DMT not needed
In Sections \ref{Section: DMT for T2} and \ref{Section: DMT for T3} we have worked hard to obtain fairly abstract results about Khovanov complexes of tangles. For the rest of this paper, we will no longer need discrete Morse theory but instead the goal of this section is to collect more concrete results about \term{Khovanov homology} groups of links. To do this, we will recall some theory by Bar-Natan enabling us to compose tangles  into links and Khovanov complexes of tangles into ``tensor complexes" which are equal to Khovanov complexes of links. For Khovanov complexes of links, it will make sense to talk about homology after taking a suitable functor. % (after we take a functor from $\Mat(\Cob(0))$ to $\mathbf{Mod}_{\Z}$, to be precise). 
 A schematic view into our arguments about Khovanov homology can be seen in Figure \ref{scheme for homology isomorphims via truncated complexes}.

\end{comment}

In Sections \ref{Section: DMT for T2} and \ref{Section: DMT for T3} discrete Morse theory was used to obtain fairly abstract results about Khovanov complexes of tangles. To collect more concrete results about Khovanov homology groups of links we will no longer need to use this tool. Instead, we will recall some theory by Bar-Natan enabling us to compose tangles  into links and Khovanov complexes of tangles into ``tensor complexes" which are equal to Khovanov complexes of links. For Khovanov complexes of links, it will make sense to talk about homology after taking a suitable functor. Homology has the advantage over chain homotopy type that an individual homology group is only affected by the neighbouring degrees which can be kept constant while slightly altering the link. A schematic view into our arguments about Khovanov homology can be seen in Figure \ref{scheme for homology isomorphims via truncated complexes}. 

 % (after we take a functor from $\Mat(\Cob(0))$ to $\mathbf{Mod}_{\Z}$, to be precise).

\begin{comment}
    What we will need from the Sections \ref{Section: DMT for T2} and \ref{Section: DMT for T3} are statements of Propositions \ref{commutative squares of T2}, \ref{first commutative square of T3} and \ref{second commutative square of T3} and some bounds on internal gradings. % which will be written in Equations \ref{internal gradings for T3 braids} and \ref{internal gradings for T2 braids}. 
The proof of this section are all build out of one key idea. When calculating homology of (tensor) complexes, 
We need to figure out, in which cases of  complexes

\end{comment}

\begin{figure}[htbp]
    \centering
    \begin{tikzpicture}
        %\draw (0, 0) rectangle (3, 2);
        \node[draw, rectangle, align=center] (A) at (0, 0) {Khovanov \\ homology of  link \\ $L_1$};
        
        \node[draw, rectangle, align=center] (B) at (4, 0) {Homology of  \\ ``tensor complex" \\ $D(\mathcal{A}, \mathcal{C}$)};

        \node[draw, rectangle, align=center] (C) at (8, 0) {Homology of \\  ``tensor complex" \\ with truncation \\ $D(\tau \mathcal{A}, \mathcal{C})$};

        \node[draw, rectangle, align=center] (A) at (0, -3) {Khovanov \\ homology of  link \\ $L_2$};
        
        \node[draw, rectangle, align=center] (B) at (4, -3) {Homology of  \\ ``tensor complex" \\ $D(\mathcal{B}, \mathcal{C}$)};

        \node[draw, rectangle, align=center] (C) at (8, -3) {Homology of \\  ``tensor complex" \\ with truncation \\ $D(\tau \mathcal{B}, \mathcal{C})$};

        %\node[draw, rectangle, align=center] (D) at (12, 0) {Homology of  \\ ``tensor complex" \\ $D(\mathcal{B}, \mathcal{C}$)};
        \node (E) at (2,0) {\large$\cong$};
        
        \node (F) at (6,0) {\large$\cong$};

        \node (K) at (2,-3) {\large$\cong$};
        
        \node (L) at (6,-3) {\large$\cong$};
        
        \node (G) at (8,-1.5) {\large$\cong$};

        \node[align=center, color=red, font=\small] (H) at (2, 1.5) {decompose with \\ Bar-Natan's theory };
        
        \node[align=center, color=red, font=\small] (J) at (6, 1.5) {carefully chosen \\ indices  };

        \node[align=center, color=blue] (I) at (10, -1.5) {$\tau \mathcal{A}\cong \tau \mathcal{B}$};

        \node[align=center, color=red, font=\small] (M) at (2, -4.5) {decompose with \\ Bar-Natan's theory };
        
        \node[align=center, color=red, font=\small] (N) at (6, -4.5) {carefully chosen \\ indices  };

        \draw[dashed, color=red!40] (E) -- (H);
        \draw[dashed, color=red!40] (F) -- (J);
        \draw[dashed, color=blue!40] (G) -- (I);
        
        \draw[dashed, color=red!40] (M) -- (K);
        \draw[dashed, color=red!40] (N) -- (L);
    \end{tikzpicture}
    \caption{Our scheme for obtaining isomorphisms of Khovanov homology groups. The blue isomorphisms of truncated complexes have been obtained in Sections \ref{Section: DMT for T2}  and \ref{Section: DMT for T3} and they will be restated in Proposition \ref{isomorphisms  of snip-complexes}.  In this Section we will justify the other isomorphisms by recalling some theory and finding bounds on indices. 
    }
    \label{scheme for homology isomorphims via truncated complexes}
\end{figure}
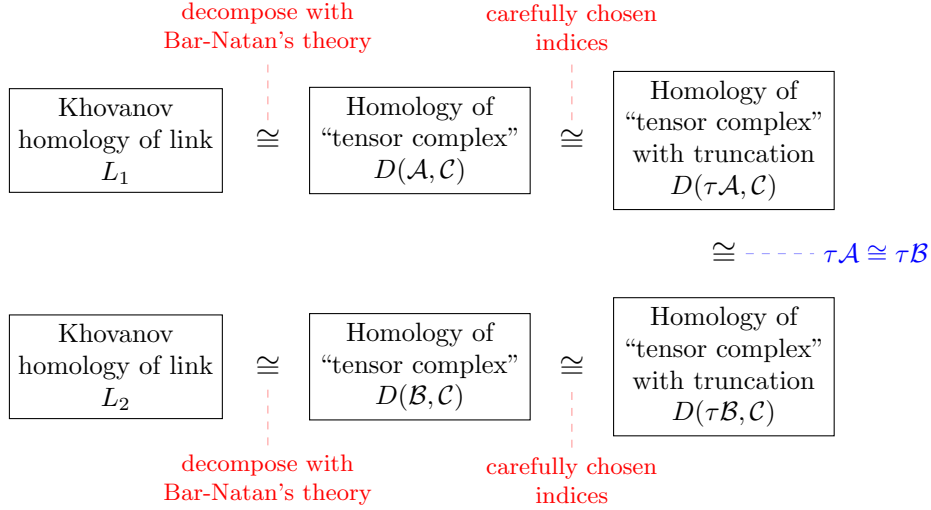

%The red isomorphisms will be justified in this section by recalling some theory and doing some  mundane work on the indices.

%Before recalling some theory of Bar-Natan, let us note that Discrete Morse theory is no longer needed for the rest of this paper.

%Our results will be obtained by using this information and by doing equations on the level of complexes and gradings. 
%Before defining compositions of tangles and complexes, let us note that Discrete Morse theory is no longer needed for the rest of this paper. 

An \term{$n$-input planar arc diagram} is a picture similar to the one drawn in Figure \ref{example of 3 input planar arc diagram} which creates a bigger tangle from $n$ smaller tangles. More formally, an $n$-input planar arc diagram $D$ is a circle with $n$ numbered empty input circles inside it. There are also oriented arcs that are closed or connected to the $2b$ outer boundary points and to the deleted circles from both ends. %The arcs are considered up to planar isotopy but their endpoints are fixed. 
To make the composition unambiguous, there are also stars $(*)$ on the sides of the circles.   Inputting arc, star and orientation matching tangles/tangle diagrams $T_1,\dots, T_n$ to $D$ yields a larger tangle/tangle diagram $D(T_1,\dots, T_n)$ in which case we say that the composition makes sense. One can do this also for complexes $(\mathcal{C}_i,d_i)$ over cobordism categories with matching boundaries to define ``tensor complex" $(D(\mathcal{C}_1,\dots , \mathcal{C}_n), d)$ with
$$
(D(\mathcal{C}_1,\dots , \mathcal{C}_n))^r= \bigoplus_{r=r_1+\dots r_n} D(\mathcal{C}_1^{r_1},\dots , \mathcal{C}_n^{r_n}).
$$
The differentials $d^r$ are defined by matrix elements with 
$$
d^r_{A\to B}=(-1)^{\sum_{j<i}r_j}D(\operatorname{id}_{\mathcal{C}_1^{r_1}},\dots,d^{r_i}_i,\dots,\operatorname{id}_{\mathcal{C}_n^{r_n}}   ) 
$$
where $A=D(\mathcal{C}_1^{r_1},\dots \mathcal{C}_i^{r_i},\dots, \mathcal{C}_n^{r_n})$ and $B=D(\mathcal{C}_1^{r_1},\dots \mathcal{C}_i^{r_i+1},\dots, \mathcal{C}_n^{r_n})$. These operations make collections of tangles and complexes into so called planar algebras.

\begin{figure}[ht]
    \centering
    
\begin{tikzpicture}[scale=0.6,
  arrowmark/.style 2 args={
    postaction=decorate,
    decoration={
      markings,
      mark=at position #1 with {\arrow{#2}}
    }
  }
]

% Define coordinates for the input disks
\coordinate (in1) at (1,1);
\coordinate (in2) at (-2,-0.1);
\coordinate (in3) at (2,-2);
\coordinate (tilt1) at (1,0);
\coordinate (tilt3) at ($(50:1)$);
\coordinate (tilt4) at ($(120:1)$);
\coordinate (tilt5) at ($(150:1)$);
\coordinate (tilt6) at ($(180:1)$);
\coordinate (tilt7) at ($(210:1)$);
\coordinate (tilt8) at ($(300:1)$);
\coordinate (tilt9) at ($(330:1)$);
% Draw the output disk
\draw[dashed] (0,0) circle (4);
% Draw the input disks
\foreach \c in {1,2,3} {
  \draw[dashed] (in\c) circle (1);
  \node at (in\c) {$\c$};
}
% Draw the arcs
\draw[->] (4,0) to[out=180,in=50,looseness=1.5] ($(in1)+(tilt3)$);
\draw[->] ($(in2)+(tilt8)$) to[out=300,in=150,looseness=1.5] ($(in3)+(tilt5)$);
\draw[->] ($(in2)+(tilt4)$) to[out=120,in=0,looseness=4] ($(in2)+(tilt1)$);
\draw[->] ($(in1)+(tilt8)$) to[out=300,in=330,looseness=1.5] ($(in2)+(tilt9)$);
\draw[arrowmark={0.5}{>}] ($(in3)+(tilt7)$) to[out=210,in=30,looseness=1.5] ($4*(tilt7)$);

\draw[->] (-0.5,2) arc (0:360:0.3) -- (-0.5,2);

\node at ($3.8*(tilt6)$) {*};
\node at ($(in1)+1.2*(tilt6)$) {*};
\node at ($(in2)+1.2*(tilt6)$) {*};
\node at ($(in3)+1.2*(tilt6)$) {*};

\node at ($4*(tilt7)$) {$\bullet$};
\node at (4,0) {$\bullet$};

\end{tikzpicture}
    \caption{An example of a 3-input planar arc diagram $D$ with 2 outer boundary points marked with $\bullet$. %The boundary $B(D)$ consists of two points marked with $\bullet$. 
    In the proceeding diagrams, the stars $(*)$ will be omitted for clarity. }
    \label{example of 3 input planar arc diagram}
\end{figure}
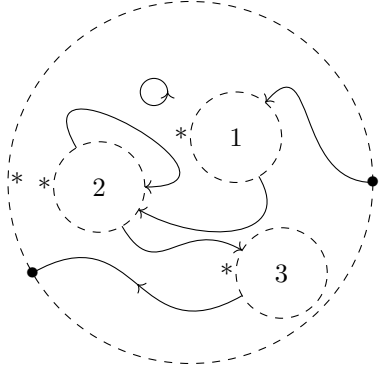

\begin{lemma}[Bar-Natan]\label{Bar-Natan planar algebra}
    Let $D$ be an $n$-input planar arc diagram. The planar algebra structure of $D$ commutes with the double bracket $\llbracket \cdot \rrbracket$, that is,
    $$
    \llbracket D(T_1,\dots,T_n)\rrbracket=D(\llbracket T_1\rrbracket,\dots,\llbracket T_n\rrbracket ).
    $$
    Moreover the planar algebra structure preserves homotopy and distributes with direct sums, that is, for compatible complexes $\mathcal{C}_1,\dots,\mathcal{C}_n,\mathcal{C}_i',\mathcal{D},\mathcal{D'}$ with $\mathcal{D}\simeq \mathcal{D'}$ we have
    $$
    \begin{array}{c}
     D(\mathcal{C}_1,\dots,\mathcal{C}_{i-1},\mathcal{D},\mathcal{C}_{i+1}\dots,\mathcal{C}_n)\simeq D(\mathcal{C}_1,\dots,\mathcal{C}_{i-1},\mathcal{D'},\mathcal{C}_{i+1}\dots,\mathcal{C}_n)      \\[2pt]
    D(\mathcal{C}_1,\dots,\mathcal{C}_i\oplus \mathcal{C}_i',\dots ,\mathcal{C}_n)\cong  D(\mathcal{C}_1,\dots,\mathcal{C}_i,\dots,\mathcal{C}_n)\oplus D(\mathcal{C}_1,\dots,\mathcal{C}'_i,\dots,\mathcal{C}_n).
    \end{array}
    $$
\end{lemma}
\begin{proof}
    See Theorem 2 from \cite{BarNatanKhovanovTangles}.
\end{proof}

Finally, to get an actual homology out of a link we need to apply a TQFT functor $\mathcal{F}$  from $\Mat(\Cob(0))$ to the category of graded $\Z$-modules where homology makes sense. %Since we are working with the dotted cobordism theory of Bar-Natan, the functor $\mathcal{F}$ is uniquely defined by the facts that it 
For us, it sufficies to know that the functor $\mathcal{F}$
preserves grading, is additive  and  $\mathcal{F}(\emptyset)=\Z\{0\}$. The \term{Khovanov homology} $\mathscr H$ of link $L$ is defined up to isomorphism by 
$$
\mathscr{H}^{*,*}(L)\cong H^{*,*}\mathcal{F}\llbracket \Tilde{L}\rrbracket
$$
where $\Tilde{L}$ is any link diagram of $L$ and $H^{*,*}$ denotes the regular cohomology functor from graded $\Z$-module cochain complexes to bigraded $\Z$-modules.

\subsection{Comparing homological gradings}

Suppose $(\mathcal{C},d)$ is a complex over an additive category $\mathbf{C}$ and $a\in \Z$. We define a \term{truncated complex} $(\tau^{\leq a}  \mathcal{C}, \partial)$ with
$$
(\tau^{\leq a}\mathcal{C})^i=
\begin{cases}
      \mathcal{C}^i, & \text{if}\ i\leq a \\
      0, & \text{otherwise}
    \end{cases}
\qquad
\partial^i=
\begin{cases}
      d^i, & \text{if}\ i< a  \\
      0, &  \text{otherwise}.
    \end{cases}
$$
and similarly we define $\tau^{\geq a} \mathcal{C}$. This terminology allows us to rephrase our previous results.

\begin{proposition}[Propositions \ref{commutative squares of T2}, \ref{first commutative square of T3} and \ref{second commutative square of T3}  rewritten] \label{isomorphisms  of snip-complexes}
    For $k,m\geq0$ there are isomorphisms of truncated complexes:
    \begin{align*}
        \tau^{\geq -m}M\llbracket\sigma_1^m\rrbracket_{\Enh}&\cong \tau^{\geq -m}(M\llbracket\sigma_1^{m+2}\rrbracket_{\Enh}\{2\}) \\
        \tau^{\leq -1}M\llbracket\sigma_1^m\rrbracket_{\Enh}&\cong \tau^{\leq -1}(M\llbracket\sigma_1^{m+2}\rrbracket_{\Enh}[2]\{6\}) \\
        \tau^{\geq -\lfloor 4k/3\rfloor}M\llbracket(\sigma_1\sigma_2)^k\rrbracket_{\Enh}&\cong \tau^{\geq -\lfloor 4k/3\rfloor}(M\llbracket(\sigma_1\sigma_2)^{k+3}\rrbracket_{\Enh}\{6\}) \\ 
        \tau^{\leq -1}M\llbracket(\sigma_1\sigma_2)^k\rrbracket_{\Enh}&\cong \tau^{\leq -1}(M\llbracket(\sigma_1\sigma_2)^{k+3}\rrbracket_{\Enh}[4]\{12\}).      
    \end{align*}
\end{proposition}

While truncations do not preserve homotopy type, we will see that they will leave some of the homology intact when moving to $\mathbf{Mod}_{\Z}$. To this end, we need some notation: for a complex $\mathcal{C}$, we denote
$$
h_{\min}(\mathcal{C})=\min\{i\in \Z \mid \mathcal{C}^i\not\cong  0\} \quad \text{ and } \quad  h_{\max}(\mathcal{C})=\max\{i\in \Z \mid \mathcal{C}^i\not\cong  0\}.
$$
All of our statements will be trivially true for zero complexes and we assumed that all complexes are finitely supported, so one does not need to worry about the existence of these minima and maxima. 

\begin{theorem}\label{isomorphisms of Khovanov homology groups with 2 input diagrams}
    Let $D$ be a 2-input planar arc diagram  without outer boundary and let  $T$ be a tangle so that $D(\sigma_1,T)$ makes sense. Let $\mathcal{C}$ be a complex with $\llbracket T \rrbracket \simeq \mathcal{C}$.  There are following isomorphisms of Khovanov homology groups:
    \begin{enumerate}    
        \item  For all $m\geq 0$, $i>h_{\max}(\mathcal{C})-m$ and $j$: 
        $$
    \mathscr{H}^{i,j}D(\sigma_1^m, T) \cong\mathscr{H}^{i,j-2} D(\sigma_1^{m+2}, T). 
    $$
        \item  For all $m\geq 0$, $i<h_{\min}(\mathcal{C}) -1$ and $j$: \label{Khovanov homology isomorphism 2}
        $$
    \mathscr{H}^{i,j} D(\sigma_1^m, T)\cong \mathscr{H}^{i-2,j-6}  D(\sigma_1^{m+2}, T). 
    $$    
    \end{enumerate}  Consequently, let $D'$ be a 2-input planar arc diagram   without outer boundary and let  $U$ be a tangle so that $D((\sigma_1 \sigma_2),U)$ makes sense. Assuming $\mathcal{C}'$ is a complex with $\llbracket U \rrbracket \simeq \mathcal{C}'$, we have:
    \begin{enumerate}
    \setcounter{enumi}{2}
        \item  For all $k\geq 0$, $i>h_{\max}(\mathcal{C}')-\lfloor 4k/3\rfloor$ and $j$: \label{Khovanov homology isomorphism 3}
        \begin{align}
    \mathscr{H}^{i,j} D'((\sigma_1\sigma_2)^k, U) \cong \mathscr{H}^{i,j-6} D'((\sigma_1\sigma_2)^{k+3}, U). \label{3rd isomorphism for 2-input diagrams}            
        \end{align}
        \item  For all $k\geq 0$, $i<h_{\min}(\mathcal{C}')-1$ and all $j$: \label{Khovanov homology isomorphism 4}
        \begin{align}
    \mathscr{H}^{i,j} D'((\sigma_1\sigma_2)^k, U) \cong \mathscr{H}^{i-4,j-12}D'((\sigma_1\sigma_2)^{k+3}, U). \label{4th isomorphism  for 2-input diagrams}        
        \end{align}
    
    \end{enumerate}

\end{theorem}
\begin{proof}
    Assume $m\geq 0$, $i>h_{\max}(\mathcal{C})-m$ and $j\in \Z$. Using Lemma \ref{Bar-Natan planar algebra}  we obtain:
    \begin{align*}
        \mathscr{H}^{i,j} D(\sigma_1^m,T)&\cong H^{i,j} \mathcal{F} \llbracket D (\sigma_1^m,T)\rrbracket  \cong H^{i,j} \mathcal{F} D(M\llbracket \sigma_1^m\rrbracket_{\Enh},\mathcal{C}).
         \end{align*}
The $\Z$-modules which affect this homology group are
$$
(\mathcal{F} D(M\llbracket \sigma_1^m\rrbracket_{\Enh},\mathcal{C}))^{i+o,j}
$$
where $o\in \{-1,0,1\}$. These modules consist of direct summands 
$$
\mathcal{F} D((M\llbracket \sigma_1^m\rrbracket_{\Enh})^a,\mathcal{C}^b) \quad \text{ with } \quad a+b\in \{i-1,i,i+1\}.
$$
None of these change from a non-zero module to a zero module as one passes from $M\llbracket \sigma_1^m\rrbracket_{\Enh}$ to $\tau^{\geq -m} M\llbracket \sigma_1^m\rrbracket_{\Enh}$ which also means that none of the relevant morphisms change either, yielding
$$
H^{i,j} M\llbracket \sigma_1^m\rrbracket_{\Enh} \cong H^{i,j}(\tau^{\geq -m} M\llbracket \sigma_1^m\rrbracket_{\Enh}).
$$
By using Proposition \ref{isomorphisms  of snip-complexes} and similar reasoning as above, we attain
    \begin{align*}
        H^{i,j} \mathcal{F} D(\tau^{\geq -m}M\llbracket \sigma_1^m\rrbracket_{\Enh},\mathcal{C})& \cong
        H^{i,j} \mathcal{F} D(\tau^{\geq -m}(M\llbracket \sigma_1^{m+2}\rrbracket_{\Enh}\{2\}),\mathcal{C})\\
        & \cong H^{i,j} \mathcal{F} D(M\llbracket \sigma_1^{m+2}\rrbracket_{\Enh}\{2\},\mathcal{C})\\
        & \cong H^{i,j-2} \mathcal{F} D(M\llbracket \sigma_1^{m+2}\rrbracket_{\Enh},\mathcal{C})\\
        & \cong \mathscr{H}^{i,j-2} D(\sigma_1^{m+2},T). 
    \end{align*}
The proofs for Claims \ref{Khovanov homology isomorphism 2}, \ref{Khovanov homology isomorphism 3} and \ref{Khovanov homology isomorphism 4} almost identical.    
\end{proof}

\begin{remark}\label{Remark for 2-input theorem}
When fixing a tangle diagram $U'$ for the tangle $U$ in the previous theorem, we can substitute $\mathcal{C}=\llbracket U' \rrbracket$. This gives $h_{\min}(\mathcal{C})=-n_-$ and $h_{\max}(\mathcal{C})=n_+$ where $n_+$ and $n_-$ are the number of positive and negative crossings of $U'$. Either Isomorphism  \ref{3rd isomorphism for 2-input diagrams} or  \ref{4th isomorphism for 2-input diagrams} can be used to reduce $k$ to $k-3$ whenever $i>n_+  -\lfloor \frac{4}{3}(k-3)\rfloor$ or $i+4<-n_- -1$. One of the two must hold given that $k\geq \frac{3}{4}(n_++n_-)+8$.
    
\end{remark}

By borrowing some simple reduction arguments from \cite{chandler_lowrance_sazdanović_summers_2022}, we can match their results about Khovanov homology of closures of $3$-braids.

    \begin{theorem}\label{Omega0 ... Omega3 contain only 2-torsion}
        The Khovanov homology groups of closures of braids from sets $\Omega_0$, $\Omega_1$, $\Omega_2$, $\Omega_3$ contain only $\Z / 2\Z$ torsion.
    \end{theorem}

    \begin{proof}
        By Section 5.1 of \cite{chandler_lowrance_sazdanović_summers_2022} it suffices to check the braid closures of  
        $$
        (\sigma_1\sigma_2)^k  \qquad \text{and} \qquad (\sigma_1\sigma_2)^{3k'+1}\sigma_1
        $$
        for $k\geq -2$ and $k'\geq 0$. Applying Theorem \ref{isomorphisms of Khovanov homology groups with 2 input diagrams} via Remark  \ref{Remark for 2-input theorem} to link diagrams $D((\sigma_1\sigma_2)^k,\emptyset)$ and $D'((\sigma_1\sigma_2)^k,\sigma_1)$ drawn in Figure \ref{Link diagrams for proof of Omega_0 to Omega_3} it is enough to check $k=-2,\dots ,7$ and $k'=0,1,2$.  \textbf{Computer Data I} \cite{computerDataFor3BraidPaper} shows that those 13 links only contain $\Z / 2\Z$ torsion.  
        \begin{figure}[ht]
            \centering
            
\begin{tikzpicture}[
  arrowmark/.style 2 args={
    postaction=decorate,
    decoration={
      markings,
      mark=at position #1 with {\arrow{#2}}
    }
  }
,scale=0.55]
    
    \begin{scope}[shift={(9,0)}]
        \draw[dashed] (1,0) circle (4);
        
    \begin{scope}[scale=0.5, shift={(-10,-3.5-1)}]
        
    \draw[dashed] (10, 3.5) circle (4);

    \draw (12.8, 5.8) rectangle (7.2, 1.2);

    \node at (10,3.5) {$(\sigma_1\sigma_2)^{k'}$};

    \draw[<-] (9,1.2) -- (9,-0.38);
    \draw[<-] (11,1.2) -- (11,-0.38);
    
    \draw[->] (9,5.8) -- (9,7.38);
    \draw[->] (11,5.8) -- (11,7.38);

    \draw[<-] (10,1.2) -- (10,-0.5);
    \draw[->] (10,5.8) -- (10,7.5);

    \end{scope}

%\draw[arrowmark={0.4}{>}] (100:3.0) arc (100:260:3.0);
%\draw[arrowmark={0.4}{>}] (90:3.3) arc (90:270:3.3);
%\draw[arrowmark={0.4}{>}] (80:3.6) arc (80:280:3.6);

%\draw (260:3.0) -- (-0.5,-2.4);
%\draw (270:3.3) -- (0,-2.5);
%\draw (280:3.6) -- (0.5,-2.4);

%\draw(-0.5,1.4) to[bend left=50]  (3.5,2.5);
%\draw(0,1.5) to[bend left=50]  (3,2.5);
\draw(0.5,1.4) to[bend left=50] (2.5,1.5);

\draw(-0.5,1.4)  --  (-0.55,1.9);
\draw(0,1.5) --  (0.2,1.9);

\draw(-0.55,1.9+0.75) to[bend left=50]  (3.5,2.5);
\draw(0.2,1.9+0.75) to[bend left=50]  (3,2.5);

    \begin{scope}[scale=0.75, shift={(-0.75,2.5)}]
    \negCrossing{0,0}
    \draw[dashed] (0.5, 0.5) circle (0.707);
    \end{scope}

\draw(-0.5,-1.4-1) to[bend left=-50]  (3.5,-1.5-1);
\draw(0,-1.5-1) to[bend left=-50]  (3,-1.5-1);
\draw(0.5,-1.4-1) to[bend left=-50] (2.5,-1.5-1);

    \draw[->] (3.5,2.5) -- (3.5,-2.5);
    \draw[->] (3,2.5) -- (3,-2.5);
    \draw[->] (2.5,1.5) -- (2.5,-2.5);

%\draw (100:3.0) -- (-0.5,2.5);
%\draw(90:3.3) -- (0,2.5);

%\draw[dashed] (-1,2.5) circle (0.5);
%\node at (-1,2.5) {$\emptyset$};
    
    \end{scope}   

    \begin{scope}[shift={(0,0)}]
        \draw[dashed] (1,0) circle (4);
        
    \begin{scope}[scale=0.5, shift={(-10,-3.5-1)}]
        
    \draw[dashed] (10, 3.5) circle (4);

    \draw (12.8, 5.8) rectangle (7.2, 1.2);

    \node at (10,3.5) {$(\sigma_1\sigma_2)^k$};

    \draw[<-] (9,1.2) -- (9,-0.38);
    \draw[<-] (11,1.2) -- (11,-0.38);
    
    \draw[->] (9,5.8) -- (9,7.38);
    \draw[->] (11,5.8) -- (11,7.38);

    \draw[<-] (10,1.2) -- (10,-0.5);
    \draw[->] (10,5.8) -- (10,7.5);

    \end{scope}

%\draw[arrowmark={0.4}{>}] (100:3.0) arc (100:260:3.0);
%\draw[arrowmark={0.4}{>}] (90:3.3) arc (90:270:3.3);
%\draw[arrowmark={0.4}{>}] (80:3.6) arc (80:280:3.6);

%\draw (260:3.0) -- (-0.5,-2.4);
%\draw (270:3.3) -- (0,-2.5);
%\draw (280:3.6) -- (0.5,-2.4);

\draw(-0.5,1.4) to[bend left=50]  (3.5,1.5);
\draw(0,1.5) to[bend left=50]  (3,1.5);
\draw(0.5,1.4) to[bend left=50] (2.5,1.5);

\draw(-0.5,-1.4-1) to[bend left=-50]  (3.5,-1.5-1);
\draw(0,-1.5-1) to[bend left=-50]  (3,-1.5-1);
\draw(0.5,-1.4-1) to[bend left=-50] (2.5,-1.5-1);

    \draw[->] (3.5,1.5) -- (3.5,-2.5);
    \draw[->] (3,1.5) -- (3,-2.5);
    \draw[->] (2.5,1.5) -- (2.5,-2.5);

%\draw (100:3.0) -- (-0.5,2.5);
%\draw(90:3.3) -- (0,2.5);

\draw[dashed] (-1,2.5) circle (0.5);
\node at (-1,2.5) {$\emptyset$};
    
    \end{scope}   
    
\end{tikzpicture}
            \caption{Link diagrams $D((\sigma_1\sigma_2)^k,\emptyset)$ (left) and $D'((\sigma_1\sigma_2)^k,\sigma_1)$ (right). }
            \label{Link diagrams for proof of Omega_0 to Omega_3}
        \end{figure}
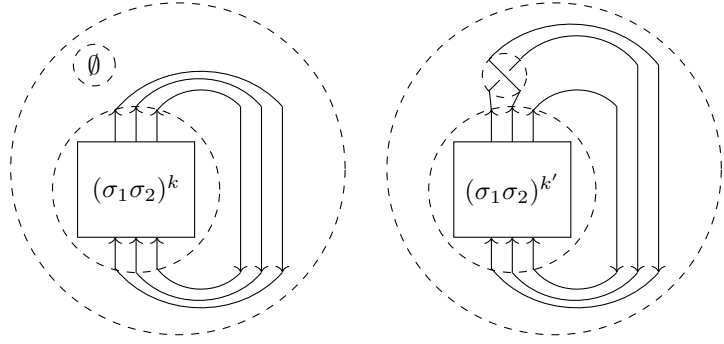
    \end{proof}

\subsection{Comparing homological and internal gradings}

To get more results we will need to also compare internal gradings in compositions of complexes. For any based complex with base indexed by $(J_i)_{i\in \Z}$ we denote the set of all cells
$$
B(\mathcal{C})= \{ \mathcal{C}_j^{i} \mid i\in \Z, \ j\in J_i\}.
$$
For a complex $\mathcal{D}$ over $\Mat(\Cob(2l))$ we denote $q_{\min} (\mathcal{D})$ and $q_{\max} (\mathcal{D})$ as the minimum and maximum internal degree shifts of $\Psi \mathcal{D}$, that is, 
\begin{align*}
    q_{\min} (\mathcal{D})&=\min\{ a\in \Z \mid R\{a\}\in B(\Psi\mathcal{D}) \} \\
    q_{\max} (\mathcal{D})&=\max\{ a\in \Z \mid R\{a\}\in B(\Psi\mathcal{D}) \}. 
\end{align*}
Additionally, we need to define the following quantity $s(D)$ for any planar arc diagram $D$ without outer boundary by
\begin{align*}
s(D)=&\# \text{strings of $D$ from input to the same input} \\
        &+\left\lfloor \tfrac{1}{2} \cdot \# \text{strings of $D$ from a input to a different input} \right \rfloor \\
        &+\#\text{closed loops of $D$}.    
\end{align*}

When going up 2 homological degrees in $M\llbracket \sigma_1^m\rrbracket_{\Enh}$ the enhanced words typically change with
$$
1\dots 1 0^{\bx}0^{\bx}0^{\bx}\dots 0^{\bx}0 \leadsto 1\dots 1 110^{\bx}\dots 0^{\bx}0 
$$
which means that the internal degree shift goes up by $4$. In comparison, when going up 2 homological degrees in $M\llbracket (\sigma_1\sigma_2)^k\rrbracket_{\Enh}$ the enhanced words are changed for example with 
$$
1\dots 1 00^{\bx}100^{\bx}1\dots 00^{\bx}10 \leadsto 1\dots 111100^{\bx}1\dots 00^{\bx}10 
$$
which means that the internal degree shift goes up by  only $3$. This difference makes is so that there will be uniform bounds on the indices of  $\Z$-modules relevant to calculating any given homology group $\mathscr H^{i,j} D((\sigma_1\sigma_2)^k, \sigma_1^m, T)$.  As previously, this will allow us to switch complexes with their truncated complexes in the homology calculations. 

A minor problem is that the internal grading is not strictly additive in the way that homological degree is.  There is some ``error" as new circles are formed in between the inputs. Luckily this error can be controlled by $s(D)$. By examining the gradings of $M\llbracket (\sigma_1\sigma_2)^k\rrbracket_{\Enh}$ we notice that we can write 
\begin{align*}
    M(\llbracket (\sigma_1\sigma_2)^k\rrbracket_{\Enh})^a=\bigoplus_{p\in I_a} P_p\left\{d\right\} %\label{internal gradings for T3 braids}
\end{align*}
where every $I_a$ is an index set, $P_p$ is a planar matching of six points on a circle and $3a/2-2k\leq d \leq 3a/2 -2k+1$. Similarly we can write
\begin{align*}
(M\llbracket(\sigma^m)\rrbracket_{\Enh})^b=\bigoplus_{q\in J_b}Q_q\{e\} %\label{internal gradings for T2 braids}
\end{align*}
where every $J_b$ is an index set , $Q_q$ is a planar matching of four points on a circle and $2b-m\leq e\leq 2b-m+1$.

\begin{theorem}\label{Main theorem for 3-input diagrams}
    Let $D$ be a 3-input planar arc diagram without outer boundary and let $T$ be a tangle, so that $D((\sigma_1\sigma_2), \sigma_1, T)$ makes sense. Let $\mathcal{C}$ be a complex with $\llbracket T \rrbracket \simeq \mathcal{C}$ and  define integer valued functions
     \begin{align*}
         v_1(i,j,k,m,D, \mathcal{C})&=- 3 i + 2 j + 4 k + 3 m- 2 s(D) + 3 h_{\min}(\mathcal{C}) - 2 q_{\max}(\mathcal{C}) -7 \\
         v_2(i,j,k,m,D, \mathcal{C})&=-3 i + 2 j + 4 k + 2 m + 2 s(D) + 3 h_{\max}(\mathcal{C}) - 2 q_{\min}(\mathcal{C}) +4.
     \end{align*}
     Then for all $i,j\in \Z$, $k,m\in \Z_{\geq 0}$ such that $v_1(i,j,k,m,D, \mathcal{C})\geq 0$ there is an isomorphism of Khovanov homology groups:
    \begin{align}
    \mathscr{H}^{i,j} D((\sigma_1\sigma_2)^k,\sigma_1^{m},T)\cong \mathscr{H}^{i,j-2} D((\sigma_1\sigma_2)^k,\sigma_1^{m+2},T).     \label{v_1 isomorphism}
    \end{align}
    Similarly for all $i,j\in \Z$, $k,m\in \Z_\geq 0$ such that $v_2(i,j,k,m,D, \mathcal{C})\leq 0$ there is an isomorphism of Khovanov homology groups
    \begin{align}
    \mathscr{H}^{i,j} D((\sigma_1\sigma_2)^k,\sigma_1^{m},T)\cong \mathscr{H}^{i-2,j-6} D((\sigma_1\sigma_2)^k,\sigma_1^{m+2},T).  \label{v_2 isomorphism}
    \end{align}
\end{theorem}

\begin{proof}
    We will carry out the investigation in the homotopic complex
    \begin{align*}
        \llbracket D((\sigma_1\sigma_2)^k,\sigma_1^m, T) \rrbracket &\simeq D(M\llbracket (\sigma_1\sigma_2)^k\rrbracket_{\Enh}, M\llbracket\sigma_1^m\rrbracket_{\Enh}, \Psi\mathcal{C}). 
    \end{align*}
    With some index sets $(K_c)_{c\in \Z}$, planar matchings $R_r$ and integers $f_r$ the chain spaces of the complex $\Psi\mathcal{C}$ can be written as
    $$
    (\Psi\mathcal{C})^c=\bigoplus_{r\in K_c} R_r\{f_r\}.
    $$    
    We can use the linearity of planar algebras to expand direct sums, that is, for $a,b,c\in \Z$ we have
    \begin{align*}
        &D((M\llbracket (\sigma_1\sigma_2)^k)\rrbracket_{\Enh})^a, (M\llbracket\sigma_1^m\rrbracket_{\Enh})^b, (\Psi\mathcal{C})^c)\\
        \cong& D\left(\bigoplus_{p\in I_a} P_p\{d_a\},\bigoplus_{q\in J_b} Q_q\{e_b\},\bigoplus_{r\in K_c} R_r\{f_r\}\right)\\
        \cong&\bigoplus_{p\in I_a, \ q\in J_b \ t\in K_c} D(P_p,Q_q,R_r) \{d_a+e_b+f_r\}.
    \end{align*} 
    By plugging this once more into the delooping isomorphism $\Psi$, we get 
    $$
    D((M\llbracket (\sigma_1\sigma_2)^m)\rrbracket_{\Enh})^a, (M\llbracket\sigma_1^m\rrbracket_{\Enh})^b, (\Psi\mathcal{C})^c)\cong \bigoplus_{
    \substack{
          p\in I_a, \ q\in J_b \\ 
          r\in K_c, \ l\in L_{p,q,t}           
    }
    %r\in I_a\ s\in J_c\  t\in K_{r,b,s}
    }\emptyset \{d_a+e_b+f_r+g_l\}
    $$
    where $L_{p,q,r}$ is the set of functions from the set of circles in $D(P_p,Q_q,R_r)$ to $\{-1,1\}$ and $g_l=\int l \, d \mu$ with $\mu$ being the counting measure. A straightforward counting argument shows that we have $|g_l|\leq s(D) $ for all $l$.

    Denote $\mathcal{F}^{i,j}$ as the composition of $\mathcal{F}$ with the projection to $(i,j)$-graded $\Z$-module. Since $\mathcal{F}$ is additive, it follows that 
    \begin{align*}
        &\mathcal{F}^{i,j} D((M\llbracket (\sigma_1\sigma_2)^k)\rrbracket_{\Enh})^a, (M\llbracket\sigma_1^m\rrbracket_{\Enh})^b, (\Psi\mathcal{C})^c)\not\cong 0\\
        \iff & \mathcal{F}^{i,j} \bigoplus_{
    \substack{
          p\in I_a, \ q\in J_b \\ 
          r\in K_c, \ l\in L_{p,q,r}           
    }
    }\emptyset \{d_a+e_b+f_r+g_l\}\not\cong 0 \\
        \iff & \Bigg( 
        \bigoplus_{
    \substack{
          p\in I_a, \ q\in J_b \\ 
          r\in K_c, \ l\in L_{p,q,r}           
    }
    } 
        \Z [a+b+c]\{d_a+e_b+f_r+g_l\}\Bigg)^{i,j}\not\cong 0  \\
        \iff & \begin{cases}
        i=a+b+c &\\
        j=d_a+e_b+f_r+g_l & \text{for some } 
             p\in I_a, \ q\in J_b, \ 
             r\in K_c,\ l\in L_{p,q,r}.
    \end{cases}         
    \end{align*}
    If $i,j,k,m,a,b,c$ are such that the above conditions are met, then there exists $d,e,f,g\in\Z$ such that the following equations hold:    
    \begin{align*}
        i&=a+b+c \\
        j&=d+e+f+g \\
        \frac{3}{2}a-2k & \leq d \leq \frac{3}{2}a-2k +1 \\
        2b-m & \leq e \leq  2b-m+1 \\
        h_{\min}( \mathcal{C})&\leq c \leq h_{\max}(\mathcal{C}) \\ 
        q_{\min}(\mathcal{C}) &\leq f \leq q_{\max}(\mathcal{C})\\
        -s(D) &\leq g \leq s(D).
    \end{align*}
    We can treat this as a convex optimisation problem over the reals and find the minimum and maximum of projection to $b$. With the help of a computer algebra system Matematica we can show that the above equations imply $u_1(i,j,k,m)\leq b \leq u_2(i,j,k,m)$ where
    \begin{align*}
        &u_1(i,j,k,m)= - 3 i + 2 j + 4 k + 2 m- 2 s(D) + 3 h_{\min}( \mathcal{C}) - 2 q_{\max}(\mathcal{C}) -4\\
        &u_2(i,j,k,m)= -3 i + 2 j + 4 k + 2 m + 2 s(D) + 3 h_{\max}( \mathcal{C}) - 2 q_{\min}(\mathcal{C}) .
    \end{align*}
    The contrapositive statement of the above is that given $i,j,k,m,a,b,c$ with $b<u_1(i,j,k,m)$ or $u_2(i,j,k,m)<b$ it follows that 
    \begin{align}
    \mathcal{F}^{i,j} D((M\llbracket (\sigma_1\sigma_2)^k)\rrbracket_{\Enh})^a, (M\llbracket\sigma_1^m\rrbracket_{\Enh})^b, (\Psi\mathcal{C})^c)\cong 0. \label{condition on FD being a zero-module}
    \end{align}
    
    The relevant modules for calculating $(i,j)$-homology are non-zero modules
    $$
    \mathcal{F}^{i+h,j} D((M\llbracket (\sigma_1\sigma_2)^k)\rrbracket_{\Enh})^a, (M\llbracket\sigma_1^m\rrbracket_{\Enh})^b, (\Psi\mathcal{C})^c)
    $$
    where $h\in \{-1,0,1\}$. By Equation \ref{condition on FD being a zero-module} we can restrict to modules with 
    $$
    b\geq \min_{h\in\{-1,0,1\}} u_1(i+h,j,k,m)=v_1(i,j,k,m,D,\mathcal{C})-m.
    $$
    Therefore, whenever $v_1(i,j,k,m,D,\mathcal{C})\geq 0$, it follows that
    \begin{align*}
    &H^{i,j}\mathcal{F}D(M\llbracket (\sigma_1\sigma_2)^k\rrbracket_{\Enh}, M\llbracket\sigma_1^m\rrbracket_{\Enh}, \Psi\mathcal{C}) \\
    \cong &H^{i,j}\mathcal{F}D(M\llbracket (\sigma_1\sigma_2)^k\rrbracket_{\Enh}, \tau^{\geq -m}M\llbracket\sigma_1^m\rrbracket_{\Enh}, \Psi\mathcal{C}). 
    \end{align*}
    The same bound applies for interchanging $M\llbracket\sigma^{m+2}\rrbracket_{\Enh}$ with its truncation and therefore given $v_1(i,j,k,m,D,\mathcal{C})\geq 0$ we can conclude
    \begin{align*}
    &\mathscr{H}^{i,j} D((\sigma_1\sigma_2)^k,\sigma_1^{m},T) \\
    \cong &H^{i,j}\mathcal{F}D(M\llbracket (\sigma_1\sigma_2)^k\rrbracket_{\Enh}, M\llbracket\sigma_1^m\rrbracket_{\Enh}, \Psi\mathcal{C}) \\
    \cong &H^{i,j}\mathcal{F}D(M\llbracket (\sigma_1\sigma_2)^k\rrbracket_{\Enh}, \tau^{\geq -m}M\llbracket\sigma_1^m\rrbracket_{\Enh}, \Psi\mathcal{C}) \\
    \cong &H^{i,j}\mathcal{F}D(M\llbracket (\sigma_1\sigma_2)^k\rrbracket_{\Enh},\tau^{\geq -m}M\llbracket\sigma_1^{m+2}\rrbracket_{\Enh}\{2\}, \Psi\mathcal{C}) \\
    \cong &H^{i,j-2}\mathcal{F}D(M\llbracket (\sigma_1\sigma_2)^k\rrbracket_{\Enh}, M\llbracket\sigma_1^{m+2}\rrbracket_{\Enh}, \Psi\mathcal{C})\\
    \cong&\mathscr{H}^{i,j-2} D((\sigma_1\sigma_2)^k,\sigma_1^{m+2},T).
    \end{align*}
    The second claim  is proven similarly; then the modules with 
    $$
    b\leq \max_{h\in\{-1,0,1\}} u_2(i+h,j,k,m)=v_2(i,j,k,m,D,\mathcal{C})-1
    $$
    can be ignored, since those will be zero-modules.        
\end{proof}

    \begin{comment}
    Bounds for dual statements
    b_min=-2 - 3 iii + 2 jjj + 4 k - 2 m + 3 o1 - 2 p2 - 2 q
    b_max= 2 - 3 iii + 2 jjj + 4 k - 2 m + 3 o2 - 2 p1 + 2 q
    \end{comment}

\subsection{Duality and algorithmic reduction of full twists}

So far we have limited our study of complexes $\llbracket(\sigma_1\sigma_2)^k \rrbracket $ and $\llbracket\sigma_1^m \rrbracket $ to non-negative $k$ and $m$. Instead of repeating this investigation for negative parameters from the ground up, we will make use of duality of  Khovanov complexes. There is a canonical contravariant functor $V$ from the category $\Cob(2b)$ to itself, which acts as the identity on objects and turns the morphisms upside down. It is quick to see that $V$ extends to an additive contravariant functor in $\Mat(\Cob(2b))$ and it can be further extended to the category of complexes $\Kom(\Mat(\Cob(2b)))$ with minor reindexing: if $\mathcal{C} \in \Kom(\Mat(\Cob(2b)))$, then we set $(V\mathcal{C})^i=\mathcal{C}^{-i}$. Furthermore, we flip signs of all internal grading shifts: $V(A\{q\})=A\{-q\}$. The following properties of $V$ are straightforward to deduce: 

    \begin{itemize}
        \item For all complexes $\C$, indices $a,q,r\in \Z$ 
        $$
        V\tau^{\leq a} \mathcal{C}=\tau^{\geq -a}V\mathcal{C} \quad \text{and} \quad V(\mathcal{C}[r]\{q\})=(V\mathcal{C})[-r]\{-q\}.
        $$
        \item If $\mathcal{C}\simeq \mathcal{D} $, then $V\mathcal{C}\simeq V \mathcal{D} $.
        \item Let $T$ be a tangle diagram and $T^!$ its mirror image. Then $V\llbracket T \rrbracket =\llbracket T^! \rrbracket$.
    \end{itemize}
From these properties one can derive the following dual proposition.% \ref{Dual proposition of snip-complex isomorphisms}. 

\begin{proposition}[Dual of Proposition \ref{isomorphisms  of snip-complexes}] \label{Dual proposition of snip-complex isomorphisms}
    For any $m\geq 0$ there is a homotopy equivalence $\llbracket \sigma_1^{-m}\rrbracket \simeq VM\llbracket \sigma_1^m\rrbracket_{\Enh}$ and isomorphisms of  truncated complexes:
    \begin{align*}
        \tau^{\leq m}VM\llbracket\sigma_1^{m}\rrbracket_{\Enh}&\cong \tau^{\leq m}((VM\llbracket\sigma_1^{m+2}\rrbracket_{\Enh})\{-2\}) \\
        \tau^{\geq 1}VM\llbracket\sigma_1^{m}\rrbracket_{\Enh}&\cong \tau^{\geq 1}((VM\llbracket\sigma_1^{m+2}\rrbracket_{\Enh})[-2]\{-6\}) .      
    \end{align*}
\end{proposition}

The next dual theorem follows from Propositions \ref{isomorphisms  of snip-complexes} and \ref{Dual proposition of snip-complex isomorphisms} analogously to how Theorem \ref{Main theorem for 3-input diagrams} followed from Proposition \ref{isomorphisms  of snip-complexes}.  

\begin{theorem}[Dual of Theorem \ref{Main theorem for 3-input diagrams}] \label{dual 3 input}
      Let $D$ be a 3-input planar arc diagram  without outer boundary and let $T$  be a tangle, so that $D((\sigma_1\sigma_2), \sigma_1, T)$ makes sense. Let $\mathcal{C}$ be a complex with $\llbracket T \rrbracket \simeq \mathcal{C}$ and  define integer valued functions
     \begin{align*}
         & w_1(i,j,k,m,D, \mathcal{C})=- 3 i + 2 j + 4 k + 2 m- 2 s(D) + 3 h_{\min}(\mathcal{C}) - 2 q_{\max}(\mathcal{C}) -6 \\
         & w_2(i,j,k,m,D, \mathcal{C})=- 3 i + 2 j + 4 k + 3 m+ 2 s(D) + 3 h_{\max}(\mathcal{C}) - q_{\min}(\mathcal{C})  +5.          
     \end{align*}
     Then for all $j,i\in \Z$, $k\in \Z_{\geq 0}$, $m\in \Z_{\leq 0}$ such that $w_1(i,j,k,m,D, \mathcal{C})\geq 0$ there is an isomorphism of Khovanov homology groups:
    \begin{align}
    \mathscr{H}^{i,j} D((\sigma_1\sigma_2)^k,\sigma_1^{m},T)\cong \mathscr{H}^{i+2,j+6} D((\sigma_1\sigma_2)^k,\sigma_1^{m-2},T). \label{w_1 isomorphism}    
    \end{align}
    Similarly for all $j,i\in \Z$, $k\in \Z_{\geq 0}$, $m\in \Z_{\leq 0}$ such that $w_2(i,j,k,m,D, \mathcal{C})\leq 0$ there is an isomorphism of Khovanov homology groups
    \begin{align}
    \mathscr{H}^{i,j} D((\sigma_1\sigma_2)^k,\sigma_1^{m},T)\cong \mathscr{H}^{i,j+2} D((\sigma_1\sigma_2)^k,\sigma_1^{m-2},T). \label{w_2 isomorphism}    
    \end{align}

\end{theorem}

We need a quick lemma to relate the homology groups of a link and its mirror image. 
\begin{lemma}\label{Khovanov homology of mirror image}
    Let $L_1$ and $L_2$ be links and suppose 
    $
    \mathscr{H}^{i_1+t,j_1}(L_1)\cong\mathscr{H}^{i_2+t,j_2}(L_2)
    $
    for $t=0$ and $t=1$. Then
    $
    \mathscr{H}^{-i_1,-j_1}(L_1^!)\cong \mathscr{H}^{-i_2,-j_2}(L_2^!).
    $
\end{lemma}
\begin{proof}
    Use duality of Khovanov complexes, see Corollary 11 of \cite{khovanov1999categorification}.
\end{proof}

We are now ready to combine our previous theorems and restate the main theorem in the following detailed and extended form.

\begin{theorem}\label{algorithm theorem}
     Let $D$ be  a 3-input planar arc diagram  without outer boundary and let  $T$ be a tangle, so that $D((\sigma_1\sigma_2), \sigma_1, T)$ makes sense. For all $a,b\in \Z_{\geq 0}$ there exist finite index sets $K,M\subset \Z$ such that for all $k_1,m_1\in \Z$ there exist $k_2\in K$ and $m_2\in M$  so that for all $i_1,j_1\in \Z$ there exist $i_2,j_2\in \Z$ such that for all $0\leq c\leq a$ and $0\leq d\leq b$ we have an isomorphism of Khovanov homology groups:
     $$
     \mathscr{H}^{i_1+c,j_1+d}(D((\sigma_1\sigma_2)^{k_1}, \sigma_1^{m_1}, T))\cong \mathscr{H}^{i_2+c,j_2+d}(D((\sigma_1\sigma_2)^{k_2}, \sigma_1^{m_2}, T)).
     $$
     Moreover, $k_2,m_2,i_2,j_2$ can be obtained from a linear time algorithm, and the sets $K$ and $M$ can be written explicitly. 
\end{theorem}
\begin{proof}
    We prove the claim by using Theorems \ref{isomorphisms of Khovanov homology groups with 2 input diagrams}, \ref{Main theorem for 3-input diagrams} and \ref{dual 3 input} in Algorithm \ref{isomorphism algorithm}. First we define two operations, flip and merge, on tangle and 3-input planar arc diagrams and illustrate them in Figure \ref{figure for merge and flip}. Formally flip takes in a diagram and outputs a diagram and merge takes in a 3-input planar arc diagram $D$ and two of its inputs and it outputs a 2-input planar arc diagram and a compatible tangle. We shall also define an integer valued function $g$ by
    $$
    g(a,b,D,\mathcal{C})=4s(D)+3(h_{\max}(\mathcal{C})-h_{\min}(\mathcal{C}))+2(q_{\max}(\mathcal{C})-q_{\min}(\mathcal{C}))+3a+2b+14   
    $$
    where $a,b$ are integers, $D$ is a planar arc diagram and  $\mathcal{C}$ is a complex. Similarly we define real valued function $h$ by 
    $$
    h(a,\mathcal{C})=\frac{3}{4}(h_{\max}(\mathcal{C})-h_{\min}(\mathcal{C})+a)+8.
    $$

    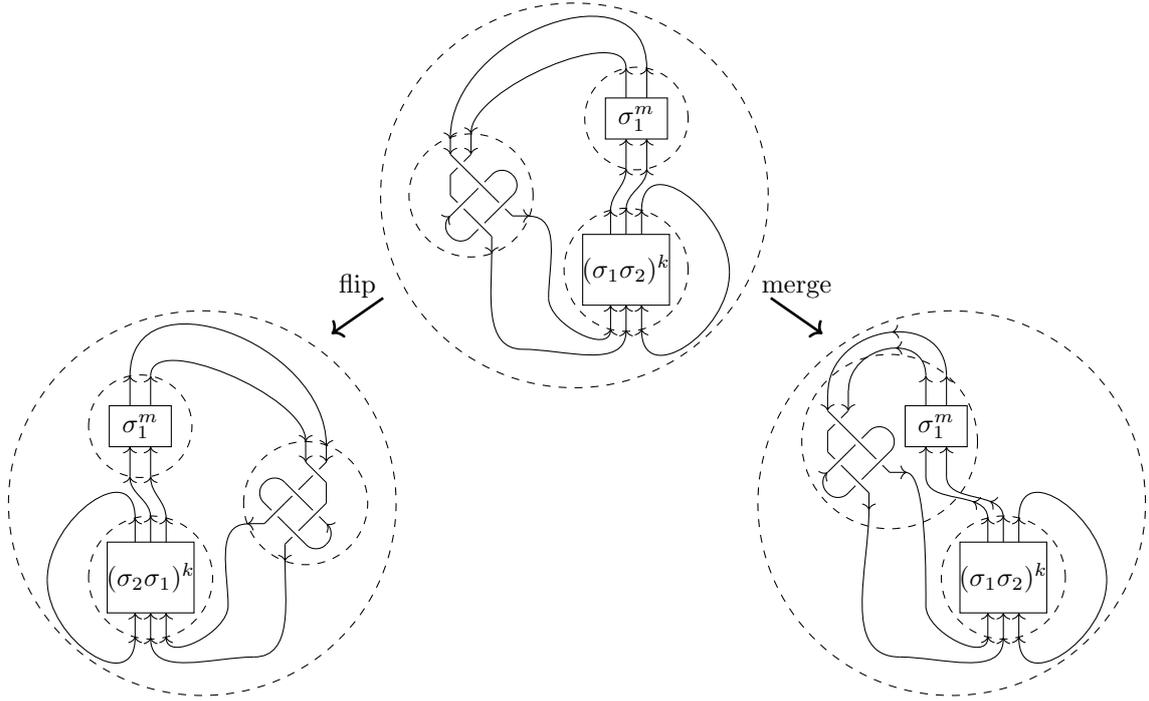
\begin{figure}
        \centering
        
\begin{tikzpicture}[
  arrowmark/.style 2 args={
    postaction=decorate,
    decoration={
      markings,
      mark=at position #1 with {\arrow{#2}}
    }
  },scale=0.68
]
%%%%%%%%%%%%%%%%%%%%%%%%%%%%%%%%%%%%%%%%%%%%%%%%% Original
  \begin{scope}[shift={(-7.3,6)}]
      
    \draw[dashed] (0,0) circle (3.75);

      \begin{scope}[shift={(1,1.5)},scale=0.4]
        
  \draw[dashed] (0.5,0) circle (2.5);

  \draw (-1, -1) rectangle (2, 1);

    \node at (0.5,0) {$\sigma_1^m$};

  % Arrows from the origin to nodes
  \draw[->] (1,1) -- (1,2.46);
  \draw[->] (0,1) -- (0,2.46);
  \draw[<-] (1,-1) -- (1,-2.46);
  \draw[<-] (0,-1) -- (0,-2.46);
  %\draw[->] (0,0) -- (-1,3.87);
  %\draw[->] (0,0) -- (1,-3.87);
  %\draw[->] (0,0) -- (-1,-3.87);

  \end{scope}

  \begin{scope}[shift={(-2,-2.5)},scale=0.3]
    \draw[dashed] (10, 3.5) circle (4);

    \draw (12.8, 5.8) rectangle (7.2, 1.2);

    \node at (10,3.5) {$(\sigma_1\sigma_2)^k$};

    \draw[<-] (9,1.2) -- (9,-0.38);
    \draw[<-] (11,1.2) -- (11,-0.38);
    
    \draw[->] (9,5.8) -- (9,7.38);
    \draw[->] (11,5.8) -- (11,7.38);

    \draw[<-] (10,1.2) -- (10,-0.5);
    \draw[->] (10,5.8) -- (10,7.5);
  \end{scope}

      \begin{scope}[shift={(-2,0)}, scale=0.4]
    %\draw[dashed] (0,0) circle (1);
    
    \draw[dashed] (0,0) circle (3);

    \negCrossing{0,-2}
    \posCrossing{1,-1}
    \posCrossing{-1,-1}
    \negCrossing{0,0}

    \lineUpBendLeftWArrow{-1,-2}
    \lineRightBendDown{-1,-2}
    \lineRightBendUp{1,1}
    \lineUpBendRight{2,0}

    \lineUp{-1,0}
    \negCrossing{-1,1}

    \draw[<-] (1,-2.84) -- (1,-2);
    
    \draw[<-] (2.84,-1) -- (2,-1);

    \draw[<-] (0,2) -- (0,3);
    \draw[<-] (-1,2) -- (-1,2.84);

    \end{scope}

    \draw[<-, looseness=1, out=270, in=270] (0.7,-2.6) to (-0.5,-2);
    \draw[-, looseness=1, out=90, in=0] (-0.5,-2) to (-0.88,-0.405);
    
    \draw[->, looseness=1, out=90, in=270] (0.7,-0.3) to (1,0.52);

    \draw[->, looseness=1, out=90, in=270] (1,-0.28) to (1.4,0.52);

    \draw[-, looseness=1, out=270, in=180] (-1.6,-1.12) to (-1,-3);
    \draw[->, looseness=1, out=0, in=270] (-1,-3) to (1,-2.62);

    \draw[-, looseness=1.5, out=90, in=90] (1.3,-0.3) to (3,-1.5);
    \draw[->, looseness=1.5, out=270, in=270] (3,-1.5) to (1.3,-2.6);

    \draw[->, looseness=0.7, out=90, in=90] (1,2.48) to (-2,1.2);
    
    \draw[->, looseness=1.3, out=90, in=90] (1.4,2.48) to (-2.4,1.1);

  \end{scope}
%%%%%%%%%%%%%%%%%%%%%%%%%%%%%%%%%%%%%%%%%%%%% merged
\begin{scope}[shift={(0,0)}]

  \begin{scope}[shift={(-2,-2.5)},scale=0.3]
    \draw[dashed] (10, 3.5) circle (4);

    \draw (12.8, 5.8) rectangle (7.2, 1.2);

    \node at (10,3.5) {$(\sigma_1\sigma_2)^k$};

    \draw[<-] (9,1.2) -- (9,-0.38);
    \draw[<-] (11,1.2) -- (11,-0.38);
    
    \draw[->] (9,5.8) -- (9,7.38);
    \draw[->] (11,5.8) -- (11,7.38);

    \draw[<-] (10,1.2) -- (10,-0.5);
    \draw[->] (10,5.8) -- (10,7.5);
  \end{scope}

    \draw[dashed] (0,0) circle (3.75);

    %%sigma^m  tangle
    
   \begin{scope}[shift={(-0.5,1.5)},scale=0.4]
        
  %\draw[dashed] (0.5,0) circle (2.5);

  \draw (-1, -1) rectangle (2, 1);

    \node at (0.5,0) {$\sigma_1^m$};

  % Arrows from the origin to nodes
  \draw[->] (1,1) -- (1,2.46);
  \draw[->] (0,1) -- (0,2.46);
  \draw[<-] (1,-1) -- (1,-2.46);
  \draw[<-] (0,-1) -- (0,-2.46);
  %\draw[->] (0,0) -- (-1,3.87);
  %\draw[->] (0,0) -- (1,-3.87);
  %\draw[->] (0,0) -- (-1,-3.87);

  \end{scope}

    % wierd T tangle

      \begin{scope}[shift={(-2,1)}, scale=0.4]
    %\draw[dashed] (0,0) circle (1);
    
    %\draw[dashed] (0,0) circle (3);

    \negCrossing{0,-2}
    \posCrossing{1,-1}
    \posCrossing{-1,-1}
    \negCrossing{0,0}

    \lineUpBendLeftWArrow{-1,-2}
    \lineRightBendDown{-1,-2}
    \lineRightBendUp{1,1}
    \lineUpBendRight{2,0}

    \lineUp{-1,0}
    \negCrossing{-1,1}

    \draw[<-] (1,-2.84) -- (1,-2);
    
    \draw[<-] (2.84,-1) -- (2,-1);

    \draw[<-] (0,2) -- (0,3);
    \draw[<-] (-1,2) -- (-1,2.84);

    \end{scope}

    \draw[dashed] (-1.2,1.2) circle (1.7);

    \draw[<-, looseness=1, out=270, in=270] (0.7,-2.6) to (-0.5,-2);
    \draw[-, looseness=0.4, out=90, in=0] (-0.5,-2) to (-0.88,1-0.405);
    
    \draw[arrowmark={0.3}{>}, looseness=1, out=90, in=270] (0.7,-0.3) to (-0.5,0.52);

    \draw[arrowmark={0.3}{>}, looseness=1, out=90, in=270] (1,-0.28) to (-0.1,0.52);

    \draw[-, looseness=1, out=270, in=180] (-1.6,-0.12) to (-1,-3);
    \draw[->, looseness=1, out=0, in=270] (-1,-3) to (1,-2.62);

    \draw[-, looseness=1.5, out=90, in=90] (1.3,-0.3) to (3,-1.5);
    \draw[->, looseness=1.5, out=270, in=270] (3,-1.5) to (1.3,-2.6);

    \draw[arrowmark={0.4}{>}, looseness=1.5, out=90, in=90] (-0.5,2.48) to (-1.995,2.2);
    
    \draw[arrowmark={0.45}{>}, looseness=1.5, out=90, in=90] (-0.1,2.48) to (-2.395,2.13);

    %\fill[white] (-1.2,1.2) circle (1.69);

\end{scope}
%%%%%%%%%%%%%%%%%%%%%%%%%%%%%%%%%%%%%%%%%%%%%%%% flipped   
\begin{scope}[shift={(-10,7)}]
    
\begin{scope}[shift={(-4.5,-7)},xscale=-1]
  \draw[dashed] (0,0) circle (3.75);

      \begin{scope}[shift={(1,1.5)},scale=0.4]
        
  \draw[dashed] (0.5,0) circle (2.5);

  \end{scope}
  \begin{scope}[shift={(-2,-2.5)},scale=0.3]
      \draw[dashed] (10, 3.5) circle (4);

  \end{scope}  

      \begin{scope}[shift={(-2,0)}, scale=0.4]
    %\draw[dashed] (0,0) circle (1);
    
    \draw[dashed] (0,0) circle (3);
    \end{scope}

    \draw[<-, looseness=1, out=270, in=270] (0.7,-2.6) to (-0.5,-2);
    \draw[-, looseness=1, out=90, in=0] (-0.5,-2) to (-0.88,-0.405);
    
    \draw[->, looseness=1, out=90, in=270] (0.7,-0.3) to (1,0.52);

    \draw[->, looseness=1, out=90, in=270] (1,-0.28) to (1.4,0.52);

    \draw[-, looseness=1, out=270, in=180] (-1.6,-1.12) to (-1,-3);
    \draw[->, looseness=1, out=0, in=270] (-1,-3) to (1,-2.62);

    \draw[-, looseness=1.5, out=90, in=90] (1.3,-0.3) to (3,-1.5);
    \draw[->, looseness=1.5, out=270, in=270] (3,-1.5) to (1.3,-2.6);

    \draw[->, looseness=0.7, out=90, in=90] (1,2.48) to (-2,1.2);
    
    \draw[->, looseness=1.3, out=90, in=90] (1.4,2.48) to (-2.4,1.1);

\end{scope}

    \begin{scope}[shift={(-2.5,-7)}, xscale=-1, scale=0.4]%, scale=0.5]
    %\draw[dashed] (0,0) circle (1);
    
    %\draw[dashed] (0,0) circle (3);

    \posCrossing{0,-2}
    \negCrossing{1,-1}
    \negCrossing{-1,-1}
    \posCrossing{0,0}

    \lineUpBendLeftWArrow{-1,-2}
    \lineRightBendDown{-1,-2}
    \lineRightBendUp{1,1}
    \lineUpBendRight{2,0}

    \lineUp{-1,0}
    \posCrossing{-1,1}

    \draw[<-] (1,-2.84) -- (1,-2);
    
    \draw[<-] (2.84,-1) -- (2,-1);

    \draw[<-] (0,2) -- (0,3);
    \draw[<-] (-1,2) -- (-1,2.84);

    \end{scope}

    \begin{scope}[shift={(-5.9,-5.5)},scale=0.4]
        
  %\draw[dashed] (0.5,0) circle (2.5);

  \draw (-1, -1) rectangle (2, 1);

    \node at (0.5,0) {$\sigma_1^m$};

  % Arrows from the origin to nodes
  \draw[->] (1,1) -- (1,2.46);
  \draw[->] (0,1) -- (0,2.46);
  \draw[<-] (1,-1) -- (1,-2.46);
  \draw[<-] (0,-1) -- (0,-2.46);
  %\draw[->] (0,0) -- (-1,3.87);
  %\draw[->] (0,0) -- (1,-3.87);
  %\draw[->] (0,0) -- (-1,-3.87);

  \end{scope}

      \begin{scope}[shift={(-8.5,-9.5)},scale=0.3]
    %\draw[dashed] (10, 3.5) circle (4);

    \draw (12.8, 5.8) rectangle (7.2, 1.2);

    \node at (10,3.5) {$(\sigma_2\sigma_1)^k$};

    \draw[<-] (9,1.2) -- (9,-0.38);
    \draw[<-] (11,1.2) -- (11,-0.38);
    
    \draw[->] (9,5.8) -- (9,7.38);
    \draw[->] (11,5.8) -- (11,7.38);

    \draw[<-] (10,1.2) -- (10,-0.5);
    \draw[->] (10,5.8) -- (10,7.5);
  \end{scope}

\end{scope}

  \draw[->, line width=1 pt] (-11,4) -- (-12,3.3) node[midway, above=3pt] {flip};

  \draw[->, line width=1 pt] (-3.5,4) -- (-2.5,3.3) node[midway, above=3pt] {merge};
    
\end{tikzpicture}

        \caption{Link diagram $D((\sigma_1\sigma_2)^k,\sigma_1^m, T)$ (top), and link diagram $D'((\sigma_1\sigma_2)^k,T')$  (bottom right) where $(D',T')=\operatorname{merge}(D,\sigma_1^m,T)$. 
        Link diagram $\operatorname{flip}(D((\sigma_1\sigma_2)^k,\sigma_1^m,T))$ (bottom left) which can be also written  as $(\operatorname{flip}D)(\operatorname{flip}(\sigma_1\sigma_2)^k, \operatorname{flip}\sigma_1^m, \operatorname{flip} T)$.
        %3-input planar arc diagram $\operatorname{flip}(D)$ (bottom) to which tangle diagram $\operatorname{flip}(T)$ is inserted.
        }
        \label{figure for merge and flip}
    \end{figure}
    
    Given $a,b,D,T$, the index sets $K$ and $M$ are defined by
    \begin{align*}
        M&=\{t\in \Z : |t| < g(a+1,b,D,\Psi \llbracket T\rrbracket ) \} \\ 
        K&=\{t\in \Z : h(a+1, \mathcal{C'} )<t < h(a, \mathcal{C'} )\}
    \end{align*}
     where $\mathcal{C'}$ is obtained by running lines \ref{code line merge 1} and \ref{code line merge 2} of Algorithm \ref{isomorphism algorithm} with $m=\max M$. When further provided $i_1,j_1, k_1, m_1$, the desired $i_2,j_2,k_2,m_2$ are given by Algorithm \ref{isomorphism algorithm}, which terminates for all inputs. A careful observation of the code shows that $k_2$ and $m_2$ do not depend on $i_1$ and $j_1$. To prove that the algorithm gives a valid result for all $k_1\geq 0$, we need to check that the requirements of Theorems \ref{isomorphisms of Khovanov homology groups with 2 input diagrams}, \ref{Main theorem for 3-input diagrams} and \ref{dual 3 input} are met at each instance where they are applied. To see this, it is enough to show the following straightforward implications
\begin{align*}
    m\geq  g(a,b,D,\mathcal{C}) \implies& \left(\begin{array}{c}
         \min_{c,d} v_1(i+c,j+d+2,k,m-2,D,\mathcal{C})\geq 0  \\
         \text{or}\\
         \max_{c,d} v_2(i+c+2,j+d+6,k,m-2,D,\mathcal{C})\leq 0
    \end{array} \right)\\
    \allowdisplaybreaks
    \\
    m\leq - g(a,b,D,\mathcal{C})\implies & \left(\begin{array}{c}
         \min_{c,d} w_1(i+c-2,j+d-6,k,m+2,D,\mathcal{C})\geq 0  \\
         \text{or}\\
         \max_{c,d} w_2(i+c,j+d-2,k,m+2,D,\mathcal{C})\leq 0
    \end{array} \right)\\
    \allowdisplaybreaks
    \\
    k\geq h(a,\mathcal{C'}) \implies & \left(\begin{array}{c}
         \min_{c,d} (i+c)>h_{\max}(\mathcal{C}') -\lfloor \frac{4(k-3)}{3} \rfloor  \\
         \text{or}\\
         \max_{c,d} (i+c+4)<h_{\min}(\mathcal{C}')-1
    \end{array} \right)
\end{align*}
     where minima and maxima are taken over $0\leq c \leq a$, $0\leq d \leq b$.
    %All three of these are routine to verify. %It is a direct consequence of Theorems \ref{isomorphisms of Khovanov homology groups with 2 input diagrams}, \ref{Main theorem for 3-input diagrams} and \ref{dual 3 input} that whenever $k_1\geq 0$ the algorithm produces a valid result, i.e. the tuple of integers $(i_2,j_2,k_2,m_2)$ meet the requirements of the theorem. 
    
Next we assume $k_1< 0$. Denoting $D'=\operatorname{flip}(D)$ and $T'=\operatorname{flip} (T^!) $ we can observe an isotopy of link diagrams for all $k,m$:
$$
D((\sigma_1\sigma_2)^k, \sigma_1^m, T)^!\cong D((\sigma_2\sigma_1)^{-k}, \sigma_1^{-m}, T^!)\cong D'((\sigma_1\sigma_2)^{-k}, \sigma_1^{-m}, T').
$$
Suppose $(i',j',k',m')$ is a tuple of integers spit out by running the algorithm with $(-i_1-a,-j_1-b,-k_1,-m_1,a+1,b, D',T') $. Since the algorithm gives a valid result for $-k_1>0$, there are $i',j',k',m'$ such that for all $0\leq c\leq a+1$ and $0\leq d \leq b$ we have
\begin{align*}
    &\mathscr H^{-i_1-a+c,-j_1-b+d}D((\sigma_1\sigma_2)^{k_1}, \sigma_1^{m_1}, T)^! \\
    \cong &\mathscr H^{-i_1-a+c,-j_1-b+d}D'((\sigma_1\sigma_2)^{-k_1}, \sigma_1^{-m_1}, T') \\
    \cong &\mathscr H^{i'+c,j'+d}D'((\sigma_1\sigma_2)^{k'}, \sigma_1^{m'}, T') \\
    \cong &\mathscr H^{i'+c,j'+d}D'((\sigma_1\sigma_2)^{-k'}, \sigma_1^{-m'}, T)^!. 
\end{align*}
By changing the summation index we can see this to be equivalent to the fact that for all $-1\leq c \leq a$ and $0\leq d \leq b$ we have 
$$
\mathscr H^{-i_1-c,-j_1-d}D((\sigma_1\sigma_2)^{k_1}, \sigma_1^{m_1}, T)^! \cong \mathscr H^{i'+a-c,j'+b-d}D((\sigma_1\sigma_2)^{-k'}, \sigma_1^{-m'}, T)^!
$$
which by Lemma \ref{Khovanov homology of mirror image} suffices to prove that for all $0\leq c\leq a$ and $0 \leq d \leq b$
$$
\mathscr H^{i+c,j+d}D((\sigma_1\sigma_2)^k, \sigma_1^m, T) \cong \mathscr H^{-i'-a+c,-j'-b+d}D((\sigma_1\sigma_2)^{-k'}, \sigma_1^{-m'}, T).
$$
\end{proof}

    %Algo1 input: $k,m,i,j,a,b,D,T$  where $k,m,i,j,a,b$ are integers, $a,b$ are positive, $D$ is a 3-input planar arc diagram with $B(D)=\emptyset$ and $T$ is a tangle diagram so that $D((\sigma_1\sigma_2), \sigma_1, T)$ makes sense.

   % Algo1 output: integers $k,m,i,j$ with

\begin{algorithm*}
\caption{Parameter reduction to a finite set}\label{isomorphism algorithm}
 \hspace*{\algorithmicindent} \textbf{Input} $(i,j,k,m,a,b,D,T)$\\
 \hspace*{\algorithmicindent} \textbf{Output} $\in \Z^4$ 
   \begin{algorithmic}[1]
%    \Procedure{MyProcedure}{a}
\If{$k< 0$}
\State $D' \gets \operatorname{flip}(D)$
\State $T' \gets \operatorname{flip}(T^!)$
\State $(i',j',k',m') \gets $\textbf{ this algorithm}$(-i-a,-j-b,-k,-m,a+1,b, D',T') $
\State \textbf{return} $(-i'-a,-j'-b,-k',-m')$
\EndIf
\State $\mathcal C \gets \Psi \llbracket T \rrbracket$ \Comment{The delooping isomorphism}
%\State $\triangleright$ Apply either Theorem \ref{Main theorem for 3-input diagrams} or \ref{dual 3 input} recursively
\While{$m\geq  g(a,b,D,\mathcal{C})$}
\If{$\min_{c,d} v_1(i+c,j+d+2,k,m-2,D,\mathcal{C})\geq 0$}
    %\State $\triangleright$ Apply isomorphism \ref{v_1 isomorphism} for all $c,d$
    \State $(i,j,k,m) \gets (i,j+2,k,m-2)$ \Comment{Isomorphism \ref{v_1 isomorphism} for all $c,d$}
\Else%If{$\max_{c,d} v_2(i+c+2,j+d+6,k,m-2,D,\mathcal{C})\leq 0$}
    %\State $\triangleright$ Apply isomorphism \ref{v_2 isomorphism}  for all $c,d$
    \State $(i,j,k,m) \gets (i+2,j+6,k,m-2)$ \Comment{Isomorphism \ref{v_2 isomorphism}  for all $c,d$}
\EndIf
\EndWhile
\While{$m\leq - g(a,b,D,\mathcal{C})$}
\If{$\min_{c,d} w_1(i+c-2,j+d-6,k,m+2,D,\mathcal{C})\geq 0$}
    %\State $\triangleright$  Apply isomorphism \ref{w_1 isomorphism} for all $c,d$
    \State $(i,j,k,m) \gets (i-2,j-6,k,m+2)$ \Comment{Isomorphism \ref{w_1 isomorphism} for all $c,d$}
\Else%If{$\max_{c,d} w_2(i+c,j+d-2,k,m+2,D,\mathcal{C})\leq 0$}
    %\State $\triangleright$ Apply isomorphism \ref{w_2 isomorphism}  for all $c,d$
    \State $(i,j,k,m) \gets (i,j-2,k,m+2)$ \Comment{Isomorphism \ref{w_2 isomorphism}  for all $c,d$}
\EndIf
\EndWhile
%\State \Comment{Move to a 2-input planar arc diagram to apply Theorem \ref{isomorphisms of Khovanov homology groups with 2 input diagrams} }
\State $(D',T') \gets \operatorname{merge}(D,\sigma_1^m, T)$ \label{code line merge 1}\Comment{Move to a 2-input planar arc diagram}
\State $\mathcal{C}'\gets \llbracket T' \rrbracket $ \label{code line merge 2}
\While{$k\geq h(a,\mathcal{C'})$}
\If{$\min_{c,d} (i+c)>o_2 -\lfloor \frac{4(k-3)}{3} \rfloor$}
    %\State $\triangleright$ Apply isomorphism \ref{3rd isomorphism for 2-input diagrams} for all $c,d$
    \State $(i,j,k,m) \gets (i,j+6,k-3,m)$ \Comment{Isomorphism \ref{3rd isomorphism for 2-input diagrams} for all $c,d$}
\Else%If{$\max_{c,d} (i+c+4)<o_1-1$}
    %\State $\triangleright$ Apply isomorphism \ref{4th isomorphism  for 2-input diagrams}  for all $c,d$
    \State $(i,j,k,m) \gets (i+4,j+12,k-3,m)$ \Comment{Isomorphism \ref{4th isomorphism  for 2-input diagrams}  for all $c,d$}
\EndIf
\EndWhile
\State \textbf{return} $(i,j,k,m)$
%\EndProcedure
\end{algorithmic}
\thisfloatpagestyle{plain}
\end{algorithm*}

We will now show some consequences of combining Theorem \ref{algorithm theorem} with computer data.

 \begin{theorem} \label{Omega4 and Omega5 contain only 2 torsion}
        The Khovanov homology groups of closures of braids from sets $\Omega_4$ and $\Omega_5$ contain only $\Z / 2\Z$ torsion.        
\end{theorem}

\begin{proof}
    We can draw braid closures of $\Omega_4$ with a  3-input planar arc diagram $D$, $3k\in \Z$ and $m<0$, see Figure \ref{Link diagrams for Omega4 and Omega5}. 
    %Making use of the Remark \ref{simplified parameters for I_1 and I_2} and 
    Running the algorithm with $(i,j,3k,m,0,0,D,\emptyset)$ reduces the homology groups of $\Omega_4$ to those with $k=-9,\dots,9$ and $m=-28,\dots, -1$. The Khovanov homology of these 532 links is calculated in \textbf{Computer Data II} \cite{computerDataFor3BraidPaper} and from this data one can find only $\Z / 2\Z$ torsion. 
    Similarly we can draw links of  $\Omega_5$ with a 3-input planar arc diagram $D'$, $3k\in \Z$ and $m>0$. This time we reduce to a different set of 532 links with  $k=-9,\dots,9$ and $m=1,\dots, 28$ which is calculated in \textbf{Computer Data III} \cite{computerDataFor3BraidPaper}, and again only $\Z / 2\Z$ torsion can be found.     
\end{proof}

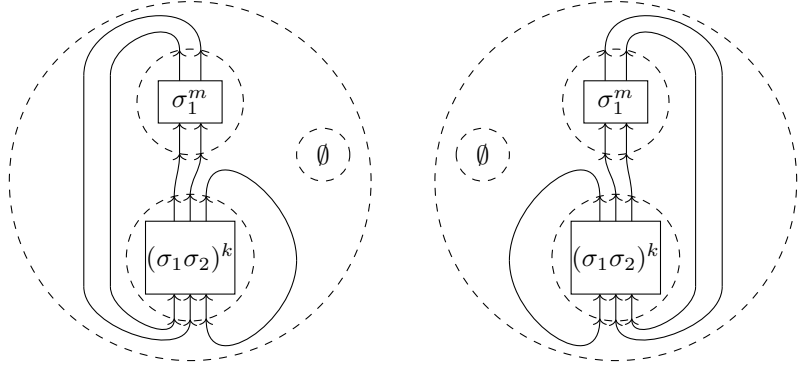
\begin{figure}[ht]
    \centering
    \begin{tikzpicture}[
  arrowmark/.style 2 args={
    postaction=decorate,
    decoration={
      markings,
      mark=at position #1 with {\arrow{#2}}
    }
  },scale=0.7,xscale=1
]

  \begin{scope}[shift={(0,0)}]
      
    \draw[dashed] (1,0) circle (3.4);

      \begin{scope}[shift={(0.8,1.5)},scale=0.4]
        
  \draw[dashed] (0.5,0) circle (2.5);

  \draw (-1, -1) rectangle (2, 1);

    \node at (0.5,0) {$\sigma_1^m$};

  % Arrows from the origin to nodes
  \draw[->] (1,1) -- (1,2.46);
  \draw[->] (0,1) -- (0,2.46);
  \draw[<-] (1,-1) -- (1,-2.46);
  \draw[<-] (0,-1) -- (0,-2.46);
  %\draw[->] (0,0) -- (-1,3.87);
  %\draw[->] (0,0) -- (1,-3.87);
  %\draw[->] (0,0) -- (-1,-3.87);

  \end{scope}

  \begin{scope}[shift={(-2,-2.5)},scale=0.3]
      \draw[dashed] (10, 3.5) circle (4);

    \draw (12.8, 5.8) rectangle (7.2, 1.2);

    \node at (10,3.5) {$(\sigma_1\sigma_2)^k$};

    \draw[<-] (9,1.2) -- (9,-0.38);
    \draw[<-] (11,1.2) -- (11,-0.38);
    
    \draw[->] (9,5.8) -- (9,7.38);
    \draw[->] (11,5.8) -- (11,7.38);

    \draw[<-] (10,1.2) -- (10,-0.5);
    \draw[->] (10,5.8) -- (10,7.5);
  \end{scope}

    \draw[<-, looseness=1, out=270, in=270] (0.7,-2.6) to (-0.5,-2);
    %maali
    %\draw[-, looseness=1, out=90, in=0] (-0.5,-2) to (-0.88,-0.405);
    
    \draw[->, looseness=1, out=90, in=270] (0.7,-0.3) to (0.8,0.52);

    \draw[->, looseness=1, out=90, in=270] (1,-0.28) to (1.2,0.52);

    %\draw[-, looseness=1, out=270, in=180] (-1.6,-1.12) to (-1,-3);
    \draw[->, looseness=1, out=270, in=270] (-1,-2) to (1,-2.62);
    \draw (-1,-2) -- (-1,2);

    \draw[-, looseness=1.5, out=90, in=90] (1.3,-0.3) to (3,-1.5);
    \draw[->, looseness=1.5, out=270, in=270] (3,-1.5) to (1.3,-2.6);

    \draw[looseness=1.3, out=90, in=90] (0.8,2.48) to (-0.5,2);
    \draw (-0.5,2) -- (-0.5,-2);
    
    \draw[looseness=1.3, out=90, in=90] (1.2,2.48) to (-1,2);

    \draw[dashed] (3.5,0.5) circle (0.5);
    \node at (3.5,0.5) {$\emptyset$};

  \end{scope}
  
  %SECOND CIRCLE
  \begin{scope}[shift={(10,0)}, xscale=-1]
      
    \draw[dashed] (1,0) circle (3.4);

      \begin{scope}[shift={(0.8,1.5)},scale=0.4]
        
  \draw[dashed] (0.5,0) circle (2.5);

  \draw (-1, -1) rectangle (2, 1);

    \node at (0.5,0) {$\sigma_1^m$};

  % Arrows from the origin to nodes
  \draw[->] (1,1) -- (1,2.46);
  \draw[->] (0,1) -- (0,2.46);
  \draw[<-] (1,-1) -- (1,-2.46);
  \draw[<-] (0,-1) -- (0,-2.46);
  %\draw[->] (0,0) -- (-1,3.87);
  %\draw[->] (0,0) -- (1,-3.87);
  %\draw[->] (0,0) -- (-1,-3.87);

  \end{scope}

  \begin{scope}[shift={(-2,-2.5)},scale=0.3]
      \draw[dashed] (10, 3.5) circle (4);

    \draw (12.8, 5.8) rectangle (7.2, 1.2);

    \node at (10,3.5) {$(\sigma_1\sigma_2)^k$};

    \draw[<-] (9,1.2) -- (9,-0.38);
    \draw[<-] (11,1.2) -- (11,-0.38);
    
    \draw[->] (9,5.8) -- (9,7.38);
    \draw[->] (11,5.8) -- (11,7.38);

    \draw[<-] (10,1.2) -- (10,-0.5);
    \draw[->] (10,5.8) -- (10,7.5);
  \end{scope}

    \draw[<-, looseness=1, out=270, in=270] (0.7,-2.6) to (-0.5,-2);
    %maali
    %\draw[-, looseness=1, out=90, in=0] (-0.5,-2) to (-0.88,-0.405);
    
    \draw[->, looseness=1, out=90, in=270] (0.7,-0.3) to (0.8,0.52);

    \draw[->, looseness=1, out=90, in=270] (1,-0.28) to (1.2,0.52);

    %\draw[-, looseness=1, out=270, in=180] (-1.6,-1.12) to (-1,-3);
    \draw[->, looseness=1, out=270, in=270] (-1,-2) to (1,-2.62);
    \draw (-1,-2) -- (-1,2);

    \draw[-, looseness=1.5, out=90, in=90] (1.3,-0.3) to (3,-1.5);
    \draw[->, looseness=1.5, out=270, in=270] (3,-1.5) to (1.3,-2.6);

    \draw[looseness=1.3, out=90, in=90] (0.8,2.48) to (-0.5,2);
    \draw (-0.5,2) -- (-0.5,-2);
    
    \draw[looseness=1.3, out=90, in=90] (1.2,2.48) to (-1,2);

    \draw[dashed] (3.5,0.5) circle (0.5);
    \node at (3.5,0.5) {$\emptyset$};

  \end{scope}

\end{tikzpicture}

    \caption{Link diagrams $D((\sigma_1\sigma_2)^k, \sigma_1^m,\emptyset)$ (left) and $D'((\sigma_1\sigma_2)^k, \sigma_1^m,\emptyset)$ (right). Coincidentally, $\operatorname{flip} D=D'$. }
    \label{Link diagrams for Omega4 and Omega5}
\end{figure}

 %and $D'((\sigma_1\sigma_2)^{3k},\sigma_1^{-q},\emptyset )$.

Finally as an example, we give an explicit description of Khovanov homologies of a subfamily of braid closures contained in $\Omega_4$. 

\begin{proposition}\label{Proposition for example of a family of links}
    The integral Khovanov homology of closures of braids $(\sigma_1\sigma_2)^{3k}\sigma_1^{-2m}\in\Omega_4$, for $k\geq 4$ and $m\geq 8$ is given in Figure \ref{exampleofKhovanovhomologiesfigure}. 
\end{proposition}
The result can be proven with the combination of \textbf{Computer Data II} and Theorem \ref{algorithm theorem} (or alternatively by using Isomorphisms \ref{3rd isomorphism for 2-input diagrams}, \ref{4th isomorphism  for 2-input diagrams}, \ref{w_1 isomorphism} and \ref{w_2 isomorphism} directly).  At the induction steps, the new tables are obtained by cutting and pasting the current ones.
By plugging in large $2k=m$ into this example, we can see that there exist closures of 3-braids with nontrivial Khovanov homology groups in homological degree 0, whose internal gradings are arbitrarily far from each other. Moreover, these nontrivial homology groups are separated by a gap of only trivial homology groups in between them.

%ababababababababababababAAAAAAAAAAAAAAAA
\newgeometry{left=2.5cm,right=2.5cm,top=3cm,bottom=3cm}
 \begin{figure}
\begin{subfigure}{\textwidth}
     \centering

     \scalebox{0.935}{
     \footnotesize\begin{tabular}{c||c|c|c|c|c|c|c|c|c|c|c|c|c|c|c|}

 & $-4k$ &  &  &  &  &  &  &  &  &  &  &  &  &  & $-4k+14$ \\
 \hline
 \hline
$-12k+2m+27$ &  &  &  &  &  &  &  &  &  &  &  &  &  & $\mathbb{{Z}}^{  }_{  }$ &$\mathbb{{Z}}^{  }_{  }$  \\
\hline
 &  &  &  &  &  &  &  &  &  &  &  &  & & $\mathbb{{Z}}^{  }_{  } \oplus \mathbb{{Z}}^{  }_{ 2 }$ & $\mathbb{{Z}}^{  }_{  }$\\
\hline
 &  &  &  &  &  &  &  &  &  &  &  & $\mathbb{{Z}}^{  }_{  }$ & $\mathbb{{Z}}^{  }_{  }$ & $\mathbb{{Z}}^{  }_{ 2 }$ & $\mathbb{{Z}}^{  }_{  }$  \\
\hline
 &  &  &  &  &  &  &  &  &  &  &  & $\mathbb{{Z}}^{  }_{  } \oplus \mathbb{{Z}}^{  }_{ 2 }$ & $\mathbb{{Z}}^{ 2 }_{  }$ &  & $\mathbb{{Z}}^{  }_{ 2 }$  \\
\hline
 &  &  &  &  &  &  &  &  &  & $\mathbb{{Z}}^{  }_{  }$ & $\mathbb{{Z}}^{  }_{  }$ & $\mathbb{{Z}}^{  }_{ 2 }$ & $\mathbb{{Z}}^{  }_{  }$ & $\mathbb{{Z}}^{  }_{  }$ &   \\
\hline
 &  &  &  &  &  &  &  &  &  & $\mathbb{{Z}}^{  }_{  } \oplus \mathbb{{Z}}^{  }_{ 2 }$ & $\mathbb{{Z}}^{ 2 }_{  }$ & $\mathbb{{Z}}^{  }_{  }$ &  &  &   \\
\hline
 &  &  &  &  &  &  &  & $\mathbb{{Z}}^{  }_{  }$ & $\mathbb{{Z}}^{ 2 }_{  }$ & $\mathbb{{Z}}^{  }_{ 2 }$ & $\mathbb{{Z}}^{  }_{ 2 }$ & $\mathbb{{Z}}^{  }_{  }$ &  &  &   \\
\hline
 &  &  &  &  &  &  &  & $\mathbb{{Z}}^{  }_{  } \oplus \mathbb{{Z}}^{  }_{ 2 }$ & $\mathbb{{Z}}^{ 2 }_{  }$ & $\mathbb{{Z}}^{  }_{  }$ &  &  &  &  &  \\
\hline
 &  &  &  &  &  & $\mathbb{{Z}}^{  }_{  }$ & $\mathbb{{Z}}^{ 2 }_{  }$ & $\mathbb{{Z}}^{  }_{  } \oplus \mathbb{{Z}}^{  }_{ 2 }$ &  &  &  &  &  &   & \\
\hline
 &  &  &  &  &  & $\mathbb{{Z}}^{  }_{ 2 }$ & $\mathbb{{Z}}^{  }_{  } \oplus \mathbb{{Z}}^{  }_{ 2 }$ & $\mathbb{{Z}}^{  }_{  }$ &  &  &  &  &  &   & \\
\hline
 &  &  &  & $\mathbb{{Z}}^{  }_{  }$ & $\mathbb{{Z}}^{ 2 }_{  }$ & $\mathbb{{Z}}^{  }_{  }$ &  &  &  &  &  &  &  &  &   \\
\hline
 &  &  &  & $\mathbb{{Z}}^{  }_{ 2 }$ & $\mathbb{{Z}}^{  }_{ 2 }$ &  &  &  &  &  &  &  &  &  &   \\
\hline
 &  &  & $\mathbb{{Z}}^{  }_{  }$ & $\mathbb{{Z}}^{  }_{  }$ &  &  &  &  &  &  &  &  &  &   & \\
\hline
 & $\mathbb{{Z}}^{  }_{  }$ &  &  &  &  &  &  &  &  &  &  &  &  &  &  \\
\hline
$-12k+2m-1$ & $\mathbb{{Z}}^{  }_{  }$ &  &  &  &  &  &  &  &  &  &  &  &  &  & \\
\hline
\end{tabular}}
\caption{Khovanov homology for closures of $(\sigma_1\sigma_2)^{3k}\sigma_1^{-2m}$ in indices $i\leq -4k+14$ and $j\in \Z$.}

\end{subfigure}

\begin{subfigure}{\textwidth}
     \centering

\begin{tikzpicture}
\node at (0,0) {
    
\small\begin{tabular}{c||c|c|c|c|c|c|c|c|c|c|}

 &      $-4k+15$            &              &              &         $-4k+18$                       &  $\cdots$              & $\cdots$  &    $-4k+2m-3$                            &              &                                &   $-4k+2m$             \\ \hline
 \hline
 $-12k+6m+1$&       &                         &              &                                &                &  &                                &              &                                & $\mathbb{Z}$   \\ \hline
 &                  &              &              &                                &                &  &                                &              & $\mathbb{Z}$                   & $\mathbb{Z}^2$ \\ \hline
 &                   &             &              &                                &                &  &                                &              & $\mathbb{Z}\oplus\mathbb{Z}_2$ & $\mathbb{Z}$   \\ \hline
 &                    &            &              &                                &                &  & \cellcolor{green!25}$\mathbb{Z}$                   & \cellcolor{green!25}$\mathbb{Z}$ & $\mathbb{Z}_2$                 &        \cellcolor{gray!75}        \\ \hline
 &                     &           &              &                                &                &  & \cellcolor{green!25}$\mathbb{Z}\oplus\mathbb{Z}_2$ & \cellcolor{green!25}$\mathbb{Z}$ &   \cellcolor{gray!75}                             &   \cellcolor{gray!75}             \\ \hline
$-12k+6m-9$ &                      &          &              &                                &                &  $\adots$ & \cellcolor{green!25}$\mathbb{Z}$                   &  \cellcolor{gray!75}            &     \cellcolor{gray!75}                           &     \cellcolor{gray!75}           \\ \hline
 $\vdots$ &                &              &              &                                &          $\adots$      &   &    \cellcolor{gray!75}                             &   \cellcolor{gray!75}            &  \cellcolor{gray!75}                               &       \cellcolor{gray!75}          \\ \hline

$-12k+2m+35$ &                                &              & \cellcolor{green!25}$\mathbb{Z}$                   & \cellcolor{green!25}$\mathbb{Z}$   & &     \cellcolor{gray!75}                           & \cellcolor{gray!75}             &     \cellcolor{gray!75}                           & \cellcolor{gray!75}        &      \cellcolor{gray!75}  \\ \hline
 &                                &              & \cellcolor{green!25}$\mathbb{Z}\oplus\mathbb{Z}_2$ & \cellcolor{green!25}$\mathbb{Z}$   & \cellcolor{gray!75} & \cellcolor{gray!75}                               &  \cellcolor{gray!75}            &  \cellcolor{gray!75}                              &     \cellcolor{gray!75}       &   \cellcolor{gray!75}  \\ \hline
 &  \cellcolor{green!25}$\mathbb{Z}$                   & \cellcolor{green!25}$\mathbb{Z}$ & \cellcolor{green!25}$\mathbb{Z}_2$                 & \cellcolor{gray!75}               & \cellcolor{gray!75} &\cellcolor{gray!75}                                &   \cellcolor{gray!75}           &    \cellcolor{gray!75}                            &    \cellcolor{gray!75}          &  \cellcolor{gray!75} \\ \hline
 & \cellcolor{green!25}$\mathbb{Z}\oplus\mathbb{Z}_2$ & \cellcolor{green!25}$\mathbb{Z}$ &      \cellcolor{gray!75}                          &   \cellcolor{gray!75}             & \cellcolor{gray!75}  &    \cellcolor{gray!75}                            &  \cellcolor{gray!75}             & \cellcolor{gray!75}                               &   \cellcolor{gray!75}           &  \cellcolor{gray!75} \\ \hline
$-12k+2m+27$ & \cellcolor{green!25}$\mathbb{Z}_2$                 & \cellcolor{gray!75}             &   \cellcolor{gray!75}                             &   \cellcolor{gray!75}             & \cellcolor{gray!75} & \cellcolor{gray!75}                               &  \cellcolor{gray!75}            &   \cellcolor{gray!75}                             &     \cellcolor{gray!75}    &     \cellcolor{gray!75}   \\ \hline
\end{tabular}
};

%\draw (0,0) circle (1);

\draw[line width=1pt, rounded corners, fill=white] (3.2,-2.8) rectangle (7.5,-0.75);

\node at (5.3,-1) {$m-8$ copies of };
\node at (5.3,-2) {

\begin{tabular}{cc}
\cellcolor{green!25}$\mathbb{Z}$ & \cellcolor{green!25}$\mathbb{Z}$ \\
\cellcolor{green!25}$\mathbb{Z}\oplus \mathbb{Z}_2$ & \cellcolor{green!25}$\mathbb{Z}$ \\
\cellcolor{green!25}$\mathbb{Z}_2$ & \\
\end{tabular}

};

\end{tikzpicture}

\caption{Khovanov homology for closures of $(\sigma_1\sigma_2)^{3k}\sigma_1^{-2m}$ in indices $i\geq -4k+15$ and $j\geq-4k+2m+2i-3$.}
\end{subfigure}

\begin{subfigure}{\textwidth}
     \centering

\begin{tikzpicture}
    
\node at (0,0) {
\small\begin{tabular}{c||c|c|c|c|c|c|c|c|c|c|c|c|c|c|c|}

             &$-4k+15$                                                 &                                                 &                                                 &                                                   &                                                 &                                                 &                                                 &  $-4k+22$                                                 &     $\cdots$     &   $-5$                                              &                                                 &                                                 &                                                   &  & 0            \\ \hline \hline
            $-6k+2m+3$ & \cellcolor{gray!75}           & \cellcolor{gray!75}           & \cellcolor{gray!75}           & \cellcolor{gray!75}             & \cellcolor{gray!75}           & \cellcolor{gray!75}           & \cellcolor{gray!75}           & \cellcolor{gray!75}             &       \cellcolor{gray!75}   &         \cellcolor{gray!75}                                        &                                                &                                                 &                                                   &  & $\mathbb{Z}$ \\ \hline
             & \cellcolor{gray!75}           & \cellcolor{gray!75}           & \cellcolor{gray!75}           & \cellcolor{gray!75}             & \cellcolor{gray!75}           & \cellcolor{gray!75}           & \cellcolor{gray!75}           & \cellcolor{gray!75}             &        \cellcolor{gray!75}  &                                              &                                                 &                                                 &                                                   &  & $\mathbb{Z}$ \\ \hline
             & \cellcolor{gray!75}           & \cellcolor{gray!75}           & \cellcolor{gray!75}           & \cellcolor{gray!75}             & \cellcolor{gray!75}           & \cellcolor{gray!75}           & \cellcolor{gray!75}           & \cellcolor{gray!75}             &        &                                                 &                                                 &                                                 & \cellcolor{red!25}$\mathbb{Z}$   &  &              \\ \hline
             & \cellcolor{gray!75}           & \cellcolor{gray!75}           & \cellcolor{gray!75}           & \cellcolor{gray!75}             & \cellcolor{gray!75}           & \cellcolor{gray!75}           & \cellcolor{gray!75}           & \cellcolor{gray!75}             &          &                                                 & \cellcolor{red!25}$\mathbb{Z}$ &                                                 & \cellcolor{red!25}$\mathbb{Z}_2$ &  &              \\ \hline
             & \cellcolor{gray!75}           & \cellcolor{gray!75}           & \cellcolor{gray!75}           & \cellcolor{gray!75}             & \cellcolor{gray!75}           & \cellcolor{gray!75}           & \cellcolor{gray!75}           & \cellcolor{gray!75}             &          &                                                 & \cellcolor{red!25}$\mathbb{Z}$ & \cellcolor{red!25}$\mathbb{Z}$ &                                                   &  &              \\ \hline
             & \cellcolor{gray!75}           & \cellcolor{gray!75}           & \cellcolor{gray!75}           & \cellcolor{gray!75}             & \cellcolor{gray!75}           & \cellcolor{gray!75}           & \cellcolor{gray!75}           & \cellcolor{gray!75}             &          & \cellcolor{red!25}$\mathbb{Z}$ &                                                 &                                                 &                                                   &  &              \\ \hline
             $-6k+2m-9$& \cellcolor{gray!75}           & \cellcolor{gray!75}           & \cellcolor{gray!75}           & \cellcolor{gray!75}             & \cellcolor{gray!75}           & \cellcolor{gray!75}           & \cellcolor{gray!75}           &                                                   &          & \cellcolor{red!25}$\mathbb{Z}$ &                                                 &                                                 &                                                   &  &              \\ \hline
             $\vdots$& \cellcolor{gray!75}           & \cellcolor{gray!75}           & \cellcolor{gray!75}           & \cellcolor{gray!75}             & \cellcolor{gray!75}           & \cellcolor{gray!75}           &                                                 &                                                   & $\adots$ &                                                 &                                                 &                                                 &                                                   &  &              \\ \hline
            $-12k+2m+35$ & \cellcolor{gray!75}           & \cellcolor{gray!75}           & \cellcolor{gray!75}           & \cellcolor{gray!75}             & \cellcolor{gray!75}           &                                                 &                                                 & \cellcolor{red!25}$\mathbb{Z}$   &          &                                                 &                                                 &                                                 &                                                   &  &              \\ \hline
             & \cellcolor{gray!75}           & \cellcolor{gray!75}           & \cellcolor{gray!75}           & \cellcolor{gray!75}             &                                                 & \cellcolor{red!25}$\mathbb{Z}$ &                                                 & \cellcolor{red!25}$\mathbb{Z}_2$ &          &                                                 &                                                 &                                                 &                                                   &  &              \\ \hline
             & \cellcolor{gray!75}           & \cellcolor{gray!75}           & \cellcolor{gray!75}           &                                                   &                                                 & \cellcolor{red!25}$\mathbb{Z}$ & \cellcolor{red!25}$\mathbb{Z}$ &                                                   &          &                                                 &                                                 &                                                 &                                                   &  &              \\ \hline
             & \cellcolor{gray!75}           & \cellcolor{gray!75}           &                                                 & \cellcolor{red!25}$\mathbb{Z}$   & \cellcolor{red!25}$\mathbb{Z}$ &                                                 &                                                 &                                                   &          &                                                 &                                                 &                                                 &                                                   &  &              \\ \hline
             & \cellcolor{gray!75}           & \cellcolor{red!25}$\mathbb{Z}$ &                                                 & \cellcolor{red!25}$\mathbb{Z}_2$ & \cellcolor{red!25}$\mathbb{Z}$ &                                                 &                                                 &                                                   &          &                                                 &                                                 &                                                 &                                                   &  &              \\ \hline
             &                                                 & \cellcolor{red!25}$\mathbb{Z}$ & \cellcolor{red!25}$\mathbb{Z}$ &                                                   &                                                 &                                                 &                                                 &                                                   &          &                                                 &                                                 &                                                 &                                                   &  &              \\ \hline
             & \cellcolor{red!25}$\mathbb{Z}$ &                                                 &                                                 &                                                   &                                                 &                                                 &                                                 &                                                   &          &                                                 &                                                 &                                                 &                                                   &  &              \\ \hline
$-12k+2m+21$ & \cellcolor{red!25}$\mathbb{Z}$ &                                                 &                                                 &                                                   &                                                 &                                                 &                                                 &                                                   &          &                                                 &                                                 &                                                 &                                                   &  &              \\ \hline
\end{tabular}
};
\draw[line width=1pt, rounded corners, fill=white] (3.2,-3.5) rectangle (7.5,-0.7);

\node at (5.3,-0.95) {$k-4$ copies of };
\node at (5.3,-2.3) {

\begin{tabular}{cccc}
             &              &              & \cellcolor{red!25}$\mathbb{Z}$   \\
             & \cellcolor{red!25}$\mathbb{Z}$ &              & \cellcolor{red!25}$\mathbb{Z}_2$ \\
             & \cellcolor{red!25}$\mathbb{Z}$ & \cellcolor{red!25}$\mathbb{Z}$ &                \\
\cellcolor{red!25}$\mathbb{Z}$ &              &              &                \\
\cellcolor{red!25}$\mathbb{Z}$ &              &              &               
\end{tabular}

};

\end{tikzpicture}
\thisfloatpagestyle{plain}
\caption{Khovanov homology for closures of $(\sigma_1\sigma_2)^{3k}\sigma_1^{-2m}$ in indices $i\geq -4k+15$ and $j\leq-4k+2m+2i-5$.}
\end{subfigure}

     \caption{The integral Khovanov homology of closures of braids $(\sigma_1\sigma_2)^{3k}\sigma_1^{-2m}\in\Omega_4$, for $k\geq 4$ and $m\geq 8$. The torsion group $\Z/2\Z$ is denoted with $\Z_2$. Gluing together the three subtables (a), (b), (c) yields a complete description of the whole homology table. }
     \label{exampleofKhovanovhomologiesfigure}
 \end{figure}
\restoregeometry

\pagebreak
\printbibliography

\end{document}